 \documentclass[11pt]{amsart}
\usepackage{natbib}
\bibpunct{[}{]}{;}{n}{,}{,}
\usepackage{amsmath,amssymb,amsbsy,amsfonts}
\usepackage{graphicx}        
\usepackage{multicol}        
\usepackage{epsfig}
\usepackage{color}
\usepackage{wrapfig}
\usepackage{apptools}
\usepackage{txfonts}
\usepackage{pxfonts}
\newtheorem{thm}{Theorem}[section]
\newtheorem{prop}[thm]{Proposition}
\newtheorem{cor}[thm]{Corollary}
\newtheorem{lem}[thm]{Lemma}
\theoremstyle{remark}
\newtheorem{rem}[thm]{Remark}

\newcommand{\n}{\noindent}

\newcommand{\G}{\mathcal{G}}
\newcommand{\Q}{\mathcal{Q}}

\def\smallcircledV{{\scriptscriptstyle{\bigcirc\kern-5.66pt\vee\kern 1.33pt}}}

\newcommand{\real}{\mathbb{R}}

\newcommand{\R}{{\mathbb R}}

\numberwithin{equation}{section}

\begin{document}

\title[Conservative Spectral Boltzmann Scheme Analysis]{Convergence and error estimates for the Lagrangian based conservative spectral method for  Boltzmann equations}
\date{}

\author[R. J. Alonso]{Ricardo J. Alonso}
\address{R.J. Alonso,  Department of Mathematics, P.U.C. Rio de Janeiro, Brazil}
              \email{ricardoalonsoplata@gmail.com}       
\author[I.  M. Gamba]{Irene M. Gamba}
\address{I. M. Gamba, Department of Mathematics, University of Texas at Austin.}
\email{gamba@math.utexas.edu}
\author[S. H. Tharkabhushanam]{Sri Harsha Tharkabhushanam}
\address{S. H. Tharkabhushanam, BJSS - London}
           \email{Harsha.Sri.t@gmail.com}

\begin{abstract}
We develop error estimates for the semi-discrete conservative spectral method for the approximation of the elastic and inelastic space homogeneous Boltzmann equation introduced by the authors in \cite{GT09}. In addition we study the long time convergence of such semi-discrete solution to equilibrium Maxwellian distribution that conserves the mass, momentum and energy associated to the initial data.   The numerical method is based on the Fourier transform of the collisional operator and a Lagrangian optimization correction that enforces the collision invariants, namely conservation of mass, momentum and energy in the elastic case, and just mass and momentum in the inelastic one.  We present a detailed semi-discrete analysis on convergence of the proposed numerical method which includes the $L^{1}-L^{2}$ theory for the scheme.  This analysis allows us to present, additionally, convergence in Sobolev spaces and convergence to equilibrium for the numerical approximation.  The results of this work answer a long standing open problem posed by Cercignani et al. in \cite[Chapter 12]{CIP} about finding error estimates for a numerical scheme associated to the  Boltzmann equation,  as well as showing  the semi-discrete numerical solution converges to the equilibrium Maxwellian distribution associated to the initial value problem.
\end{abstract}
\maketitle

{\small
{\bf Key words.} {Nonlinear integral equations,  Rarefied gas flows, Boltzmann equations, conservative spectral methods.}\\

{\bf AMS subject classifications.}  {45E99, 35A22, 65C20}

\section{Introduction}
The Boltzmann Transport Equation is an integro-differential transport equation that describes the evolution of a single point probability density function $f(t, v, x)$ defined as the probability of finding a particle at position $x$ with kinetic velocity $v$ at time $t$. The mathematical and computational difficulties associated to the Boltzmann equation are due to the non local and non linear nature of the binary collision operator, which is usually modeled as a bilinear integral form in $d$-dimensional velocity space and unit sphere $\mathbb{S}^{d-1}$.

The focus of this manuscript is to provide a complete consistency and error analysis and long time convergence to statistical equilibrium states for the Lagrangian based conservative spectral scheme proposed in \cite{GT09} to solve the dynamics of elastic binary collisions.  In particular, the results of this work answer a long standing open problem posed by Cercignani, Illner and Pulvirenti  in \cite[Chapter 12]{CIP} about finding error estimates for a consistent non linear Boltzmann deterministic scheme for elastic binary interactions in the case of hard potentials.

The problem of computing efficiently the Boltzmann transport equation has interested many authors that have introduced different approaches.  These approaches can be classified as stochastic methods known as Direct Simulation Montecarlo Methods (DSMC \cite{bird, RjaWa05, Wagner92,GRW04}) and deterministic methods (Discrete Velocity Models \cite{kawa81, broadwell, bobyCerci99,mieuss00,HerParSea}, Boltzmann approximations - Lattice Boltzmann, BGK and Spectral methods \cite{Gab-Par-Tos, bobylevRjasanow0, pareschiRusso, bobylevFT, BCG00, BR00, pareschiPerthame, ibragRjasanow, filbetRusso1, filbetRusso, mouhotPareschi}).  Spectral based methods, our choice for this work, have been developed by I.M. Gamba and H.S Tharkabhushanam \cite{GT09} inspired in the work developed a decade earlier by  Pareschi, Gabetta and Toscani \cite{Gab-Par-Tos} and later by Bobylev and Rjasanow~\cite{bobylevRjasanow0} and Pareschi and Russo~\cite{pareschiRusso}.  The practical implementation of these methods are supported by the ground breaking work of Bobylev~\cite{bobylevFT} using the Fourier transformed Boltzmann equation to analyze its solutions in the case of Maxwell type of interactions. After the introduction of the inelastic Boltzmann equation for Maxwell type interactions and the use of the Fourier transform for its analysis in Bobylev, Carrillo and Gamba \cite{BCG00}, the spectral based approach is becoming the most suitable tool to deal with deterministic computations of kinetic models associated with the full Boltzmann collisional integral, both for elastic or inelastic interactions.  Recent implementations of spectral methods for the non-linear Boltzmann are due to Bobylev and Rjasanow~\cite{bobylevRjasanow0} who developed a method using the Fast Fourier Transform (FFT) for Maxwell type of interactions and then for Hard-Sphere interactions~\cite{BR99} using generalized Radon and X-ray transforms via FFT.  Simultaneously, L. Pareschi and B. Perthame \cite{pareschiPerthame} developed  similar scheme using FFT for Maxwell type of interactions.  Using \cite{pareschiRusso, pareschiPerthame}, Filbet and Russo in \cite{filbetRusso1} and  \cite{filbetRusso} have implemented an scheme to solve the space inhomogeneous Boltzmann equation.  We also mention the work of I. Ibragimov and S. Rjasanow~\cite{ibragRjasanow} who developed a numerical method to solve the space homogeneous Boltzmann Equation on an uniform grid for a Variable Hard Potential interactions with elastic collisions. This particular work has been a great inspiration for the current paper and was one of the first steps in the direction of a new numerical method.  

The aforementioned works on deterministic solvers for non-linear BTE have been restricted to elastic, conservative interactions.  Mouhot and Pareschi \cite{mouhotPareschi} have studied some approximation properties of the schemes. Part of the difficulties in their strategy arises from the constraint that the numerical solution has to satisfy conservation of the initial mass. To this end, the authors propose the use of a periodic representation of the distribution function to avoid aliasing.  Closely related to this problem is the fact that spectral methods do not guarantee the positivity of the solution due to the combined effects of the truncation in velocity domain (of the equation) and the application of the Fourier transform (computed for the truncated problem).  In addition to this, there is no \textit{a priori} conservation of mass, momentum and energy in \cite{filbetRusso}, \cite{filbetRusso1} and \cite{mouhotPareschi}.  In fact, the authors in \cite{FM} presented a stability and convergence analysis of the spectral method for the homogeneous Boltzmann equation for binary elastic collisions using the periodization approach proposed in those previous references. In their results, the spectral scheme enforced only mass conservation; as a consequence, the numerical solutions converge to the constant state, hence, destroying the time asymptotic behavior predicted by the Boltzmann $\mathcal{H}$-Theorem.

It is shown in this manuscript that the conservative approach scheme proposed in \cite{GT09} is able to handle the conservation problem in a natural way, by means of Lagrange multipliers, and enjoys convergence and correct long time asymptotic to the Maxwelliam equilibrium.  Our approximation by conservative spectral Lagrangian schemes and corresponding computational method is based on an  alternative approach to the work in \cite{bobylevRjasanow0} and \cite{ibragRjasanow}. This spectral approach combined with a constrain minimization problem works for elastic or inelastic collisions and energy dissipative non-linear Boltzmann type models for variable hard potentials.  We do not use periodic representations for the distribution function and the only restriction of the current method is that it requires that the distribution function to be Fourier transformable at any time step.  This requirement is met by imposing $L^2$-integrability to the initial datum. The required conservation properties of the distribution function are enforced through an optimization problem with the desired conservation quantities set as the constraints.  The correction to the distribution function that makes the approximation conservative is very small but crucial for the evolution of the probability distribution function according to the Boltzmann equation.

More recently, this conservative spectral method for the Boltzmann equation was applied to the calculation of the Boltzmann flow for anisotropic collisions, even in the Coulomb interaction regime \cite{GaHa}, where the solution of the Boltzmann equation approximates solution for Landau equation \cite{Landau37, LanLifStat}.  It has also been extended to systems of elastic and inelastic hard potential problems modeling of a multi-energy level gas \cite{MHGM14}.  In this case, the formulation of the numerical method accounts for both elastic and inelastic collisions. It was also used for the particular case of a chemical mixture of monatomic gases without internal energy. The conservation of mass, momentum and energy during collisions is enforced through the solution of constrained optimization problem to keep the collision invariances associated to the mixtures. The implementation was done in the space inhomogeneous setting (see \cite{MHGM14}, section~4.3), where the advection along the free Hamiltonian dynamics is modeled by time splitting methods following the initial approach in \cite{GT08}.  The effectiveness of the scheme applied to these mixtures has been compared with the results obtained by means of the DSMC method and excellent agreement has been observed.

In addition, this conservative spectral Lagrangian method has been implemented in a system of electron-ion in plasma 
modeled by a $2\times2$ system of Poisson-Vlasov-Landau equations \cite{zhang-gamba-Landau16} using time splitting methods, that is, staggering the time steps for advection of the Vlasov-Poisson system and the collisional system including recombinations. The constrained optimization problem is applied to the collisional step in a revised version from \cite{GT09} where such minimization problem was posed and solved in Fourier space, using  the exact formulas for the Fourier Transform of the collision invariant polynomials.  The benchmarking for the constrained optimization implementation for the mixing problem was done for an example of a space homogeneous system where the explicit decay difference for electron and ion temperatures is known \cite{zhang-gamba-Landau16}, section~7.1.2. Yet, the used scheme captures the total conserved temperature being a convex sum of the ions and electron temperatures respectively.

The key note results of the manuscript are stated in Theorem \ref{CT} in section 3.  The proof of this theorem relies on the Lagrangian correction problem that enforces conservation at the numerical level. This is a key idea that shows that the conservative spectral scheme converges to the Gaussian (Maxwellian) distribution in velocity space.  Indeed, the enforcement of the collision invariants is sufficient to show  the convergence result to the Maxwellian equilibrium in the case of an scalar space homogeneous Boltzmann equation for binary elastic interactions.  This is exactly how the Boltzmann $\mathcal{H}$-Theorem works \cite{CIP}; the equilibrium Maxwellian \eqref{eqMax} is proven to be the stationary state due to the conservation properties combined with the elastic collision law. 
    
{In the case of inelastic collisions for either Maxwell type o hard sphere interactions for constant rate of local energy law  \cite{BCG00,GPV04,MMR06}  or  
visoelastic particle type of interactions \cite{ALod,AlonsoLods2014}, where local energy rates depend on the local impact angle, making them an elastic interaction as the interaction is glanzing, the number of  collision invariants to be enforced is just $d+1$ polynomials. In addition, trivial stationary states are either a singular distribution or vacuum, and it has been shown that there also are  non-trivial attracting self-similar solutions that  develop power tail distributions in the self-similar framework, as computed in \cite{GT09} and references therein for an in depth discussion of the phenomena. 
In particular,  it would not correct to  use approximating schemes that enforce local or global Maxwellian behavior will eventually generate errors.
In fact, in the case of the scalar homogeneous Boltzmann for binary inelastic collisions of Maxwell type, the scheme is able to accurately compute the evolution to self similar states with power tails, by exhibiting the predicted corresponding moment growth as performed in \cite{GT09}. }\\
The conservative spectral Lagrangian has also been implemented to numerically simulate a  gas mixture system for chemically interacting gases, \cite{MHGM14} and \cite{zhang-gamba-Landau16}, were recombination terms depend of mass ratios, even if the particle-particle interaction is elastic. IN particular, while each component of the gas mixture does not conserve energy, the total system does.  The resulting conservation scheme, then, enforces the proper collision invariants for the total system by enforcing a convex combination of the thermodynamic macroscopic quantities, but not for the collision invariants of  individual components.\\
Enforcing the system to conserve total quantities by the suitable constrain minimization problem associated to initial data for the mixture will select the correct equilibrium states associated to each system component.  A proof of this statement would require to adjust the  {\em Conservation Correction Estimate} of Lemma~\ref{t1} now extended to the adequate convex combination of collision invariants corresponding to the initial data of the system, as it was computed in \cite{MHGM14} for a $2\times2$ Neon Argon gas mixture, or a $5\times5$ multi-energy level gas mixture using the classical hard sphere model, as well as in \cite{zhang-gamba-Landau16} for an electron ion plasma mixture using the Landau equation for Coulomb potentials.\\
{The paper is organized as follows. In section 2,  the preliminaries and description of the spectral method for space homogenous Boltzmann equation are presented.  In section 3, we introduce the optimization problem proving the basic estimates including spectral accuracy and consistency results in both elastic and inelastic collisions in Theorem \ref{t1}.  In sections 4, 5 and 6 we develop the existence, convergence and error estimates for the elastic interactions scheme, which heavily relies  of the analytical properties of the model for space homogeneous, monoatomic, single component elastic interacting gas for hard potentials and integrable angular cross section kernel.  Finally, in section 7 we show local stability and long time convergence of the method.  In this section we prove that, in fact, all constant in the estimates are uniform in time.  We point out that it is possible to carry out this program for the inelastic framework of viscoelastic interactions, as all the necessary analytical tools are already available in \cite{ALod,AlonsoLods2014}. }
The methodology we follow is summarized in the following steps:\medskip
\begin{enumerate}
\item  In section 4 we prove \textit{a priori} estimates for the moments and the $L^{2}_{k}$-norms of the scheme under small negative mass assumption.  The analysis involves a coupled estimate on moments and $L^{2}$-norm due to the fact that spectral methods fundamentally need the $L^{2}$-theory.  An estimate for the amount of the negative mass produced by the scheme along time is proven as well.\medskip
\item We use the \textit{a priori} estimates of section 4 to prove global existence in section 5.  The key ingredient is to keep the negative mass formation under the numerical scheme in control.  We, then, show propagation of regularity.\medskip
\item In the section 6 we develop the error estimates of the scheme using the propagation of moments and Sobolev norms provided in sections 4 and 5.  The core of the document finishes, in section 7, with a result on the local stability and exponential convergence of the scheme to the thermal equilibrium.  This last part helps to make all constants found in previous sections uniform in time.\medskip
\end{enumerate}   

Finally, some conclusion are drawn in section 8 and a useful toolbox is given in the appendix.
\section{Preliminaries}
\subsection{The Boltzmann equation and its Fourier representation}
The initial value problem associated to the space homogeneous Boltzmann transport equation modeling the statistical evolution of a single point probability distribution function $f(t,v)$ is given by 
\begin{equation}\label{singleEq}
\frac{\partial f}{\partial t} (t, v)  =  Q(f, f)(t, v)\ \mbox{in}\ (0,T]\times\mathbb{R}^d,
\end{equation}
with initial data $f(0,v)=f_0$.  The weak form of the collision integral is given by 
\begin{align}\label{singleEq1}
\int_{\mathbb{R}^d} \!\!\Q(f, f)(v)  \phi(v)  \text{d}v\! = \!\int_{\mathbb{R}^{2d}} \!\int_{\mathbb{S}^{d-1}}\!\!
f(v, t) f(w, t) [\phi(v') \!-\!  \phi(v) ] B(|u|,\hat{u}\cdot\sigma) \text{d}\sigma \text{d}w \text{d}v \, ,
\end{align}
where the corresponding velocity interaction law exchanging velocity pairs $\{v,w\}$ into post-collisional pairs $\{v',w'\}$ is given by the law
\begin{equation}\label{singleEq2}
v'= v + \frac{\beta}{2}(|u|\sigma - u)\ \ \mbox{and}\ \ w' = w - \frac{\beta}{2}(|u|\sigma - u),
\end{equation}
where $\beta\in(1/2,1]$ is the energy dissipation parameter, $u=v-w$ is the relative velocity  and $\sigma\in \mathbb{S}^{d-1}$ is the unit direction of the post collisional relative velocity $u'=v'-w'$.  The parameter $\beta$ is related to the degree of inelasticity of the interactions, with $\beta=1$ being elastic and $\beta<1$ inelastic interactions.

The collision kernel, quantifying the rate of collisions during interactions, carries important properties that are of fundamental importance for the regularity therory of the Boltzmann collisional integral. It is assumed to be
\begin{equation}\label{ck}
B(|u|, \hat{u}\cdot\sigma) = |u|^{\lambda}\, b(\hat{u}\cdot\sigma)\, , \ \ \text{with}\ \ 0\leq \lambda\leq 1\ .
\end{equation}
 The scattering angle $\theta$ is defined by $\cos\theta= \hat{u} \cdot \sigma$, where the hat stands for unitary vector.  Further,  we assume that the differential cross section kernel $b(\hat{u}\cdot\sigma)$ is integrable in $\mathbb{S}^{d-1}$, referred as the {\sl Grad cut-off assumption}  \cite{grad1969}, and it is renormalized in the sense that
\begin{align}\label{grad-cut-off}
\int_{ \mathbb{S}^{d-1} }\! b(\hat{u} \cdot \sigma) \!\text{d}\sigma &\!= \!\big|\mathbb{S}^{d-2}\big|\!\int_{_{0}}^{^{\pi}} \!\!b(\cos\theta) \sin^{d-2}\theta \text{d}\theta = \big|\mathbb{S}^{d-2}\big|\!\!\int_{_{-1}}^{^{1}} \!\!b(s) (1 - s^2)^{(d-3)/2} \text{d}s \!= \!1\,,
\end{align}
where the constant $\big|\mathbb{S}^{d-2}\big|$ denotes the Lebesgue measure of $\mathbb{S}^{d-2}$.  The parameter $\lambda$ in \eqref{ck} regulates the collision frequency and accounts for inter particle potentials occurring in the gas.  These interactions are referred to as Variable Hard Potentials (VHP) whenever $0< \lambda< 1$, Maxwell Molecules type interactions (MM) for $\lambda=0$ and Hard Spheres (HS) for $\lambda=1$.  In addition, if kernel $b$ is independent of the scattering angle we call the interactions isotropic, otherwise, we refer to them as anisotropic Variable Hard Potential interactions.

It is worth mentioning that the weak form of the collisional form \eqref{singleEq1} also takes the following  weighted double mixing  convolutional form
\begin{align}\label{singleEq3}
\int_{\mathbb{R}^d} Q(f, f)(v) \, \phi(v) \, \text{d}v =
\int_{\mathbb{R}^{2d}} 
f(v) f(v-u) \, \G(v,u) \text{d}u \text{d}v \, .
\end{align}
The weight function defined by
\begin{align}\label{singleEq4}
\G(v,u) = \int_{\mathbb{S}^{d-1}}
\big[\phi(v') -  \phi(v) \big]\, B(|u|,\hat{u}\cdot\sigma)\, \text{d}\sigma
\end{align}
depends on the test function $\phi(v)$, the  collisional kernel $B(|u|, \hat{u}\cdot\sigma)$ from \eqref{ck} and the exchange of collisions law \eqref{singleEq2}. This is actually a generic form of a Kac master equation formulation for a binary multiplicatively interactive stochastic Chapman-Kolmogorov  birth-death  rate processes, were the weight function $\G(v,u)$ encodes the detailed balance properties,  collision invariants as well as existence, regularity and decay rate dynamics to equilibrium. 

We also denote by $'\!v$ and $'\!w$ the pre-collision velocities corresponding to $v$ and $w$.  In the case of elastic collisions (i.e. $\beta=1$) the pairs $\{'\!v,'\!w\}$ and $\{v',w'\}$ agree, otherwise, extra caution is advised. 

\medskip

\noindent
{\bf Collision invariants and conservation properties.}  The collision law \eqref{singleEq2} is equivalent to the following relation between the interacting velocity pairs 
$$v+w=v'+w'\quad \text{and} \quad |v|^2+|w|^2=|v'|^2+|w'|^2 - \beta(1-\beta) B(|u|,\hat{u}\cdot\sigma).$$  In particular, when testing with the polynomials  $\varphi(v)=1, \ v_j, \ |v|^2$ in $\mathbb R^d$, yields the following conservation relations  
\begin{eqnarray}\label{col-invariants}
\frac{\text{d}}{\text{d}t} \int_{_{\R^d}}  \hspace{-0.1in} f \begin{pmatrix} 1  \\ v_j\\   |v|^2  \end{pmatrix}
\text{d}v\, =  \int_{_{\R^{2d}}}   \hspace{-0.1in} f(v_*)f(v) \int_{\mathbb{S}^{d-1}}\begin{pmatrix} 0  \\ 0 \\   -\beta(1-\beta) 
\end{pmatrix}B(|u|,\hat{u}\cdot\sigma)\text{d}\sigma \text{d}v_{*}\text{d}v\,.
\end{eqnarray}
The polynomials that make the collisional integral vanish are called collision invariants. Clearly, the elastic case when $\beta=1$, the homogeneous Boltzmann equation has $d+2$ collision invariants and corresponding conservation laws, namely mass, momentum and kinetic energy.  For the inelastic case of $\beta<1$, the number of invariants and conserved quantities is $d+1$.  

Finally, when  testing with  $\varphi(v)=\log f(v)$ yields the inequality ($\mathcal H$-Theorem holding for the elastic case)
\begin{align}\label{H-theo}
\begin{split}
\frac{\text{d}}{\text{d}t} \int_{\R^d}  &f\log f \text{d}v =      \int_{\R^d}  Q(f)\log f \text{d}v \\
&=\int_{_{\R^{2d}\times \mathbb{S}^{d-1} }}  \hspace{-0.2in} f(w)f(v) \left(\log\left(\frac{f(w')f(v')}{f(w)f(v)}\right)  + \frac{f(w')f(v')}{f(w)f(v) } -1\right) B(|u|,\hat{u}\cdot\sigma)\text{d}\sigma \text{d}w\text{d}v \\
\ \ \  &\qquad \qquad\qquad\qquad+
\int_{_{\R^{2d}\times \mathbb{S}^{d-1} }}   \hspace{-0.1in} f(w)f(v) \left(\frac1{(2\beta-1)J_{\beta}} -1\right)  B(|u|,\hat{u}\cdot\sigma)\text{d}\sigma \text{d}w \text{d}v \\
\ \ \ & \hspace{-0.3cm}\le  \int_{_{\R^{2d}\times \mathbb{S}^{d-1} }}   \hspace{-0.1in} f(w)f(v) \int_{\mathbb{S}^{d-1}}
\left(\frac1{(2\beta-1)J_{\beta}} -1\right)  B(|u|,\hat{u}\cdot\sigma)\text{d}\sigma \text{d}w\text{d}v =0 \;  \text{ iff }   \; \beta=1\, . 
\end{split}
\end{align} 
Recall the following fundamental result in elastic particle theory:

\medskip
\noindent{\bf The Boltzmann Theorem} (for $\beta=1$).  $$\int_{\R^d}  Q(f)\log f =0 \ \iff \ \log f(v)= a+ {\bf b}\cdot v - c|v|^2\,,$$ where $f\in L^1(\R^d)$ for $c>0$, where the parameters $a$, $\textbf{b}$ and $c$ are determined by the initial state moments given by the $d+2$ collision invariants.  
\\
That means, given an initial state $f_0(v)\ge 0$  for a.e.  $v\in \R^d$  and $\int_{\R^d} f_0(v)(1+|v|^2)\, dv< \infty\,.$
In the limit as $t\to+\infty$, we expect that $f(t, v)$ converges to the {\em equilibrium Maxwellian} distribution, i.e.
\begin{equation}\label{eqMax}
f(t,v)\to \mathcal{M}_0[m_0,u_0,\Theta_0](v):={m_0}(2\pi \Theta_0)^{-d/2} \, \exp\left(-\frac{|v-u_0|^2}{2 \Theta_0}\right)\,,
\end{equation}
where the density mass, momentum and energy are defined by
$$
m_0:=\int_{\R^d} f_0 (v) \,dv\,,\quad u_0:=\frac1{m_0}\int_{\R^d} f_0 (v)\, dv\,,\quad \Theta_0:=
({d\,m_0})^{-1}     \int_{\R^d}  |v-u_0|^2\, f_0(v)\,dv\,.
$$

\medskip
\noindent {\bf The Fourier formulation of the collisional form.}
One of the pivotal points in the success of the spectral numerical method for the computation of the non-linear Boltzmann equation lies in the simplicity of the representation of the collision integral in Fourier space by means of its weak form.  Indeed taking the Fourier multiplier as the test function, i.e. 
\begin{equation*}
\psi(v)=\frac{e^{-i \zeta \cdot v}}{(\sqrt{2\pi})^d}
\end{equation*}
in the weak formulation \eqref{singleEq1}, where $\mathbf{\zeta}$ is the Fourier variable, one obtains the Fourier transform of the collision integral
\begin{align*}
\widehat{Q(f,f)}(\zeta)&=\frac{1}{(\sqrt{2\pi})^d} \int_{\mathbb{R}^d} Q(f, f) e^{-i \zeta \cdot v} \text{d}v \nonumber \\
&= \frac{1}{(\sqrt{2\pi})^d}\int_{\mathbb{R}^{2d}}\int_{\mathbb{S}^{d-1}} f(v) f(w)  B(|u|,\hat{u}\cdot\sigma) \left(e^{-i\zeta \cdot v'} - e^{-i \zeta \cdot v}\right) \text{d}\sigma \text{d}w \text{d}v\, .
\end{align*}
Thus, using (\ref{ck}, \ref{singleEq3}, \ref{singleEq4}) yields
\begin{align}\label{weakQHat2}
\begin{split}
\widehat{Q(f,f)}(\mathbf{\zeta})& = \frac{1}{(\sqrt{2\pi})^d} \int_{\mathbb{R}^{2d}} f(v) f(w) \int_{\mathbb{S}^{d-1}} |u|^{\lambda} b(\hat{u}\cdot\sigma)e^{-i \zeta \cdot v}\left(e^{-i \frac{\beta}{2}\zeta \cdot (|u|\sigma - u))} - 1\right) \text{d}\sigma \text{d}w \text{d}v \\
&=\frac{1}{(\sqrt{2\pi})^d} \int_{\mathbb{R}^d} \left(  \int_{\mathbb{R}^d} f(v) f(v - u) e^{-i\zeta \cdot v} \text{d}v\right) G_{\lambda, \beta}(u,\zeta) \text{d}u \\
&=\frac{1}{(\sqrt{2\pi})^d}\int_{\mathbb{R}^d}  \widehat{f\; \tau_{-u}f}(\zeta)  \, G_{\lambda, \beta}(u,\zeta) \, \text{d}u\,, 
\end{split}
\end{align}
where the weight function $ G_{\lambda, \beta}(u,\zeta) $ is defined by the spherical integration
\begin{equation} \label{gDef1}
G_{\lambda, \beta}(u,\zeta):=|u|^{\lambda} \int_{\mathbb{S}^{d-1}} b(\hat{u}\cdot\sigma)\left(e^{-i \frac{\beta}{2}\zeta \cdot (|u|\sigma - u))} - 1\right) \text{d}\sigma\,.
\end{equation}
Note that \eqref{gDef1} is valid for both isotropic and anisotropic interactions.  In addition, the function $G_{\lambda, \beta}(u, \zeta)$ is oscillatory and trivially bounded by $|u|^{\lambda}$ due to the integrability of $b(\cdot)$ from the Grad's cut-off assumption.  
Further simplification ensues for the three dimensional isotropic case where a simple computation gives
\begin{equation}
G^{\text{iso}}(u, \zeta) =|u|^{\lambda}
\left(e^{i \frac{\beta}{2}\zeta\cdot u}\; \text{sinc}\left(\frac{\beta |u| |\zeta|}{2}\right) - 1 \right).
\label{gDef2}
\end{equation}
In addition, recalling elementary properties of the Fourier transform yields
\begin{align*}
\widehat{f\;\tau_{-u}f}(\zeta)&=\frac{1}{(\sqrt{2\pi})^d}\;\hat{f}\ast\widehat{\tau_{-u}f}(\zeta) =\frac{1}{(\sqrt{2\pi})^d}\int_{\mathbb{R}^d}\hat{f}(\zeta-\xi)\widehat{\tau_{-u}f}(\xi)\text{d}\xi\nonumber\\
&=\frac{1}{(\sqrt{2\pi})^d}\int_{\mathbb{R}^d}\hat{f}(\zeta-\xi)\hat{f}(\xi)e^{-i\xi\cdot u}\text{d}\xi\,.
\end{align*}
Hence, using this last identity into  \eqref{weakQHat2},  we finally obtain the following  structure in Fourier space 
\begin{equation}
\widehat{Q(f,f)}(\zeta)=\frac{1}{(2\pi)^d} \int_{\mathbb{R}^d} \hat{f}(\zeta - \xi) \hat{f}(\xi) \widehat{G_{\lambda, \beta}}(\xi, \zeta) \text{d}\xi\,,
\label{qHatN2}
\end{equation}
where
\begin{equation}\label{hatG}
\widehat{G_{\lambda, \beta}}(\xi, \zeta) = \int_{\mathbb{R}^d} G_{\lambda, \beta}(u, \zeta) e^{-i \xi \cdot u} \text{d}u\,.
\end{equation}
That is,  the Fourier transform of the collision operator $\widehat{Q(f,f)}(\zeta)$ is a weighted convolution of the inputs in Fourier space with weight $\widehat{G_{\lambda, \beta}}(\xi, \zeta)$.

As an example, we compute the weight for the isotropic case in three dimensions.  Assume that $f$ has support in the ball or radius $\sqrt{3}L$, hence, the domain of integration for the relative velocity is the ball of radius $2\sqrt{3}L$.  Using polar coordinates $u=r\omega$,
\begin{align}\label{fourier-kernel1}
\begin{split}
\widehat{G^{\text{iso}}}(\xi, \zeta) &= \int^{\infty}_{0} \int_{\mathbb{S}^2} r^2 G^{\text{iso}}(r\omega, \zeta) e^{-i r \xi \cdot \omega} \text{d}\omega \text{d}r \\
&= 4 \int_0^{2\sqrt{3}L} r^{\lambda + 2}\left(\text{sinc}\left(\frac{r \beta |\zeta|}{2}\right) \text{sinc}\left(r |\frac{\beta}{2} \zeta-\xi|\right) - \text{sinc}\left(r |\xi|\right)\right)\text{d}r.
\end{split}
\end{align}
A point worth noting here is that the numerical calculation of expression \eqref{qHatN2} results in $O(N^{2d})$ number of operations, where $N$ is the number of discretization in each velocity component (i.e. $N$ counts the total number of Fourier modes for each  $d$-dimensional velocity) space.  However it may be possible to reduce the number of operations to $O(N^{2d-1}\text{log}N)$ for any anisotropic kernel and any initial state.  Due to the oscillatory nature of the weight function \eqref{fourier-kernel1} even in the simple case of 3 dimensions for the hard sphere case, when $b(\hat{u}\cdot\sigma)= 4\pi$, such  calculation can not be accomplished by $N\text{log}N$ if the initial state is far from a Maxwellian state or has an initial discontinuity, as claimed in \cite{filbetRusso1}.

\medskip
\noindent {\bf Notation and spaces.}  Before continuing with the discussion, we recall the definition of the Lebesgue's spaces $L^{p}_{k}(\Omega)$ and the Hilbert spaces $H^{\alpha}_{k}(\Omega)$.  These spaces will be used along the manuscript.  The set $\Omega$ could be any measurable set in the case of the $L^p_{k}$ spaces or any open set in the case of the $H^{\alpha}_{k}$ spaces, however, for our present purpose $\Omega \text{ is either } (-L,L)^{d} \text{ or } \mathbb{R}^{d}$ most of the time.
\begin{align*}
L^{p}_{k}(\Omega) : &= \Big\{ f \,:\,  \|f\|_{L^{p}_{k}(\Omega)}:=\Big(\int_{\Omega} \big| f(v)\langle v \rangle^{\lambda k} \big|^{p}\text{d}v\Big)^{\frac{1}{p}}<\infty\Big\}\,,\quad\text{ with } \; p\in[1,\infty)\,,\; k\in\mathbb{R}\,,\\
H^{\alpha}_{k}(\Omega) :&= \Big\{ f\; : \; \|f\|_{H^{\alpha}_{k}(\Omega)}:= \Big(\sum_{\beta\leq\alpha} \|D^{\beta}f\|^{2}_{L^{2}_{k}(\Omega)}\Big)^{\frac{1}{2}}<\infty\Big\}\,,\quad\text{ with }\; \alpha\in\mathbb{N}^{d}\,,\;k\in\mathbb{R}\,,
\end{align*}
where $\langle v \rangle := \sqrt{1+|v|^{2}}$.  The standard definition is used for the case $p=\infty$,
\begin{equation*}
L^{\infty}_{k}(\Omega) : = \Big\{ f \,: \,  \|f\|_{L^{\infty}_{k}(\Omega)}:=\text{esssup}\big|f(v)\langle v \rangle^{\lambda k}\big|<\infty\Big\}\,,\quad\text{ with }\; k\in\mathbb{R}\,.
\end{equation*}
It will be commonly used the following shorthand to ease notation when the domain $\Omega$ is clear from the context
\begin{equation*}
\|\cdot\|_{L^p_{k}(\Omega)} = \|\cdot\|_{L^p_{k}}\,,
\end{equation*}
and the subindex $k$ may be omitted in the norms for the classical spaces $L^{p}$ and $H^{\alpha}$.  In addition, following the notation and language of the classical analysis of the Boltzmann equation, and including the fact that numerical solutions are not nonnegative in general, the {\sl moments of a function} $f$ are denoted by
\begin{equation}\label{defmom}
m_{k}(f):=\int_{\mathbb{R}^{d}}\big|f(v)\big|\,|v|^{\lambda k}\,\text{d}v.
\end{equation}
\subsection{Choosing a computational cut-off domain}\label{sec-cut-off}    In order to make a  good approximation to the probability density $f(t,v)$, defined for all $v\in\Omega:=\mathbb{R}^d$,  solution of the dynamical homogeneous Boltzmann equation initial value problem, we need to solve the proposed  spectral numerical scheme in a computational domain given by the bounded  set  $\Omega_L\in \mathbb{R}^d$  for a  sufficiently large $N$ Fourier modes, as it will be defined next in the beginning of section~\ref{sec:consIsomoment}.
In particular, the global collision operator $Q(f,f)(t,v)$, defined weakly in \eqref{singleEq1}, needs to be approximated in such computational domain $\Omega_L$, that will be carefully chosen below for the specific task of solving the homogeneous Boltzmann equation, for a particularly chosen initial data being a probability density with a prescribed finite and positive initial mass and kinetic energy (i.e. the choice of the computational domain depends on the initial data, as it will be carefully explained).

It will be clear, after the discussion of the approximating numerical scheme for the space  homogeneous Boltzmann initial value problem and Theorem \ref{CT}, that there are two sources of error: one due to the mode truncation and other due to domain truncation.  Both are always present due to the global nature of the equation.  
{The key point in the choice of the  computational domain $\Omega_L=(-L,L)^d$ is that the time dynamics of the analytical  solution remains bounded  and decays with Maxwellians tails if initially so following the result in \cite{GPV08}. In particular, it is possible to choose a large enough cut-off length $L$, depending on the initial data with a Gaussian decay rate, whose approximating $g_0(v)$ satisfies condition \eqref{gas0}, and the  supp$\{g_0(v)\} \subset \Omega_{aL}$, for $0<a\ll1$. As a consequence the periodization of the domain is not necessary, since the analytical result from  \cite{GPV08} combined with the conservation algorithm, secures that numerical solution will take values very close to zero (i.e. below machine 	accuracy)  near the boundary $\partial\Omega_L$.  That means it is enough to choose  $\Omega_{L}$ such that, at least, most of the mass and energy of the true Boltzmann solution $f$ will be contained in it during the simulation time.  }

One possible strategy for choosing the size of $\Omega_L$   is as follows:  assume, without loss of generality, a bounded initial datum $f_0$ with compact support and having zero momentum $\int f_0\,v \,dv= 0$.  Then, 
\begin{equation}\label{contfo}
f_0(v) \leq \frac{C_0\,m_0}{ ({2\pi \Theta_{0}})^{d/2} } \,e^{-\frac{r_0 |v|^{2} }{ 2 \Theta_{0} } }\,,
\end{equation}
where $m_{0}:=\int f_0 $ is the initial mass, $\Theta_{0}:=\int f_{0}|v|^{2}$ is the initial temperature, and $r_0\in(0,1]$ and $C_0\geq1$ are the stretching and dilating constants.  Since the Boltzmann flow propagates Gaussian weighted Lebesgue norms, refer to \cite{desvillettes93, wennberg97, Bobylev97, BGP04, MV04, GPV08, AG, AG08, ACGM} for this and more related theoretical facts on the equation.  Thus, there are some uniform in time constants $r$ and $c$ depending on the moments of the initial data us much of the potential rates $\lambda$ and the angular part $b$ of the collision kernel, namely  $r:=r(f_0,\lambda,b)\in(0,r_0]$ and $C:=C(f_0,\lambda,b)\geq C_0\geq1$ such that
\begin{equation}\label{GaussianControl}
f(t,v) \leq \frac{C\,m_0}{ ({2\pi \Theta_{0}})^{d/2} }\,e^{-\frac{r |v|^{2} }{ 2 \Theta_{0} } }=: M(f_0,C,r)\,,\quad t >0\,.
\end{equation}
Now, choose a small quantity $\delta\ll 1$ being the mass proportion of the tails associated to the  Maxwellian $M(f_0,C,r)$ from \eqref{GaussianControl} that uniformly  controls the solution $f(t,v)$ as in \cite{GPV08}.  That is, 
\begin{align*}
\int_{\Omega^{c}_{L}}f(t,v)\langle v \rangle^{2} \text{d}v \leq \int_{\Omega^{c}_{L}}M(f_0,C,r)\langle v \rangle^{2} \text{d}v\leq \delta \int_{\Omega_{L}}f_0(v)\langle v \rangle^{2} \text{d}v=\delta(m_0+\Theta_0)\,.
\end{align*}
Therefore, the parameter $\delta$,  for the solution of the Boltzmann equation , thanks to the $L^{\infty}$ control of the solution in \cite{ GPV08},  is interpreted as a domain cut-off error tolerance that \emph{remains uniform in time} and solely depends on the approximated  initial state, say, $0\le g_0(v)$  on the chosen $\Omega_L$ so that  the magnitude of $\delta(m_0+\Theta_0)$ is well below machine accuracy.  Clearly,  the mass proportion $\delta$ must be small enough for  $\text{supp}(g_0)\subset\subset \Omega_{L}$.  Equivalently, one needs to choose the size of $L$, (or  the measure of the computational  domain $\Omega_L$), such that 
\begin{equation}\label{delta_error}
\frac{\int_{\Omega^{c}_{L}}M(f_0, C, r)\langle v \rangle^{2} \text{d}v}{m_0+\Theta_0}\leq \delta\approx 0\,.
\end{equation}
In order to minimize the computational effort,  one should pick the smallest of such domains, that is $\Omega_{L}$, such that\\
\begin{equation}\label{delta_error1}
 \text{for a fixed }\ a\ll1\, , \ \ \text{supp}(f_0)\subset \Omega_{aL}\ \ \text{and}\ 
\end{equation}
 that $\Omega^{c}_{aL}$  \eqref{delta_error}  in the sense  that the numerically approximated initial probability density vanishes in a neighborhood on boundary $\Omega_L$, to beyond several order down of machine accuracy.
  In addition, under this conditions we invoke the Restriction operators in  Sobolev spaces argument, such us \eqref{extension_estimate} in the subsection below, which allow us to make rigorous  semi discrete error estimates in Sobolev norms with respect to the solution $f(t,\cdot)\in \mathbb R^d$ of the homogeneous Boltzmann Cauchy problem (\ref{singleEq} - \ref{grad-cut-off}) under consideration. 
  
Finally, for such estimate \eqref{delta_error1} to be of practical use one would need to compute the precise value of the constants $C$ and $r$.  As a general matter, these constants comes from available analytical estimates, although quantitative, are likely far from optimal.  The result is that the choice \eqref{delta_error1} most of the times overestimate the size of the simulation domain.  It is reasonable then, for practical purposes, to simply set $r_o = r = 1$ and choose $C=C_o\geq1$ as the smallest constant satisfying \eqref{contfo} (which always exists for any compactly supported and bounded $f_0$).  That this choice of parameters is natural, it is noted from the fact that 
\begin{equation*}
\max\Big\{g_0\,,\, f_{\infty}:=\frac{m_0}{ (2\pi \Theta_{0})^{d/2}  }\,e^{-\frac{|v|^{2} }{ 2 \Theta_{0} } } \Big\} \leq M(f_0,C,1)\,,
\end{equation*}
with equality if and only if $f_0$ is the equilibrium Maxwellian as in \eqref{eqMax} (in such a case $C=1$). 

This propagation property secures a stable numerical simulation of the Boltzmann equation, provided the numerical preservation of the conservation laws or corresponding collision invariants.  It also secures, as we will see, the convergence of the numerical scheme to the analytic solution of the initial value problem and the correct long time evolution of such numerical approximation.  In this way, the numerical scheme will converge to the equilibrium Maxwellian as defined in \eqref{eqMax}.

We note that the discussion of this section is fairly independent of the choice computational scheme and applies to new approaches such as the recently developed in \cite{DGBTE-Zhang-Gamba14} for a Galerkin approach to the computation of the space homogeneous  Boltzmann equation for binary interactions.

\subsection{Fourier series,  projections.}\label{ProExt}
In the implementation of a spectral method the single most important analytical tool is the Fourier transform.  Thus, for $f\in L^{1}(U)$ with $U$ open in $\mathbb{R}^{d}$, the Fourier transform is defined by
\begin{equation}\label{e1p}
\widehat{f}(\zeta):=\frac{1}{(\sqrt{2\pi})^d}\int_{U}f(v)e^{-i\zeta\cdot v}\text{d}v\,.
\end{equation}
The Fourier transform allow us to express the Fourier series in a rather simple and convenient way.  Indeed, fixing a domain of work $\Omega_L:=(-L,L)^d$ for $L>0$, recall that for any $f\in L^2(\Omega_L)$ one can use the \textit{Fourier series} to express $f$ as
\begin{equation}\label{e1}
f(v) = \frac{1}{(2L)^d}\sum_{k\in\mathbb{Z}^d}\widehat{f}(\zeta_k)e^{i\zeta_{k}\cdot v},
\end{equation}
where $\zeta_k:=\frac{2\pi k}{L}$ are the spectral modes and $\widehat{f}(\zeta_k)$ is the Fourier transform of $f$ evaluated in such modes.\\

The mode projection operator is defined as $\Pi^N_{L}: L^2(\Omega_L) \rightarrow L^2(\Omega_L)$ as
\begin{equation}\label{e2}
\left(\Pi^{N}_{L} f\right)(v): = \left(\frac{1}{(2L)^d}\sum_{|k|\leq N} \widehat{f}(\zeta_k) e^{i \zeta_k \cdot v}\right),
\end{equation}
in other words, it is the \textit{orthogonal projection} on the ``first $N^{d}$'' basis elements.  Also observe that for any integer $\alpha$ the derivative operator commutes with the projection operator in $H^{\alpha}_{o}(\Omega_{L})$.  Indeed, note the identity for any $f\in H^{\alpha}_{o}(\Omega_{L})$,
\begin{align}\label{derivative_projection}
\begin{split}
\partial^{\alpha}\left(\Pi^{N}_{L} f\right)(v)&=\frac{1}{(2L)^d}\sum_{|k|\leq N} (i\zeta_{k})^{\alpha}\widehat{f}(\zeta_k) e^{i \zeta_k \cdot v}\\
&=\frac{1}{(2L)^d}\sum_{|k|\leq N} \widehat{\partial^{\alpha}f}(\zeta_k) e^{i \zeta_k \cdot v}=\left(\Pi^{N}_{L}\partial^{\alpha} f\right)(v)\,.
\end{split}
\end{align}
Recall that Parseval's theorem readily shows 
\begin{enumerate}
\item  $\big\|\Pi^{N}_{L}f\big\|_{L^{2}(\Omega_L)} \leq \big\| f  \big\|_{L^{2}(\Omega_L)} $ for any $N$, and with equality for $N=\infty$.  Also,
\item  $\big\|(\textbf{1} -  \Pi^{N}_{L})f\big\|_{L^{2}(\Omega_L)} \searrow 0$ as $N\rightarrow \infty$.
\end{enumerate}

\medskip
\noindent\textbf{Extension operator for Sobolev regularity propagation.}  The restriction of the original problem posed in $\mathbb{R}^{d}$ to an approximation problem posed in a bounded domain $\Omega_{L}$ introduce some technical issues at the boundary generated by the truncation.  We deal with this problem introducing the following scaled cut-off function defined by
\begin{align}\label{cutoff}
\begin{split}
&\chi(v):=\chi_{L}(v)=\phi(v/L),  \ \text{with} \ \phi \ \text{a smooth nonnegative function,}\\
&\text{such that}\ \ \text{supp}\{\phi\}\subset 0.99[-1,1]^{d} \, ,\ \ \text{with} \ \ \phi\equiv 1\ \text{in}\  0.95[-1,1]^{d} \,.
\end{split}
\end{align} 
The cut-off function $\chi$ allows for the scheme propagation of higher Sobolev regularity estimates (it is not necessary for $L^{2}$-convergence) as it smooths out the boundary without incurring in meaningful error (provided $\Omega_{L}$ was well chosen as previously discussed in subsection~\ref{sec-cut-off}
).  Using the product rule, it follows that
\begin{equation}\label{extension_estimate}
\|\chi g\|_{H^{\alpha}(\Omega_{L})}\leq \| \chi \|_{\mathcal{C}^{\alpha}}\|g\|_{H^{\alpha}(\Omega_{L})} \leq C\,\|g\|_{H^{\alpha}(\Omega_{L})}\,,
\end{equation}
for any function $g\in H^{\alpha}(\Omega_{L})$.  Note also that the constant $C:=C_{\chi}$, that controls the operator norm, can be taken independent of $L\geq1$.  It is important to observe that the function $\chi g$ vanishes in near  $\partial\Omega_L$, and so  it can be considered as a function in $H^{\alpha}(\mathbb{R}^{d})$ after using the extension operator who assigns the zero value to any point in the  complement of  $\Omega_L$, that is  $E(\chi g) = 0 $ in $\mathbb R^d  \backslash\Omega_L$. In addition  the Sobolev norms of such extension coincide with those  of the restricted $\chi g $, which takes values in compactly supported set in   $\Omega_L$ that vanishes in a neighborhood of the boundary $\partial\Omega_L$, relative to  $\Omega_L$. That precisely means
\begin{equation}\label{extension-norm}
\|E(\chi g)\|_{H^{\alpha}(\mathbb R^d)} = \|\chi g\|_{H^{\alpha}(\Omega_{L})}\,,\qquad g\in H^{\alpha}(\Omega_{L})\,. 
\end{equation}
Therefore our choice of the  the cut-off function $\chi$ enable us to implement an extension operator by null values to all space (for a full discussion of extension operators see \cite{SE}).   These properties will be useful when comparing the continuous and semi-discrete solutions, which lie in different domains.  Furthermore, in the case of $L^{2}$-convergence one can simply take $\chi\equiv1$.

\begin{rem}\label{why-not-period}
A common technique found in the literature to deal with the domain truncation is by periodization of the initial data.  Why we do not periodized the initial data, but rather use the extension method on Sobolev spaces for functions that vanish in a given bounded domain?  The answer is that the approximated data and solution in our problem are probability densities that rapidly decay at large values of velocities $v$.  In the particular case of the homogeneous Boltzmann equation approximation for hard potentials and angular integrable collision cross section, as it is developed in this theory, the  crucial issue is to choose the  computational domain large enough, depending on the initial data as previously discussed.  Under this choice, the cut-off function $\chi$ effectively implements an extension at the cost of a negligible error, as it will be show to be of the order $O(1/L^{\lambda k})$ with $k>0$ depending on the number of moments of the initial data.  
\end{rem} 
\begin{rem}\label{rem2.2} 
In this deterministic approach, as much as with Montecarlo methods like the Bird scheme~\cite{bird}, the  $x$-space inhomogeneous Hamiltonian transport for non-linear collisional forms are performed by time operator splitting algorithms. That means, depending on the problem, the computational $v$-domain $\Omega_L$  can be updated with respect to the characteristic flow associated to  underlying Hamiltonian dynamics.
\end{rem}


\section{Spectral conservation method}\label{sec:consIsomoment}
 We first introduce a formal  analytical  viewpoint needed to study  the convergence, stability and error estimates for the semi-discrete solution associated with  the spectral method derived in \cite{GT09}.
 
After the cut-off domain $\Omega_L$ has been fixed, we applied the projection operator \eqref{e2} to
 both sides of equation \eqref{singleEq} to arrive to
\begin{equation}\label{fourierBTE00}
\frac{\partial \Pi^{N}_{L} f}{\partial t} (t,v)  =  \Pi^{N}_{L} Q(f, f)(t,v)\,,  \; \mbox{ in }\ (0,T]\times\Omega_L\,.
\end{equation}
Then, it is reasonable to expect that for such a domain $\Omega_L$ and for sufficiently large number of modes $N$ the approximation
\begin{equation}\label{fourierBTE0}
\Pi^{N}_{L} Q(f, f)\sim \Pi^{N}_{L} Q(\Pi^{N}_{L} f,\Pi^{N}_{L} f)\,,  \; \mbox{ in }\ (0,T]\times\Omega_L
\end{equation}
will be valid. \\ 
Next, there are two issues worth noting: (1) for functions supported in $\Omega_{L}$ the gain operator $Q^{+}$ is supported in $\Omega_{\sqrt{2}L}$, thus, we will consider it, for simplicity, as a function in $\Omega_{2L}\, $, and (2) the operator $Q^{-}$ can be exactly computed with a small computational effort since it is a multiplication operator with a standard convolution.  As a consequence,  one is lead to consider the scheme
\begin{equation}\label{fourierBTE-0}
\begin{split}
\frac{\partial g}{\partial t} (t,v) &= \Pi^{N}_{2L}Q^{+}(\chi g,\chi g)(t,v) - Q^{-}(g,\chi g)(t,v)\\
&=:Q_{u}(g,g)(t,v), \ \ \mbox{ in } (0,T]\times\Omega_L\,, \\
g_0(v):&=g^N_0=\Pi^{N}_{L}f_0(v) \quad \mbox{initial data},
\end{split}
\end{equation}
and expect that it should be a good approximation to $\Pi^{N}_{L} f$.  Here, $Q_{u}$ stands for the unconserved collision operator.  In other words, we define the numerical solution to be $g_{N}:=g$ and expect to show that this finite mode solution will be a good approximation to the solution of the Boltzmann problem in the cut-off domain, that is $g\approx f$ in $\Omega_L\,$, provided the number of modes $N$ used is sufficiently large.  Classical spectral accuracy theorems would guarantee such approximation, yet, fixing the number of Fourier nodes to say $N^*$ would strip the conservation properties, as $Q_u$ does not preserve the $d+2$ collision invariants after each time step, and that generates a source of cumulative error that heavily constrain the meaningful simulation time of the scheme.\\
This problem was overcame in \cite {GT09} with the conservative spectral scheme we are now analysing. They introduce a conservation correction by solving a Lagrangian constrained minimization problem each time step (with $O(N)$ in computational complexity), where the objective function to be minimize is the $L^2(\Omega_L)$-distance from the unconseved $Q_u$ to the minimizer $X^*=:Q_c$ subject to the constrain of preserving the $d+2$ collision invariants. To be more precise, the following problem is computationally solved in \cite{GT09}:

\medskip
\noindent\textit{Minimization Elastic Problem (E):} Consider the Banach space
\begin{equation}\label{elasticLagrangeProblem0}
\mathcal{B}^e=\left\{X\in L^{2}(\Omega_L): \int_{\Omega_L} X=\int_{\Omega_L} Xv=\int _{\Omega_L} X|v|^{2}=0\right\},
\end{equation}
and the minimization problem
\begin{equation}\label{elasticLagrangeProblem}
X^*:=\min_{X\in\mathcal{B}^e } \, \mathcal{A}^e(X) := \min_{X\in\mathcal{B}^e } \int_{\Omega_L} \big(Q_u(f,f)(v) - X\big)^2 \text{d}v.
\end{equation}
The solution of this problem applied to our semi discrete framework will be addressed in the next subsection~\ref{subsec:consDIT}. 
It can be solved by an algorithm, describe below in (\ref{sdcm1} - \ref{min-problem}),  that delivers a unique explicit algorithm discrete vector form \begin{equation}\label{min_c}
Q_c(f,f):= X^*\in N^d \;, 
\end{equation}  
associated to any discretization of $f$ on $N^d$ Fourier modes, where the constrain in 
\eqref{elasticLagrangeProblem0} is given by the linear equation $\textbf{C}^e Q_c \ = \ \textbf{a}^e$, where the vector $ \textbf{a}^e =0 \in N^{d+2}$ for the elastic problem or $ \textbf{a}^e =0 \in N^{d+1}$ for the inelastic one. The matrix $\textbf{C}^e$ is explicitly pre-computed depending on the quadrature rule used to compute the integrals associated to the collision invariants.
 
\medskip
 
In the following sections we intent to prove this formalism under reasonable assumptions.  In fact, we study a modification of this problem, namely, the convergence towards $f$ of the solution $g$ of the problem
\begin{equation}\label{fourierBTE}
\begin{split}
\frac{\partial g}{\partial t} (t,v) &= Q_{c}(g,g)(t,v), \ \ \mbox{ in } (0,T]\times\Omega_L\,, \\
g_0(v):&=g^N_0=\Pi^{N}_{L}f_0(v) \ \ \mbox{initial data}\,,
\end{split}
\end{equation}
with   $Q_c(f,f)$ the solution of the Lagrangian constrained problem (\ref{elasticLagrangeProblem}, \ref{min_c}), and the initial data 
 $g_0$ satisfies the following condition
\begin{equation}\label{gas0}
 \frac{\int_{\{g_0<0\}}|g_0(t,v)|\langle v \rangle^{2}\text{d}v}{\int_{\{g\geq0\}}g_0(t,v) \langle v \rangle^{2}\text{d}v}\ \leq\ \epsilon,\quad \text{and}\quad \|g_0\|_{L^{2}(\Omega_L)}<\infty\,,
\end{equation}
for some fixed $0<\epsilon\le 1/4$,  where the operator $Q_c(g,g)$ is defined as the  $L^2(\Omega_L)$-closest function to $Q_{u}(g,g)$ having null mass, momentum and energy.  
\smallskip

We summarize the main results on convergence, error estimates and asymptotic behavior in the  the following theorem, whose rigorous proof  is developed in the rest of the manuscript.  As mentioned in the introduction, the following theorem is proved for the classical elastic model $\beta=1$.  A rigorous proof for the inelastic model can be done, at least, for some special regimes such as the viscoelastic particle model \cite{ALod,AlonsoLods2014} with analog  arguments.  Additional considerations about self-similar scaling are needed to obtain sharp longtime behavior associated to the model, which will be properly addressed by the authors in an incoming manuscript.

\begin{thm}[\textbf{Error estimates and convergence to Maxwellian equilibrium}]\label{CT}
Fix a nonnegative initial data  $f_0\in L^{1}_{k}\cap L^{2}(\mathbb{R}^{d})$ with $k\geq k_{*}(f_{0})\geq2$, and let $f\geq0$ be the solution of the Boltzmann equation \eqref{singleEq} with \eqref{grad-cut-off}.  Then, there exist a cut-off domain $L_{0}(f_0)>0$ and a number of modes $N_0:=N(L_0,f_0)>0$ such that \medskip
\begin{itemize}
\item [1.] {\bf Semi-discrete existence and uniqueness: }\  Taking $g_0=\Pi^{N}_{L}f_0$, the semi-discrete problem \eqref{fourierBTE} has a unique solution $g\in\mathcal{C}(0,T; L^{1}_{k}\cap L^{2}(\Omega_{L}))$ for any $T>0$, $L\geq L_{0}$\,, $N\geq N_0$\,.
\medskip
\item [2.]  {\bf $L^2_{k'}$-error estimates:}\ Taking $f_0\in L^{1}_{2}\cap L^{2}_{k}(\mathbb{R}^{d})$, $k''\geq0$, $k_{*}(f_{0}) \leq k'\leq k-1-\frac{d^{+}}{2\lambda}-k''$, then
\begin{align*}
\hspace{0.5cm}&\sup_{t\geq0}\|f - g \|_{L^{2}_{k'}(\Omega_L)} \leq C(f_0)\big(\|f_{0}-g_{0}\|_{L^{2}_{k'}(\Omega_L)} \\
&\hspace{1cm}+ O(L^{\lambda k'}/N^{(d-1)/2}) + O(1/L^{d/2+\lambda k''})\big)^{\frac{1}{1+\theta}}\,,\quad L\geq L_{0}\,,\; N\geq N_0\,.
\end{align*}

\item [3.]   {\bf $H^{\alpha}_{k'}$-error estimates:}\  For the smooth case, taking $f_0\in L^{1}_{2}\cap H^{\alpha_0}_k(\mathbb{R}^{d})$, $k''\geq0$, $k_{*}(f_{0}) \leq k'\leq k - 1 - \alpha/2 - \frac{d^{+}}{2\lambda} - k''$, it follows for any $0\leq\alpha\leq\alpha_0$
\begin{align*}
&\sup_{t\geq0}   \|f-g\|_{H^{\alpha}_{k'}(\Omega_L)} \leq C(f_0)\big(
 \|f_{0}-g_{0}\|_{H^{\alpha}_{k'+\alpha/2}(\Omega_L)} \\
& + O(L^{\lambda (k'+\alpha/2) + \alpha_0}/N^{(d-1)/2 + \alpha_0 }) + O(1/L^{d/2+\lambda k''})\big)^{\frac{1}{1+\theta}}\,,\quad L\geq L_{0}\,,\;N\geq N_0\,.\medskip
\end{align*}
In all cases $k, k' \geq k_{*}(f_{0})\geq 2$ where $k_{*}(f_{0})$ is a threshold required that only depends on $f_0$.  Also, the constant $C(f_0):=C(k',\alpha,f_0)$ in items 2. and 3. depend on $f_0$ by means of its initial regularity, and the constant $\theta:=\theta(k',\alpha)>0$. 

\medskip

\item [4.]   {\bf Convergence to the equilibrium Maxwellian:}\ For every $\delta>0$ there exist a simulation time 
$T(\delta) \sim \nu^{-1}\ln\big( \|f_{0}\|_{H^{\alpha}(\mathbb{R}^{d})}/\delta\big)$ such that for any $\alpha\leq\alpha_0$
\begin{equation*}
\sup_{ t \geq T(\delta) }   \|g - \mathcal{M}_0  \|_{H^{\alpha}(\Omega_L)}  \leq \delta\,, \quad L\geq L_0\,,\; N\geq N_0\,,
\end{equation*}
where $\nu>0$ is the spectral gap of the linearized Boltzmann operator, and $\mathcal{M}_0$ is the equilibrium Maxwellian \eqref{eqMax} having the same mass, momentum and kinetic energy of the initial datum $f_0$.
\end{itemize}
\end{thm} 
\noindent
The proof of these statements in Theorem~\ref{CT} is done in the next four sections.   Before starting with the details of such proof, we introduce the shorthand notation 
\begin{equation}\label{O_r}
O_{r}:=O(L^{-r})\,, \quad  \ r>0\, , 
\end{equation}
which will be extensibly used throughout the manuscript.

\subsection{Conservation Method - An Extended Isoperimetric problem}
\n In the course of this section we fix $f\in L^{2}(\Omega_L)$.  Due to the truncation of the velocity domain the unconserved discrete operator $Q_{u}\in N^d$ defined for $N^d$ Fourier modes, as a function in $\Omega_{L}$, does not preserve mass, momentum and energy.  Such conservation property is at the heart of the kinetic theory of the Boltzmann equation, thus, it is desirable for a numerical solution to possess it.  In order to achieve this, we enforce these moment conservation artificially by imposing them as constraints in a optimization problem.

\smallskip

Hence,  we first focus on the general form of solution of the minimization problem (\ref{elasticLagrangeProblem0}, \ref{elasticLagrangeProblem}), whose proof   is presented next.
\begin{lem}[Elastic Lagrange Estimate]\label{l1}
The problem \eqref{elasticLagrangeProblem} has a unique minimizer given by
\begin{equation}\label{lagrange_ela}
Q_c(f,f)(v) := X^{\star}= Q_u(f,f)(v) - \frac{1}{2} \Big(\gamma_1 + \sum_{j=1}^d \gamma_{j+1} v_{j} + \gamma_{d+2} |v|^2\Big), 
\end{equation}
where $\gamma_j$, for $1\leq j \leq d+2$, are Lagrange multipliers associated with the elastic optimization problem.  They are given by
\begin{align}\label{Lconstrains}
\begin{split}
\gamma_1 &= O_{d} \rho_u + O_{d+2} e_u \, , \\
\gamma_{j+1} &= O_{d+2} \mu^j_u \, , \quad j = 1, 2,\cdots,d, \\
\gamma_{d+2} &= O_{d+2} \rho_u + O_{d+4} e_u \,.
\end{split}
\end{align}
with $O_{r}$ defined in  \eqref{O_r} and the parameters $\rho_u, e_u, \mu_u^j$ are the numerical moments of the unconserved numerical collision operator, defined below in \eqref{constraintEqns}.  The minimized objective function can be estimated by
\begin{align}\label{lagrangeEstimate}
\begin{split}
\mathcal{A}^e&\big(Q_c(f,f)(v)\big) =  \big\| Q_u(f,f) - Q_c(f,f)(v)  \big\|^2_{L^2(\Omega_L)} \\
&\leq C(d) \Big( 2\gamma^2_1 L^d + \sum_{j=1}^d \gamma^2_{j+1}L^{d+2} + \gamma^2_{d+2} L^{d+4} \Big) \leq \frac{C(d)}{L^d} \, \Big( \rho^2_u + \frac{e^2_u}{L^{d+1}}  + \sum_{j=2}^{d+1} \mu^2_{j}  \, \Big)\,.
 \end{split}
\end{align}
In the particular case of dimension $d=3$  the estimate becomes
\begin{align}\label{lagrangeEstimate-d3}
\begin{split}
\big\| Q_u(f,f) - &Q_c(f,f) \big\|^2_{L^2(\Omega_L)} 
\\
&= 2\gamma_1^2 L^3  + \tfrac{2}{3} \sum_{j=2}^{4}\gamma_j^2 L^{5} + 4 \gamma_1\gamma_d L^{5} + \tfrac{38}{15} \gamma_{5}^2\,  L^{7}\le\frac{C}{L^3} \, \Big( \rho^2_u + \frac{e^2_u}{L^4}  + \sum_{j=2}^{4} \mu^2_{j}  \, \Big).
\end{split}
\end{align}
\end{lem}
\begin{proof}
From calculus of variations when the objective function is an integral equation and the constraints are integrals, the optimization problem can be solved by forming the Lagrangian functional and finding its critical points.  Set
\begin{align*}
\psi_1(X):&= \int_{\Omega_L} X(v) \text{d}v\,,\\
\psi_{j+1}(X):&= \int_{\Omega_L} v_j X(v) \text{d}v\,, \quad \forall j = 1, 2, \cdots , d,\\
\psi_{d+2}(X):&= \int_{\Omega_L} |v|^2 X(v) \text{d}v\,,
\end{align*}
and define
\begin{align*}
\mathcal{H}(X,X', \boldsymbol{\gamma}) := \mathcal{A}^e(X) + \sum_{i = 1}^{d+2} \gamma_i \psi_i(X) =\int_{\Omega_L} h(v,X, X', \boldsymbol{\gamma}) \text{d}v.
\end{align*}
We introduced
\begin{equation*}
h(v, X, X', \boldsymbol{\gamma}):=\big(Q_u(f,f)(v) - X(v)\big)^2 + \gamma_1 X(v) + \sum_{j=1}^d \gamma_{j+1} v_j X(v) + \gamma_{d+2} |v|^2 X(v).
\end{equation*}
In order to find the critical points one needs to compute $D_{X}\mathcal{H}$ and $D_{\gamma_j}\mathcal{H}$.  The derivatives $D_{\gamma_j}\mathcal{H}$ just retrieves the constraint integrals.  For multiple independent variables $v_j$ and a single dependent function $X(v)$ the Euler-Lagrange equations are
\begin{equation*}
D_2 h(v, X,  X', \boldsymbol{\gamma})=\sum_{j = 1}^{d} \frac{\partial D_3h}{\partial v_j} (v, X, X', \boldsymbol{\gamma})=0\,.
\end{equation*}
We used the fact that $h$ is independent of $X'$.  This gives the following equation for the conservation correction in terms of the Lagrange multipliers
\begin{align}\label{criticalQ}
\begin{split}
2\big(X(v) - Q_u(f,f)(v)\big) + &\gamma_1 + \sum_{j=1}^d \gamma_{j+1} v_{j} + \gamma_{d+2} |v|^2 = 0,\\
\mbox{ and therefore},\ \ Q_c(f,f)(v) = &X^{\star}(v) := Q_u(f,f)(v) - \frac{1}{2}\Big(\gamma_1 + \sum_{j=1}^d \gamma_{j+1} v_{j} + \gamma_{d+2} |v|^2\Big).
\end{split}
\end{align}
Let $g(v, \gamma) = \gamma_1 + \sum_{j=1}^d \gamma_{j+1} v_{j} + \gamma_{d+2} |v|^2$. Substituting \eqref{criticalQ} into the constraints $\psi_{j}(X^{\star})=0$ gives
\begin{align}\label{constraintEqns}
\rho_u \mathrel{\mathop:}=& \int_{\Omega_L} Q_u(f,f)(v) \text{d}v = \frac{1}{2} \int_{\Omega_L} g(v, \gamma) \text{d}v \nonumber \\
\mu^j_u \mathrel{\mathop:}=& \int_{\Omega_L} v_j Q_u(f,f)(v) \text{d}v= \frac{1}{2} \int_{\Omega_L} v_j g(v, \gamma) \text{d}v, \quad\quad j = 1, 2,\cdots, d, \\
e_u \mathrel{\mathop:}=& \int_{\Omega_L} |v|^2 Q_u(f,f)(v) \text{d}v = \frac{1}{2} \int_{\Omega_L} |v|^2  g(v, \gamma) \text{d}v\, .\nonumber
\end{align}
Identities \eqref{constraintEqns} form a system of $d+2$ linear equations with $d+2$ unknown variables that can be uniquely solved. Solving for the critical $\gamma_j$,
\begin{align}\label{criticalLambda}
\gamma_1 &= O_d \rho_u + O_{d+2} e_u \, ,\nonumber \\
\gamma_{j+1} &= O_d \mu^j_u \, , \quad\quad j = 1, 2,\cdots, d, \\
\gamma_{d+2} &= O_{d+2} \rho_u + O_{d+4} e_u \, ,\nonumber
\end{align}
Hence, relation \eqref{Lconstrains} holds.  Substituting these values of critical Lagrange multipliers \eqref{criticalLambda} into \eqref{criticalQ} gives explicitly the critical $Q_c(f,f)(v)\ := \ X^\star(v)$.  Moreover, the objective function $\mathcal{A}^e(X)$ can be computed at its minimum as
\begin{align}\label{iso-estimate}
\begin{split}
\mathcal{A}^e(Q_c(f,f)) &= \big\|Q_u(f,f) - Q_c(f,f) \big\|^2_{L^2(\Omega_L)} = \int_{\Omega_L} \big(Q_u(f,f)(v) - X^\star(v)\big)^2 \text{d}v \\
&= \tfrac{1}{4} \int_{\Omega_L} \Big(\gamma_1 + \sum_{j=1}^d \gamma_{j+1} v_{j} + \gamma_{d+2} |v|^2\Big)^2 \text{d}v \\
&\qquad\le \tfrac{d+2}{4} \int_{\Omega_L} \Big( \gamma^2_1 + \sum_{j=1}^d (\gamma_{j+1} v_{j} )^2 + \gamma^2_{d+2} |v|^4\Big) \\
&\qquad\quad\quad\le C(d) \Big( 2\gamma^2_1 L^d + (\sum_{j=1}^d \gamma^2_{j+1})L^{d+2} + \gamma^2_{d+2} L^{d+4} \Big), 
\end{split}
\end{align}
where $C(d)$ is an universal constant depending on the dimension of the space.  Hence, using the relation 
\eqref{criticalLambda} to replace  into  the right hand side of \eqref{iso-estimate},  yields a bound from above
to the difference of the conserved and unconserved approximating collision operators
\begin{align}\label{iso-estimate2}
\big\|Q_u(f,f) - Q_c(f,f) \big\|^2_{L^2(\Omega_L)} 
\le \frac{C(d)}{L^d} \, \Big( \rho^2_u + \frac{e^2_u}{L^{d+1}}  + \sum_{j=2}^{d+1} \mu^2_{j}  \, \Big)\, , 
\end{align}
and therefore, the lagrage estimate \eqref{lagrangeEstimate} holds.  Upon simplification one can obtain a detailed estimate for the $3$-dimensional case, given by 
\begin{align}\label{lagrangeEstimate-d3_2}
\begin{split}
\big\|Q_u(f,f) - Q_c(f,f) \big\|^2_{L^2(\Omega_L)} &= 2\gamma_1^2 L^3  + \tfrac{2}{3}(\gamma_2^2 + \gamma_3^2 + \gamma_4^2)L^5 +4 \gamma_1 \gamma_5 L^5 + \tfrac{38}{15} \gamma_5^2 L^7 \\
&\le\frac{C}{L^3} \, \Big( \rho^2_u + \frac{e^2_u}{L^4}  + \sum_{j=2}^{4} \mu^2_{j}  \, \Big)\, ,
\end{split}
\end{align}
which is precisely \eqref{lagrangeEstimate-d3}.  That this critical point is in fact the unique minimizer follows from the strict convexity of $\mathcal{A}^e$.
\end{proof}
\n Similarly, as it was proposed also in the simulations of \cite{GT09}, one can form the optimization problem for the inelastic case. The only difference is that now only $d+1$-collision invariants are conserved:

\medskip

\noindent\textbf{Inelastic Problem (IE):} Minimize in the Banach space
\begin{equation*}
\mathcal{B}^i=\left\{X\in L^{2}(\Omega_L): \int_{\Omega_L} X=\int_{\Omega_L} Xv=0\right\},
\end{equation*}
the functional
\begin{equation}\label{inelasticLagrangeProblem}
\mathcal{A}^i(X) := \int_{\Omega_L} \big(Q_u(f,f)(v) - X\big)^2 \text{d}v.
\end{equation}
\medskip

As in the elastic case, we state  a rather similar analog to the  Lagrange estimate for the inelastic collision law. The proof of this statement is similar to the case of elastic interacions, and we leave it to the readers.
\begin{lem}[Inelastic Lagrange Estimate]\label{l2}
The problem \eqref{inelasticLagrangeProblem} has a unique minimizer given by
\begin{equation}\label{lagrange_ine}
Q^{ine}_c(f,f)(v) := X^\star(v) = Q_u(f,f)(v) - \frac{1}{2} \Big(\gamma_1 + \sum_{j=1}^d \gamma_{j+1} v_{j}\Big) \,.
\end{equation}
The $\gamma_j$ are Lagrange multipliers associated with the {\sl inelastic} optimization problem given by
\begin{align}\label{inelasticCriticalLambda}
\begin{split}
\gamma_1 &= O_d \rho_u \, , \\
\gamma_{j+1} &= O_{d+2} \mu^j_u \, , \quad j = 1, 2,\cdots, d \, ,
\end{split}
\end{align}
In particular, for the three dimensional case the minimized objective function is
\begin{equation}\label{inelasticLagrangeEstimate}
\mathcal{A}^{i}(X^\star) = \big\|Q_u(f,f) - Q^{ine}_c(f,f) \big\|^2_{L^2(\Omega_L)} = 2\gamma_1^2 L^3  + \tfrac{2}{3}\big(\gamma_2^2 + \gamma_3^2 + \gamma_4^2\big)L^5 \,.
\end{equation}
\end{lem}

\medskip

\noindent{\bf Conservation Correction Estimate.}  We develop here an useful estimate between the unconserved and conserved discrete collisional forms.

\medskip

\noindent \textbf{Definition.}  For any fixed $f\in L^{2}(\Omega_L)$ the \textit{conserved operator} $Q_c(f,f)$ is defined as the minimizer of problem \textbf{(E)} defined by \eqref{lagrange_ela} (or problem \textbf{(IE)} in the inelastic case defined by \eqref{lagrange_ine}). 

\medskip

\noindent Note that the minimized objective function \eqref{lagrangeEstimate} in the elastic optimization problem depends only on the unconserved moments $\rho_u, \mu_u$, and $e_u$ of $Q_u(f,f)$. Since these quantities are expected to be approximations to zero, then the conserved projection operator is a perturbation of $Q_{u}(f,f)$ by a second order polynomial in the elastic case.   Similarly, it is a perturbation by a first order polynomial in the inelastic case.  
%
%
%
%
\begin{thm}[Conservation Correction Estimate/Elastic case]\label{t1}
Fix $f\in L^2(\Omega_L)$, then the accuracy of the conservation minimization problem is proportional to the spectral accuracy.  That is, for any $k'\geq k\geq0$ it follows that
\begin{align}\label{conservationEstimate}
\begin{split}
\big\| \big(Q_c(f,f) - Q_u(f,f)\big)|v|^{\lambda k} & \big\|_{L^{2}(\Omega_L)} \leq  \frac{C\,L^{\lambda k}}{(2\lambda k+d)^{1/2}} \big\|(\textbf{1} - \Pi^{N}_{2L})Q^{+}(\chi f, \chi f)\big\|_{L^2(\Omega_L)} \\
&\hspace{-0.3cm} + \frac{1}{(2\lambda k+d)^{1/2}}O_{(d/2+\lambda (k'-k))}\big(m_{k'+1}(f)m_{0}(f)+Z_{k'}(f)\big)\,,
\end{split}
\end{align}
where $C$ is a universal constant and $Z_{k'}(f)$ is defined by 
\begin{equation}\label{Zk0}
Z_{k'}(f):=\sum^{k'-1}_{j=0}
\left(\begin{array}{c}
 k'\\j
\end{array}\right)
m_{j+1}(f)m_{k'-j}(f).
\end{equation}
depending on the moments up to order $k'$ ( See also Appendix \eqref{Zk}).  As before, we are using the shorthand $O_{r}:=O(L^{-r})$.
\end{thm}
%
%
\begin{proof}
Using lemma \ref{l1} for elastic interactions, given a $0\leq k\in \R$, estimate
\begin{align}\label{after}
\begin{split}
\Big\| \left(Q_c(f,f) - Q_u(f,f)\right)|v|^{\lambda k} &\Big\|_{L^{2}(\Omega_L)} = \Big\| \tfrac{1}{2} \Big(\gamma_1 + \sum_{j=1}^d \gamma_{j+1} v_{j} + \gamma_{d+2} |v|^2\Big)|v|^{\lambda k}\Big\| _{L^{2}(\Omega_L)} \\
&\leq \frac{C\,L^{\lambda k}}{(2\lambda k+d)^{1/2}}\left( |\gamma_1| L^{d/2}+ |\gamma_j| L^{1+d/2}+|\gamma_{d+2}| L^{2+d/2}\right)\,.
\end{split}
\end{align}
For any $f \in L^{2}(\Omega_L)$ the Lagrange multipliers $\gamma_j$, $1\leq j\leq d+2$ can be estimated by observing that
\begin{align}\label{e1t1}
\begin{split}
\Big|\int_{\Omega_L}Q_{u}(f,f)(v)\psi(v)&\text{d}v  \Big| = \\
\Big|\int_{\Omega_L}\big(Q_{u}(f,f)(v) - Q(\chi f ,\chi f)(v)&\big)\psi(v)\text{d}v - \int_{\mathbb{R}^{d}\setminus\Omega_L}Q(\chi f ,\chi f)(v)\psi(v)\text{d}v\Big|\\
\leq \big\|(\textbf{1} - \Pi^{N}_{2L})&Q^{+}(\chi f,\chi f)\big\|_{L^2(\Omega_L)} \| \psi \|_{L^2(\Omega_L)}+I_{\psi}.
\end{split}
\end{align}
for $I_{\psi}$ defined by 
\begin{equation}\label{Ipsi}
I_{\psi}:= \Big|\int_{\mathbb{R}^{d}\setminus\Omega_L}Q^{+}(\chi f,\chi f)(v)\psi(v)\text{d}v - \int_{\mathbb{R}^{d}\setminus 0.95\Omega_L} Q^{-}(f(1-\chi),f)\psi(v)\text{d}v\Big|.
\end{equation}
Since
\begin{align}\label{e2t1}
\|  \,1\, \|_{L^{2}(\Omega_L)}& \sim L^{d/2} \, , \nonumber \\
\|  \,v_j\, \|_{L^{2}(\Omega_L)}&\sim L^{d/2+1}, \quad \text{for}\ \ j = 1, 2, 3, . . . , d \, ,  \\
\|\,  |v|^2 \|_{L^{2}(\Omega_L)}&\sim L^{d/{2} + 2} \, ,\nonumber
\end{align}
then,  for  $\psi =1, v^j, |v|^2$ with $j=1,2,\dots d$  the corresponding estimate 
 \eqref{e1t1}  combined with  \eqref{e2t1}  yield the following estimates to the unconserved moments defined in \eqref{constraintEqns}  
\begin{align}\label{finalConsError}
| \rho_u | & \leq   CL^{d/2}\big\| (\textbf{1}-\Pi^{N}_{2L})Q^{+}(\chi f, \chi f) \big\|_{L^2(\Omega_L)} + I_1\, , \nonumber \\
|\mu^j_u| & \leq   CL^{d/2+1} \big\| (\textbf{1}-\Pi^{N}_{2L})Q^{+}(\chi f, \chi f) \big\|_{L^2(\Omega_L)} + I_{v_j},  \quad j = 1, 2, 3,  . . . d \, ,  \\
|e_u| & \leq   CL^{d/2+2} \big\| (\textbf{1}-\Pi^{N}_{2L})Q^{+}(\chi f, \chi f) \big\|_{L^2B(\Omega_L)} + I_{|v|^{2}}\, .\nonumber
\end{align}
Therefore, using \eqref{finalConsError} in \eqref{criticalLambda}  Lagrange multipliers are estimated by 
\begin{align}\label{finalerror2}
|\gamma_1|= O_{d/2}&\big\| (\textbf{1}-\Pi^{N}_{2L})Q^{+}(\chi f, \chi f) \big\|_{L^2(\Omega_L)} +O_{d}I_1+O_{d+2}I_{|v|^2},\nonumber\\
|\gamma_j|= O_{d/2+1}&\big\| (\textbf{1}-\Pi^{N}_{2L})Q^{+}(\chi f, \chi f) \big\|_{L^2(\Omega_L)} +O_{d+2}I_{v_j},\quad j = 1, 2, 3,  . . . d \, ,\\
|\gamma_{d+2}|= O_{d/2+2}&\big\| (\textbf{1}-\Pi^{N}_{2L})Q^{+}(\chi f, \chi f) \big\|_{L^2(\Omega_L)} +O_{d+2}I_{1}+O_{d+4}I_{|v|^2}.\nonumber
\end{align}
Finally,  the Lagrangian critical parameters from \eqref{after}  are estimated  by \eqref{finalerror2} to  yield
\begin{align*}
\big\|\big(Q_c(f,f)- Q_{u}(f,f)\big)|v|^{\lambda k} \big\|_{L^2({\Omega_L})}&=\frac{C}{(2\lambda k+d)^{1/2}}\Big(L^{\lambda k} \big\| (\textbf{1}-\Pi^{N}_{2L})Q^{+}(\chi f,\chi f)\big\|^2_{L^2(\Omega_L)} \\
&\hspace{-.5cm} +  O_{d/2-\lambda k}\,I_{1} + O_{d/2+1-\lambda k}\,I_{v_j} + O_{d/2+2-\lambda k}\,I_{|v|^{2}} \Big).
\end{align*}
In order to estimate the second term in the above inequality, the terms $I_{\psi}$ defined in \eqref{Ipsi} are estimated combining classical moment estimates for binary collisional integrals for elastic interactions with hard potentials as shown in Theorem \ref{moments} in the appendix.  In particular,  for any $k'\geq0$ and $\lambda \in [0,2]$ 
\begin{align*}
\max\big\{I_{1}, L^{-1}I_{v_j}, L^{-2}I_{|v| ^{2}}\big\}&\leq CL^{-\lambda k'}\big(m_{k'+1}(\chi f)\;m_{0}(\chi f)+Z_{k'}(\chi f)\big)\\
&\leq CL^{-\lambda k'}\big(m_{k'+1}(f)\;m_{0}(f)+Z_{k'}(f)\big).
\end{align*}
Therefore, a simple calculation shows
\begin{equation*}
O_{d/2-\lambda k}I_{1} + O_{d/2+1-\lambda k}I_{v_j} + O_{d/2+2-\lambda k}I_{|v|^{2}} = O_{d/2+\lambda(k'-k)}\big(m_{k'+1}(f)m_0(f)+Z_{k'}(f)\big), 
\end{equation*}
and so inequality \eqref{after} holds. 

This estimate also follows for the  \textit{Inelastic collisions} case. Their computations follow in a similar fashion using lemma \ref{l2}, the Lagrange multipliers \eqref{inelasticCriticalLambda} and the first two inequalities in \eqref{finalConsError}.
\end{proof}

\subsection{Semi Discrete Conservation Method: Lagrange Multiplier Method}
\label{subsec:consDIT}
In this subsection we consider the discrete version of the conservation scheme. For such a discrete formulation, the conservation routine is implemented as a Lagrange multiplier method where the conservation properties of the discrete distribution are set as constraints. Let $M = N^d$, the total number of Fourier modes. For elastic collisions, $\rho = 0,\, \textbf{m} = (m_1, \cdots, m_d) = (0, \cdots, 0)$ and $e = 0$ are conserved, whereas for inelastic collisions, $\rho = 0$ and $\textbf{m} = (m_1, \cdots , m_d) = (0, \cdots,0)$ are conserved. Let $\omega_j > 0$ be the integration weights for $1\leq j \leq M$ and define
\begin{equation}\label{sdcm1} \textbf{Q}_{u} = \Big( \begin{array}{ccccc}
Q_{u,1} & Q_{u,2} & \cdots & Q_{u,M} \end{array} \Big)^T 
\end{equation}
as the distribution vector at the computed time step, and
\begin{equation}\label{sdcm2}  \textbf{Q}_{c} = \Big( \begin{array}{ccccc}
Q_{c,1} & Q_{c,2} & \cdots & Q_{c,M} \\ \end{array} \Big)^T 
\end{equation}
as the corrected distribution vector with the required moments conserved. For the elastic case, let
\begin{equation} \label{sdcm3} 
 \textbf{C}^e_{_{(d+2) \times M}} =  \left( \begin{array}{ccc}
 & \omega_j & \\
 & v_1\, \omega_j & \\
 & \cdots & \\
 & v_d\, \omega_j &\\
 & |v_j|^2\,\omega_j &  \\
\end{array} \right)
\quad 1 \leq j \leq M\,,
\end{equation} 
be the integration matrix where the $w_j, j=1\dots M$ are fixed set of quadrature points, and
\begin{equation} \label{sdcm4} 
\textbf{a}^e_{_{(d+2) \times 1}} =  \Big( \begin{array}{ccccc}
\frac{\text{d}}{\text{d}t}\rho & \frac{\text{d}}{\text{d}t}m_1 &\cdots & \frac{\text{d}}{\text{d}t}m_d & \frac{\text{d}}{\text{d}t}e \\
\end{array} \Big)^T  
\end{equation} 
{be the vector of conserved quantities. With this notation in mind, the semi discrete conservation method corresponding to (\ref{elasticLagrangeProblem0}, \ref{elasticLagrangeProblem}) is written as the constrained optimization problem}

\begin{align}\label{min-problem} 
\begin{split}
&\qquad\qquad \text{\it Find  the vector } \ \ \textbf{Q}_{c}  \in  \real^M\,  , \ \ \text{\it such that it is the unique solutions of } \\ 
&\mathcal{A}\big(\textbf{Q}_{c}\big) = \Big\{ \text{min}\, \| \textbf{Q}_{u} - \textbf{Q}_{c} \|_2^2 : \textbf{C}^e \textbf{Q}_{c} = \textbf{a}^e\; \text{with}\; \textbf{C}^e \in \mathbb{R}^{{d+2} \times M},\, \textbf{Q}_{u} \in \mathbb{R}^M,\, \textbf{a}^e \in \mathbb{R}^{d+2} \Big\}.
\end{split}
\end{align}
\\

\n In order to solve the constrained minimization problem $\mathcal{A}\big(\textbf{Q}_{c}\big)$, we employ the Lagrange multiplier method proposed by two of the authors \cite{GT09} in 2009.  The proposed algorithm works as follows. 

Let $\boldsymbol{\gamma} \in \mathbb{R}^{d+2}$ be the Lagrange multiplier vector. Then the scalar objective function to be optimized is given by
\begin{equation}
\text{L}\big(\textbf{Q}_{c}, \boldsymbol{\gamma}\big) = \sum_{j = 1}^M \big|Q_{u,j} - Q_{c,j}\big|^2 + \boldsymbol{\gamma}^T(\textbf{C}^e \textbf{Q}_{c} - \textbf{a}^e) \, ,
\label{lagrange}
\end{equation}
where $\textbf{C}^e$ is given by the integration matrix that computes the number of collision invariants associated to the conservation problem (i.e. $d+2$ for the elastic case or $d+1$ for the inelastic one). This matrix is independent of the solution and the time parameter.    Hence, it can be precomputed and use for different initial data and time steps.

Equation \eqref{lagrange} can be solved explicitly for the corrected distribution value and the resulting equation of correction be implemented numerically in the code.  Indeed, taking the derivative of $\text{L}\big(\textbf{Q}_{c}, \boldsymbol{\gamma}\big)$ with respect to $Q_{c,j}$, for $1\leq j \leq M$ and $\gamma_i$, for $1\leq i \leq d+2$
\begin{equation}\label{lambEq}
\frac{\partial \text{L}}{\partial Q_{c,j}}  =  0\,, \quad j = 1,\cdots, M \quad\Rightarrow\quad 
\textbf{Q}_{c} = \textbf{Q}_{u} + \frac{1}{2} (\textbf{C}^e)^T \boldsymbol{\gamma} \, .
\end{equation}
Moreover,
\begin{equation*}
\frac{\partial \text{L}}{\partial \gamma_i} =  \,, \quad i = 1, \cdots, d+2 \quad \Rightarrow \quad \textbf{C}^e \textbf{Q}_{c} = \textbf{a}^e \, ,
\end{equation*}
retrieves the constraints. Solving for $\boldsymbol{\gamma}$,
\begin{equation}
\textbf{C}^e (\textbf{C}^e)^T \boldsymbol{\gamma} = 2 (\textbf{a}^e - \textbf{C}^e \textbf{Q}_{u}) \, .
\end{equation}
Now $\textbf{C}^e (\textbf{C}^e)^T$ is symmetric and, because $\textbf{C}^e$ is an integration matrix, it is also positive definite. As a consequence, the inverse of $\textbf{C}^e (\textbf{C}^e)^T$ exists and one can compute the value of $\boldsymbol{\gamma}$ simply by
\begin{equation}
\boldsymbol{\gamma} = 2 (\textbf{C}^e (\textbf{C}^e)^T)^{-1} (\textbf{a}^e - \textbf{C}^e \textbf{Q}_{u}) \, .
\nonumber
\end{equation}
Substituting $\boldsymbol{\gamma}$ into \eqref{lambEq} and recalling that $\textbf{a}^e = \textbf{0}$,
\begin{eqnarray}
\textbf{Q}_{c} & = & \textbf{Q}_{u} + (\textbf{C}^e)^T \big(\textbf{C}^e (\textbf{C}^e)^T\big)^{-1} (\textbf{a}^e - \textbf{C}^e \textbf{Q}_{u}) \nonumber \\
& = & \left[\mathbb{I} - (\textbf{C}^e)^T \big(\textbf{C}^e (\textbf{C}^e)^T\big)^{-1} \textbf{C}^e \right] \textbf{Q}_{u} \nonumber \\
& =: & \Lambda_N(\textbf{C}^e)\, \textbf{Q}_{u} \, ,
\label{fEqCons}
\end{eqnarray}
where $\mathbb{I} = N \times N$ identity matrix.  In the sequel, we regard this conservation routine as \emph{Conserve}.  Thus,
\begin{equation}
Conserve(\textbf{Q}_{u}) = \textbf{Q}_{c} = \Lambda_N(\textbf{C}^e)\; \textbf{Q}_{u} \, .
\label{conservationRoutine}
\end{equation}
Define $D_t $ to be any time discretization operator of arbitrary order. Then, the discrete problem that we solve reads
\begin{equation}
D_t \textbf{f} = \Lambda_N(\textbf{C}^e)\,\textbf{Q}_{u} \, .
\label{completeScheme1}
\end{equation}
Thus, multiplying \eqref{completeScheme1} by $\textbf{C}^e$ it follows the conservation of observables
\begin{equation}
D_t \big( \textbf{C}^e \textbf{f} \big) = \textbf{C}^e D_t \textbf{f} = \textbf{C}^e \Lambda_N(\textbf{C}^e)\,\textbf{Q}_{u} = 0\,,
\label{conservationobsevables}
\end{equation}
where we used the commutation $\textbf{C}^e D_{t} = D_{t} \textbf{C}^e$ valid since $\textbf{C}^e$ is independent of time, see \cite{GT09} for additional comments.
\section{A priori estimates, propagation of moments and $L^{2}_{k}$-norm}\label{sec:apriori}

In this section we prove $L^1_{k}$ and $L^2_{k}$ estimates for the approximation solutions $\{g_N\}$ of the problem \eqref{fourierBTE} in the elastic case.  For this purpose, we use several well known results that require different integrability properties for the angular kernel $b$.  Thus, we will work with a bounded $b$ to avoid as much technicalities as possible and remarking that a generalization for $b\in L^{1}(\mathbb{S}^{d-1}) $ can be made at the cost of technical work \cite{ACGM, AG, MV04}.  For technical reasons this assumption helps since estimates for the gain part  of the collision operator become bilinear, that is, the role of the inputs can be interchanged without essentially altering the constants in the estimates.  We also restrict ourselves to the case of variable hard potentials and hard spheres $\lambda\in(0,1]$ and remark than the theory for Maxwell molecules $\lambda=0$ needs a different approach.

Recall that we have imposed conservation of mass, momentum and energy by building the operator $Q_c(g,g)$ with a constrained minimization procedure.  Thus,
\begin{equation*}
\int_{\Omega_L}g(t,v)\psi(v)\text{d}v=\int_{\Omega_L}g_0(v)\psi(v)\text{d}v
\end{equation*}
for any collision invariant $\psi(v)=\{1,v,|v|^{2}\}$.  However, due to velocity-mode truncation, the approximating solution $g$ in general may be negative in some small portions of the domain.  This is one of the important technical difficulty that we have to overcome.

Before starting with the calculations recall the smoothing property of the gain collision operator $Q^{+}$ given in \cite[Theorem 2.1]{BD}
\begin{equation}\label{C-universal}
\| Q^{+}(f,f) \|_{\dot{H}^{(d-1)/2}(\mathbb{R}^{d})}\leq C\|b\|_{L^{2}(\mathbb{S}^{d-1})}\|f\|^{2}_{L^{2}_{1+\lambda^{-1}}(\mathbb{R}^{d})}\,,
\end{equation}
{where $C$ is a universal constant only depending on the space dimension $d$.}

Therefore, recalling that $\text{supp}(Q^{+}(\chi g,\chi g))\subset \Omega_{2L}$ and using Parseval's theorem, it follows that (for $a>0$)
\begin{align*}
&\big\|(\textbf{1}-\Pi^{N}_{2L})Q^{+}(\chi g,\chi g)\big\|^{2}_{L^{2}(\Omega_{2L})}\\
&\ \ \ =\sum_{|k|\geq N}\big|\widehat{Q^{+}(\chi g,\chi g)}(\xi_{k})\big|^{2}=\sum_{|k|\geq N}\frac{1}{|\xi_{k}|^{2a}}\big|\widehat{(-\Delta)^{a/2}Q^{+}(\chi g,\chi g)}(\xi_{k})\big|^{2}\\
&\lesssim \frac{1}{N^{2a}}\sum_{|k|\geq N}\big|\widehat{(-\Delta)^{a/2}Q^{+}(\chi g,\chi g)}(\xi_{k})\big|^{2} \leq \frac{1}{N^{2a}}\big\|(-\Delta)^{a/2}Q^{+}(\chi g,\chi g) \big\|_{L^{2}(\Omega_{2L})}\,.
\end{align*}
As a conclusion of previous two facts, choosing $a=\frac{d-1}{2}$, we obtain an important estimate used in the following arguments
\begin{equation}\label{IPE}
\big\|(\textbf{1}-\Pi^{N}_{2L})Q^{+}(\chi g,\chi g)\big\|_{L^{2}(\Omega_{2L})} \leq \frac{C}{N^{\frac{d-1}{2}}}\|\chi g\|^{2}_{L^{2}_{1+\lambda^{-1}}(\Omega_{L})}\,,
\end{equation}
since  $\chi g$ vanishes outside a compactly supported set in $\Omega_L$, so we make use of the extension norm identity \eqref{extension-norm} that asserts $ \|\chi g\|_{L_{1+\lambda^{-1}}(\Omega_{2L})} = \|\chi g\|_{L_{1+\lambda^{-1}}(\Omega_{L})}$\, .
\subsection{Differential estimates for moments of the scheme}
In the analysis of the following two sections, we assume that a {semi-discrete } solution $g\in\mathcal{C}\big(0,T;L^{2}(\Omega_L)\big)$ for problem \eqref{fourierBTE} with initial condition $g_0\in L^{2}(\Omega_L)$ exists {satisfying condition \eqref{gas0}}.  We denote $T_{\epsilon}\geq0$ any time such that the smallness relation for the negative mass and energy of $g(t,v)$ and its boundedness in $L^{2}$ holds
\begin{equation}\label{gas}
 \sup_{t\in[0,T_\epsilon]}\Bigg(\epsilon(t):=\frac{\int_{\{g<0\}}|g(t,v)|\langle v \rangle^{2}\text{d}v}{\int_{\{g\geq0\}}g(t,v) \langle v \rangle^{2}\text{d}v}\Bigg)\leq\epsilon,\qquad \sup_{t\in[0,T_{\epsilon}]}\|g(t,\cdot)\|_{L^{2}(\Omega_L)}<\infty\,,
\end{equation}
for some fixed $\epsilon>0$ {sufficiently small to be specified below in \eqref{conep}}.   Observe that the \textit{conservation scheme} and this assumption implies that semi-discrete moments up to order 2 are controlled by the initial datum.  Indeed, for $k=\{0,2\}$
\begin{align*}
\int_{\Omega_L}|g||v|^{k}=\int_{\Omega_L}g|v|^{k} - & 2\int_{\Omega_L}g^{-}|v|^{k}=\int_{\Omega_L}g_0|v|^{k}-2\int_{\Omega_L}g^{-}|v|^{k}\\
\leq & \int_{\Omega_L}g_0|v|^{k}+2\epsilon\int_{\Omega_L}g^{+}|v|^{k}\leq \int_{\Omega_L}g_0|v|^{k}+2\epsilon\int_{\Omega_L}|g||v|^{k}.
\end{align*}
Indeed, choosing $\epsilon\leq1/4$ one obtains,
\begin{equation}\label{uc}
\int_{\Omega_L}|g(t,v)||v|^{k}\text{d}v\leq 2\int_{\Omega_L}g_0|v|^{k}\text{d}v,\quad \text{for}\;t\in[0,T_{\epsilon}]\,,\quad k=1,2.
\end{equation}

{{\emph Remark:}  Conditions \eqref{gas0} and \eqref{gas} are a sort of stability condition for the semi-discrete scheme.}

Next, we start getting estimates for the discrete conserved form \eqref{fourierBTE}.   Indeed, taking  right hand side  from \eqref{fourierBTE-0} combined with those of (\ref{fourierBTE00},\ref{fourierBTE0}), the discrete equation \eqref{fourierBTE}  for the numerical scheme  can be written in $(0,T_{\varepsilon}]\times\Omega_{L}$ as
\begin{equation}\label{basic-equation}\begin{split}
\frac{\text{d}g}{\text{d}t} &= Q_{c}(g,g)   = Q_{c}(g,g) \\
&- Q_{u}(g,g) + Q(\chi g,\chi g) \!-\! (\textbf{1} \!-\! \Pi^{N}_{2L})Q^{+}(\chi g,\chi g) \!-\! Q^{-}(1-\chi)g,\chi g)\; ,
\end{split}\end{equation} 
as the second term in this equation is actually  null.

In the next Lemma we prepare estimates to obtain an ordinary differential inequality that will yield uniform estimates    to the numerical moments  to the semi discrete solutions corresponding to the  initial value problem \eqref{fourierBTE}. 
\begin{lem}\label{L1prop}
Let $g$ solution of the numerical scheme satisfying \eqref{gas} and set $k\geq k_0\geq 2$.  Then,
\begin{align}\label{moment_ode}
\frac{\text{d}}{\text{d} t}m_k(g) \leq C_{k}\,\big(m_{0}(g_0) + m_{k}(g) \big) - \frac{\mu_{ \frac{\lambda}{2} }\,m_0(g_0)}{4}\,m_{k+1}(g) + C\,\frac{L^{\lambda k+d/2}}{N^{\frac{d-1}{2}}}\|g\|^{2}_{L^{2}_{1+\lambda^{-1}}(\Omega_{L})}\,,
\end{align}
{for any $g_0(v)$ satisfying the energy ratio condition \eqref{gas0}. In addition, $\mu_{\frac{\lambda}{2}}$, $k_{0}$ are constants given by  by \eqref{mu_k} and in \eqref{k_0}, respectively,  defined in the proof of this lemma.}
\end{lem}
\begin{proof}
We fix $k>0$ and $L>0$ and keep in mind that $g_0$ has support in $\Omega_L$, and thus, possesses moments of any order.  Multiply equation \eqref{basic-equation} by $\text{sgn}(g)|v|^{\lambda k}$ and integrate in $\Omega_L$
\begin{multline*}
\frac{\text{d}}{\text{d} t}\int_{\Omega_L}\big|g(v)\big||v|^{\lambda k}\text{d}v\\
=\int_{\Omega_L}Q(\chi g,\chi g)(v)\,\text{sgn}(g)(v)\,|v|^{\lambda k}\text{d}v - \int_{\Omega_L}Q^{-}((1-\chi)g,\chi g)(v)\,\text{sgn}(g)(v)\,|v|^{\lambda k}\text{d}v \\
+\int_{\Omega_L}\big(Q_{c}(g,g)(v)-Q_{u}(g,g)(v)\big)\,\text{sgn}(g)\,|v|^{\lambda k}\text{d}v - \int_{\Omega_{L}}(\textbf{1} - \Pi^{N}_{2L})Q^{+}(\chi g,\chi g)(v)\,\text{sgn}(g)\,|v|^{\lambda k}\\
\leq \int_{\Omega_L}Q^{+}(|\chi g|,|\chi g|)(v)|v|^{\lambda k}\text{d}v -\int_{\Omega_L}Q^{-}(g,\chi g)(v)\text{sgn}(\chi g)(v)|v|^{\lambda k}\text{d}v\\
+\big\|\big(Q_c(g,g)-Q_{u}(g, g)\big)|v|^{\lambda k}\big\|_{L^{1}(\Omega_L)} + \big\|(\textbf{1} - \Pi^{N}_{2L})Q^{+}(\chi g,\chi g)|v|^{\lambda k}\big\|_{L^{1}(\Omega_L)} .
\end{multline*}

We estimate each term starting with the loss collision operator.  Use $g = | g | - 2g^{-}$ to conclude that
\begin{align*}
\int_{\Omega_L}Q^{-}( g,\chi g)(v)\text{sgn}(g)(v)|v|^{\lambda k}\text{d}v\geq\int_{\Omega_L}&|g(v)||v|^{\lambda k}\int_{\mathbb{R}^{d}}|\chi g(v_*)||v-v_*|^\lambda \text{d}v_{*}\text{d}v\\
& - C_{d,\lambda}\,\epsilon \,\|g_0\|_{L^{1}_{2\lambda^{-1}}(\Omega_L)}\big(m_{k+1}(g)+m_{k}(g)\big)\, , 
\end{align*}
{where $\epsilon$ is the bound from the energy quotient from \eqref{gas}.} Whence, 
\begin{align}\label{e0L1prop}
\begin{split}
\int_{\Omega_L}Q^{+}(|\chi g|,|\chi g|)&(v)|v|^{\lambda k}\text{d}v -\int_{\Omega_L}Q^{-}(g,\chi g)(v)\text{sgn}(g)(v)|v|^{\lambda k}\text{d}v\\
& \leq \int_{\Omega_L}Q(|\chi g|,|\chi g|)(v)|v|^{\lambda k}\text{d}v\\
&\hspace{-2cm}- \int_{\Omega_L}Q((1 - \chi) |g|,|\chi g|)(v)|v|^{\lambda k}\text{d}v + {C_{d,\lambda}}\,\epsilon\,\|g_0\|_{L^{1}_{2\lambda^{-1}}(\Omega_L)}\big(m_{k+1}(g)+m_{k}(g)\big)\,.
\end{split}
\end{align}

\medskip

Using the \textit{conservative property of the scheme} it follows from the discussion in \cite{BGP04,ALod}
\begin{multline*}
\int_{\Omega_L}Q(|\chi g|,|\chi g|)(v)|v|^{\lambda k}\text{d}v\\
\leq \int_{\mathbb{R}^{d}}Q(|\chi g|,|\chi g|)(v)|v|^{\lambda k}\text{d}v \leq Z_{k}(g)-\mu_{k}\,m_0(g_0)\,m_{k+1}(\chi g), \quad \tfrac{2}{\lambda}<k\in\mathbb{Z},
\end{multline*}
where $Z_{k}(g)$ depends on the moments of $g$ of order \textit{less or equal} than $k$ and $\mu_{k}\nearrow1$ as $k\rightarrow\infty$ being a universal parameter given by
\begin{equation}\label{mu_k}
\mu_{k} := 1 - \frac{1}{2^k}\int_{\mathbb{S}^{d-1}}\left({1+\hat{u}\cdot\sigma}\right)^{k} b(\hat{u}\cdot\sigma)\, \text{d}\sigma\in (0,1)\,.
\end{equation}
We refer to \cite[Lemma 3]{BGP04} for details and proof.  Choose
{\begin{equation}\label{conep}
\epsilon\leq \min\Big{\{}\tfrac 14\,, \, \frac{\mu_{\frac{ \lambda }{ 2 }}\,m_0(g_0)}{2\, C_{d,\lambda} \,\|g_0\|_{L^{1}_{2\lambda^{-1}}(\Omega_L)} }\Big{\}}
\end{equation}}
in \eqref{e0L1prop} to conclude that,
\begin{align}\label{e1L1prop}
\begin{split}
\frac{\text{d}}{\text{d} t}m_k(g) \leq Z_{k}(g)  - \frac12 &{\mu_{ \frac{\lambda}{2} }\,m_0(g_0)}\, m_{k+1}(g)  + \big\|(Q_c(g,g)-Q_{u}(g,g))|v|^{\lambda k}\big\|_{L^{1}(\Omega_L)}\\
& + \big\|(\textbf{1} - \Pi^{N}_{2L})Q^{+}(\chi g,\chi g)|v|^{\lambda k}\big\|_{L^{1}(\Omega_L)}\,.
\end{split}
\end{align}
Using Cauchy-Schwarz inequality and   \eqref{conservationEstimate} from  Theorem \ref{t1},   it follows, for any $k'\geq k\geq0$, that 
\begin{align*}
\big\|(Q_c(g,g)&- Q_{u}(g, g))|v|^{\lambda k}\big\|_{L^{1}(\Omega_L)}\leq L^{d/2}\big\|(Q_c(g,g)-Q_{u}(g, g))|v|^{\lambda k}\big\|_{L^{2}(\Omega_L)}\\
& \leq \frac{C\,L^{\lambda k+d/2}}{(2\lambda k+d)^{1/2}}  \big\|(\textbf{1}-\Pi^{N}_{2L})Q^{+}(\chi g,\chi g)\big\|_{L^{2}(\Omega_L)}\\
&\hspace{2cm} + \frac{ O(L^{-\lambda(k'-k)})}{(2\lambda k+d)^{1/2}}  \big(m_{k'+1}(g)\;m_{0}(g_0)+Z_{k'}(g)\big)\,.
\end{align*}
Therefore, after choosing $k'=k>2$, one concludes that
\begin{align}\label{import}
&\frac{\text{d}}{\text{d} t}m_k(g)\leq 2\,Z_{k}(g) - \Bigg(\frac12\, \mu_{ \frac{\lambda}2 }\,m_0(g_0)-{\frac{\bf C}{ (2\lambda k+d )^{1/2}} }
\Bigg)\,m_{k+1}(g)\nonumber\\
&\ \ \ +CL^{\lambda k+d/2}\big\|(\textbf{1}-\Pi^{N}_{2L})Q^{+}(\chi g,\chi g)\big\|_{L^{2}(\Omega_L)} \leq C_{k}\,\big(m_{0}(g_0) + m_{k}(g) \big) \\
&\hspace{1cm} - \frac14{\mu_{ \frac{\lambda}{2} }\,m_0(g_0)}\,m_{k+1}(g)+CL^{\lambda k+d/2}\big\|(\textbf{1}-\Pi^{N}_{2L})Q^{+}(\chi g,\chi g)\big\|_{L^{2}(\Omega_L)}\,, \nonumber
\end{align}
{where ${\bf C}$ is a constant independent of $k$ and $\lambda$. }

\medskip
  
In the last inequality we used the classical fact that $Z_{k} \leq C_{k}\,\big(m_{0}(g) + m_{k}(g)\big)$ for some large constant $C_{k}$ depending only on $k$.  We also chose $k$ sufficiently large to make the largest moment an absorption term,
\begin{equation}\label{k_0}
k\geq k_0:={\frac{1}{2\lambda} \bigg(\frac{ \bf C}{\mu_{\frac{\lambda}{2}} m_0(g_0)}\bigg)^{2} - \frac{d}{2\lambda}\, \ge 2.}
\end{equation} 
{Finally, use estimate \eqref{IPE} in \eqref{import} to obtain the semi-discrete moment ordinary differential inequality \eqref{moment_ode}. }
\end{proof}
\begin{lem}\label{mainlemma} {[Lower bound estimate]}
Let $h(v)$ be a function satisfying \eqref{gas} for $\epsilon<1/2$.  Assume also that $\int_{\mathbb{R}^{d}} h(w)\,w\,\text{d}w=0$, and that
\begin{equation}\label{m_mu}
m_{\mu}:=\int_{\mathbb{R}^{d}}|h(w)||w|^{2+\mu}\text{d}w <\infty\,,\quad \mu>0.
\end{equation}
Then,
\begin{equation}\label{gas1}
\left(h\ast|u|^{\lambda}\right)(v)\geq \frac{C(h)\,\langle v \rangle^{\lambda}}{\max\big\{1,m^{(2-\lambda)/\mu}_{\mu}\big\}},
\end{equation}
with  ${C(h)}>0$ depending only on the mass, energy of $h$.  
\end{lem}
\begin{proof}
Notice that in the ball $B(0,r)$ one has for any $R>0$ and $\mu>0$,
\begin{align}\label{lbe1}
\begin{split}
\int_{|v-w|\leq R}\!\!h(w)|v-w|^2\text{d}w&\!=\!\int_{\mathbb{R}^d}\!\!h(w)|v-w|^2\text{d}w-\int_{|v-w|\geq R}h(w)|v-w|^2\text{d}w\\
&\geq {C(h)}\, \langle v \rangle^2-\frac{1}{R^\mu}\int_{|v-w|\geq R}|h(w)||v-w|^{2+\mu}\text{d}w\,.
\end{split}
\end{align} 
For the last inequality we expanded the square in the integral of the right side and used the fact that the momentum of $g$ is zero.  We use in the right side integral of \eqref{lbe1} the inequality $|v-w| \leq \langle v \rangle \langle w \rangle$ and the fact that $m_{\mu}<\infty$ to obtain
\begin{equation*}
\int_{|v-w|\leq R}h(w)|v-w|^2\text{d}w\geq {C(h)} \langle v \rangle^2-\frac{m_{\mu}}{R^\mu}\langle v\rangle^{2+\mu}\geq \frac{C(h)}{2}\langle v \rangle ^{2}\,, \quad \forall\; v \in B(0,r)\,,
\end{equation*}
provided 
{\begin{equation}\label{R_mu}
R:=\big(2m_{\mu}/C(h)\big)^{1/\mu}r. 
\end{equation} }
Therefore, using the control \eqref{gas}
\begin{align*}
\int_{\mathbb{R}^d}h(w)|v-w|^{\lambda}\text{d}w&=\int_{\mathbb{R}^d}|h(w)||v-w|^{\lambda}\text{d}w-2\int_{\{h<0\}}|h(w)||v-w|^{\lambda}\text{d}w\\
\geq (1-2\epsilon)\int_{\mathbb{R}^d}&|h(w)||v-w|^{\lambda}\text{d}w
\geq (1-2\epsilon)\int_{|v-w|\leq R}|h(w)||v-w|^{\lambda}\text{d}w\\
\geq&\frac{1-2\epsilon}{R^{2-\lambda}}\int_{|v-w|\leq R}|h(w)||v-w|^{2}\text{d}w\geq\frac{1-2\epsilon}{2R^{2-\lambda}}\,{C(h)}\langle v \rangle^{2}\,,
\end{align*}
valid for any $v\in B(0,r)$ and provided $\epsilon<\frac{1}{2}$.  Moreover, for any $\lambda\in(0,1]$
\begin{align*}
\int_{\mathbb{R}^d}h(w)|v-w|^{\lambda}\text{d}w &\geq (1-2\epsilon)\int_{\mathbb{R}^d}|h(w)||v-w|^{\lambda}\text{d}w\\
&\geq (1-2\epsilon)\|h\|_{L^{1}_{2\lambda^{-1}}}\big(|v|^\lambda - 2\big),
\end{align*}
as a consequence,
\begin{equation}\label{gas2}
\int_{\mathbb{R}^d}\!\!h(w,t)|v-w|^{\lambda}\!\text{d}w\!\geq\!(1-2\epsilon)\left(\frac{C(h)}{2R^{2-\lambda}}\,\textbf{1}_{B(0,r)}+\|h\|_{L^{1}_{2\lambda^{-1}}}\big(|v|^\lambda - 2\big)\,\textbf{1}_{B(0,r)^c} \right).
\end{equation}
Inequality \eqref{gas1} follows from \eqref{gas2} choosing $r=3^{1/\lambda}$ in the definition \eqref{R_mu} of $R$.
\end{proof}
\subsection{Time differential estimates for the $L^{2}_k$-norm of the conservative semi-discrete scheme}
The lower bound on the collision operator given in Lemma \ref{mainlemma} will allow us to control the $L^{2}_{k}$-norms of $g$.  Multiplying equation \eqref{basic-equation} by $g\langle v \rangle^{2\lambda k}$ and integrating on $\Omega_L$ on has
\begin{multline*}
\frac{1}{2}\frac{\text{d}}{\text{d} t}\|g\|^{2}_{L^{2}_{k}(\Omega_L)}
=\int_{\Omega_L}\langle v \rangle^{2\lambda k}g\;Q^{+}(\chi g,\chi g)\text{d}v - \int_{\Omega_L}\langle v \rangle^{2\lambda k}g\;Q^{-}( g,\chi g)\text{d}v\\
+ \int_{\Omega_L}\langle v \rangle^{2\lambda k}g\;\big(Q_{c}(g,g)-Q_{u}(g,g)\big)\text{d}v - \int_{\Omega_L}\langle v \rangle^{2\lambda k}g\;\big(\textbf{1}-\Pi^{N}_{2L}\big)Q^{+}(\chi g,\chi g)\text{d}v\\
\leq \int_{\Omega_L}\langle v \rangle^{2\lambda k}g\;Q^{+}(\chi g,\chi g)\text{d}v - \int_{\Omega_L}\langle v \rangle^{2\lambda k}g\;Q^{-}( g,\chi g)\text{d}v \\
+ \bigg(\big\|(Q_{c}(g,g)-Q_{u}(g,g))|v|^{\lambda k}\big\|_{L^{2}(\Omega_L)} + \big\|(\textbf{1}-\Pi^{N}_{2L})Q^{+}(\chi g,\chi g)\,|v|^{\lambda k}\big\|_{L^{2}(\Omega_L)}\bigg)\,\|g\|_{L^{2}_{k}(\Omega_L)}\,. 
\end{multline*} 
Using smoothing properties of the gain collision operator, see Theorem \ref{Tlp} in the appendix or refer to \cite{MV04,AG}, the lower bound control \eqref{gas1}, {and noticing that $C(g) = C(g_0)$ due to the conservation routine}, it follows that
\begin{align*}
\int_{\Omega_L}&\langle v \rangle^{2\lambda k}g\;Q^{+}(\chi g,\chi g)\text{d}v -  \int_{\Omega_L}\langle v \rangle^{2\lambda k}g\;Q^{-}( g,\chi g)\text{d}v \\
&\quad \leq \bigg(\frac{\max\big\{1,m^{(2-\lambda)/\mu}_{\mu}\big\}}{C(g_0)}\bigg)^{\theta_{1}}\|g\|^{\theta_{2}}_{L^{1}_{k}(\Omega_{L})}\|g\|^{1+1/d}_{L ^{2}_{k}(\Omega_{L})}\\
&\qquad\ \  - {C(g_0)}\bigg(\frac{1}{\max\big\{1,m^{(2-\lambda)/\mu}_{\mu}\big\}} 
- {\frac{C}{L^{2-\lambda}}}\bigg)\,\|g\| ^{2}_{L^{2}_{k+1/2}(\Omega_{L})}\,,
\end{align*}
with constant $C(g_0)$ depending only on mass and energy, $m_{\mu}$ defined in \eqref{m_mu}, and some universal $\theta_{1}>1,\theta_{2}>1$.  Meanwhile, using again estimates from Theorem \ref{t1}, the rest of the terms can be controlled by
\begin{align*}
\big\|(Q_{c}(g,g) &- Q_{u}(g,g))|v|^{\lambda k}\big\|_{L^{2}(\Omega_L)} + \big\|(\textbf{1}-\Pi^{N}_{2L})Q^{+}(\chi g,\chi g)\,|v|^{\lambda k}\big\|_{L^{2}(\Omega_L)}\\
&\leq L ^{\lambda k} \big\|(1-\Pi^{N}_{2L})Q^{+}(\chi g,\chi g)\big\|_ {L^{2}{(\Omega_L)}}+O_{d/2}\big(m_{k+1}(g)\,m_0(g)+Z_{k}(g)\big)\,,
\end{align*}
with $O_{r}$ defined in  \eqref{O_r}.  Therefore,  we conclude, provided $L\geq 2\max\big\{1,m^{1/\mu}_{\mu}\big\}$, that
\begin{align}\label{L2-diff-inequality}
\begin{split}
\frac{\text{d}}{\text{d} t}&\|g\|_{L^{2}_{k}(\Omega_L)} \leq \bigg(\frac{\max\big\{1,m^{(2-\lambda)/\mu}_{\mu}\big\}}{C(g_0)}\bigg)^{\theta_{1}}\|g\|^{\theta_{2}}_{L^{1}_{k}(\Omega_{L})}\|g\|^{1/d}_{L ^{2}_k{(\Omega_L)}}\\
&\hspace{2.5cm}  -\frac{C(g_0)}{\max\big\{1,m^{(2-\lambda)/\mu}_{\mu}\big\}}\|g\|_{L^{2}_{k+1/2}{(\Omega_L)}} \\
&+ L ^{\lambda k} \big\|(1-\Pi^{N}_{2L})Q^{+}(\chi g,\chi g)\big\|_ {L^{2}{(\Omega_L)}} + O_{d/2}\big(m_{k+1}(g)\,m_0(g)\!+\!Z_{k}(g)\big)\,.
\end{split}
\end{align}
Plugging \eqref{IPE} into \eqref{L2-diff-inequality} proves the first part following lemma.
\begin{lem}\label{L2propagation}
Fix $k\geq0$ and $\mu>0$ and assume $g$ is a solution of the numerical scheme satisfying \eqref{gas} for a small proportion $\epsilon\leq\epsilon(g_0)$ and cut-off domain $L \geq 2\max\big\{1,m^{1/\mu}_{\mu}\big\}$.  Then, the following differential inequality holds
\begin{align}\label{L2-diff-inequality-1}
\begin{split}
\frac{\text{d}}{\text{d} t}\|g\|_{L^{2}_{k}(\Omega_L)} \leq \bigg(\frac{\max\big\{1,m^{(2-\lambda)/\mu}_{\mu}\big\}}{C(g_0)}\bigg)^{\theta_{1}}&\|g\|^{\theta_{2}}_{L^{1}_{k}(\Omega_{L})}\|g\|^{1/d}_{L ^{2}_k{(\Omega_L)}}-\frac{C(g_0)}{\max\big\{1,m^{(2-\lambda)/\mu}_{\mu}\big\}}\|g\|_{L^{2}_{k+1/2}{(\Omega_L)}}\\
&\hspace{-2.5cm} + C\,\frac{L ^{\lambda k}}{N^{\frac{d-1}{2}}}\|g\|^{2}_{L^{2}_{1+\lambda^{-1}}(\Omega_L)} + O_{d/2}\Big(m_{k+1}(g)\,m_0(g)+Z_{k}(g)\Big)\,,
\end{split}
\end{align}
for some universal $\theta_{1}, \theta_{2}>1$ and  $O_{r}$ defined in  \eqref{O_r}.  Moreover, the negative part of $g$ satisfies
\begin{align}\label{nmest}
\begin{split}
\frac{d}{dt}\|g^{-}\|_{L^{2}(\Omega_L)}  \leq  \ & C  \|g_{0}\|_{L^{1}_{2}(\Omega_L)}\|g^{-}\|_{L^{2}(\Omega_L)} \!\\
&+ \!\frac{C}{N^{\frac{d-1}{2}}}\|g\|^{2}_{L^{2}_{1+\lambda^{-1}}(\Omega_L)} + O_{d/2+\lambda(k-1)}m_{k}(g)m_{0}(g_0)\,.
\end{split}
\end{align}

\end{lem}
\begin{proof}
For the part related to the negative mass, note that writing $g=g^{+}+g^{-}$ it follows that
\begin{align}\label{nmest1}
Q^{+}(g,g)\,g\,\textbf{1}_{\{g\leq0\}}&=\big(Q^{+}(g^{+},g^{+})+{Q^{+}(g^{+},g^{-})+Q^{+}(g^{-},g^{+})}+Q^{+}(g^{-},g^{-})\big) g\, \textbf{1}_{\{g\leq0\}} \nonumber\\
&\leq\left(Q^{+}(g^{+},g^{-})+Q^{+}(g^{-},g^{+})\right) g\,\textbf{1}_{\{g\leq0\}}\,.
\end{align}
Thus, using Young's inequality \cite{ACG, AC, MV04} one concludes that 
\begin{align*}
\int_{\Omega_L}Q^{+}(g,g) \, g \, &\textbf{1}_{\{g\leq0\}}\,\text{d}v \leq \int_{\Omega_L}\big(Q^{+}(g^{+},g^{-})+Q^{+}(g^{-},g^{+})\big)\,g\,\textbf{1}_{\{g\leq0\}}\,\text{d}v\nonumber\\
&\leq C\,\|b\|_{\infty}\|g^{+}\|_{L^{1}_{1}(\Omega_L)}\|g^{-}\|^{2}_{L^{2}(\Omega_L)} \leq C\,\|g_0\|_{L^{1}_{2\lambda^{-1}}(\Omega_L)}\,\|g^{-}\|^{2}_{L^{2}(\Omega_L)}\,,
\end{align*}

In this last inequality was important the bilinear estimates for $Q^{+}$ valid for $b\in L^{\infty}$.   Recall, additionally, that Lemma \ref{mainlemma} implies
\begin{equation*}
\int_{\Omega_L}Q^{-}(g,g)\,g\,\textbf{1}_{\{g\leq0\}}\,\text{d}v \geq \frac{C(g_0)}{\max\big\{1,m^{(2-\lambda)/\mu}_{\mu}\big\}}\|g^{-}\|^{2}_{L^{2}_{1/2}(\Omega_L)}\geq0.
\end{equation*}
As a consequence, multiplying equation \eqref{basic-equation} by $g^{-}$, integrating in $\Omega_{L}$ and invoking Theorem \ref{t1} with $k'=k-1$ and $k=0$ , one concludes that
\begin{align*}
\frac{d}{dt}\|g^{-}&\|_{L^{2}(\Omega_L)}  \leq  C \|g_{0}\|_{L^{1}_{2\lambda^{-1}}(\Omega_L)}\|g^{-}\|_{L^{2}(\Omega_L)} \\
&+ C\,\big\|(\textbf{1}-\Pi^{N}_{2L})Q^{+}(\chi g,\chi g)\big\|_{L^{2}(\Omega_L)} + O_{d/2+\lambda(k-1)}m_{k}(g)m_{0}(g_0)\,.
\end{align*}
The proof follows after plugging \eqref{IPE} in this estimate.
\end{proof}
\subsection{Uniform propagation of moments and $L^{2}_{k}$-norms} Now we are ready to prove uniform propagation of the scheme provided the requirement on the negative mass \eqref{gas} is met for $0<\epsilon\leq\epsilon(g_0)$.  Since Lemma \ref{L1prop} and Lemma \ref{L2propagation} hold for the aforementioned conditions on $\epsilon(g_0)$, one has the following to estimates on the $k$-moment and the $L^{2}$-norm.
\begin{align*}
\frac{\text{d}}{\text{d} t}&m_k(g)\leq C_{k}\,\big(m_{0}(g_0) + m_{k}(g) \big) - \frac14{\mu_{ \frac{\lambda}{2} }\,m_0(g_0)}\,m_{k+1}(g)+C\,\frac{L^{\lambda k+d/2}}{N^{\frac{d-1}{2}}}\|g\|^{2}_{L^{2}_{1+\lambda^{-1}}(\Omega_{L})}\,,\quad k\geq k_0\,,\\
\frac{\text{d}}{\text{d} t}&\|g\|_{L^{2}(\Omega_L)}\leq \bigg(\frac{\max\big\{1,m^{(2-\lambda)/\mu}_{\mu}\big\}}{C(g_0)}\bigg)^{\theta_{1}}\|g\|^{\theta_{2}}_{L^{1}(\Omega_{L})}\|g\|^{1/d}_{L ^{2}{(\Omega_{L})}}-\frac{C(g_0)}{\max\big\{1,m^{(2-\lambda)/\mu}_{\mu}\big\}}\|g\|_{L^{2}_{1/2}{(\Omega_{L})}}\\
&\hspace{3.0cm} + \frac{C}{N^{\frac{d-1}{2}}}\|g\|^{2}_{L^{2}_{1+\lambda^{-1}}(\Omega_{L})} + O_{d/2}\|g_{0}\|^{2}_{L^{1}_{2}(\Omega_{L})}\,.
\end{align*}
Note that using Young's inequality
\begin{align*}
\bigg(\frac{\max\big\{1,m^{(2-\lambda)/\mu}_{\mu}\big\}}{C(g_0)}\bigg)^{\theta_{1}}\|g\|^{\theta_{2}}_{L^{1}(\Omega_{L})}&\|g\|^{1/d}_{L ^{2}{(\Omega_{L})}}\\
&\hspace{-3.2cm}\leq C_{1}(g_{0}) + C_{2}(g_{0})  m^{\theta_{1}(1+d'/d)(2-\lambda)/\mu}_{\mu} + \frac{C(g_0)}{2\max\big\{1,m^{(2-\lambda)/\mu}_{\mu}\big\}}\|g\|_{L^{2}{(\Omega_{L})}}\,.
\end{align*}
Now, choose $\mu=\lambda k - 2$, so that $m_{\mu}=m_{k}(g)$, and take $k\geq k_0$ such that $\theta_{1}(1+d'/d)(2-\lambda)/\mu\leq 1$.  Then, adding previous two differential equations, one has
\begin{align*}
\frac{\text{d}}{\text{d} t}\big(m_k(g)+\|g\|_{L^{2}(\Omega_L)}\big) \leq \bigg( C_{k}(g_0) & - c(g_0)m^{1+1/k}_{k}(g)  - \frac{C(g_0)}{\max\big\{1,m^{(2-\lambda)/\mu}_{k}\big\}}\|g\|_{L^{2}{(\Omega_{L})}} \bigg) \\
& +  \frac{C\,L^{\lambda(k+2) + d/2 + 2}}{N^{\frac{d-1}{2}}}\|g\|^{2}_{L^{2}(\Omega_{L})}\,,
\end{align*} 
thus, defining $X(t) := m_k(g)+\|g\|_{L^{2}(\Omega_L)}$ and performing some algebra it follows that
\begin{equation}\label{propation-final}
\frac{\text{d}X}{\text{d} t} \leq \max\big\{1,m^{(2-\lambda)/\mu}_{k}\big\}\bigg(C_{k}(g_0) - c(g_{0})X + \frac{C\,L^{\lambda(k+2) + d/2 + 2}}{N^{\frac{d-1}{2}}}X^{2+(2-\lambda)/\mu}\bigg)\,.
\end{equation}
With this estimate we are in position of proving the following proposition.
\begin{prop}\label{propagation-moments-L2norms}
Fix $k\geq k_{*}$ and assume $g$ is a solution of the numerical scheme satisfying \eqref{gas} for $0<\epsilon\leq \epsilon(g_{0})$ with cut-off domain $L\geq 2\,\max\{1, m^{1/(\lambda k -2)}_{k}\}$.  Then, there exists a threshold $\eta(g_{0})>0$ only depending on $g_{0}$ such that if $$\frac{L^{\lambda(k+2)+d/2+2}}{N^{(d-1)/2}}\leq\eta(g_{0})\,,$$ then,
\begin{align*}
\sup_{t\geq0} m_{k}(g) &\leq \max\big\{C_{k}(g_0),m_{k}(g_0),\|g_{0}\|_{L^{2}(\Omega_{L})}\big\}=:\mathfrak{c}^{k}_{1}(g_0)\,,\quad\text{and}\\
\sup_{t\geq0} \|g\|_{L^{2}_{k'}(\Omega_{L})} &\leq \max\big\{C_{k'}(g_0),m_{k'+1}(g_0),\|g_{0}\|_{L^{2}_{k'}(\Omega_{L})}\big\}=:\mathfrak{c}^{k'}_{2}(g_0)\,,\quad \forall\,0\leq k'\leq k-1\,.
\end{align*}
Here $k_{*}\geq k_0$ is such that $\theta_{1}(1+d'/d)(2-\lambda)/(\lambda k_{*} - 2)\leq 1$ and $C_{k}(g_0)$ a constant depending on mass and energy of $g_0$ and $k$.
\end{prop}  
\begin{proof}
Consider the polynomial $p(x)=C_{k}(g_0) - c(g_{0})x + C\,\eta\, x^{2+(2-\lambda)/\mu}$.  Note that for sufficiently small $\eta$, depending only on $k\geq k_{*}\geq 2$ and the mass and energy of $g_0$, this polynomial has two positive roots $r_{1}$ and $r_{2}$.  As $\eta$ vanishes, $r_{1} \searrow C_{k}(g_{0})/c(g_0)$ and $r_{2}\nearrow\infty$.  Thus, choose $0<\eta$ sufficiently small such that
\begin{equation*}
m_{k}(g_0) + \|g_0\|_{L^{2}(\Omega_{L})}<r_{2}\,,
\end{equation*}
then, the differential inequality \eqref{propation-final} written as
\begin{equation*}
\frac{\text{d}X}{\text{d} t} \leq \max\big\{1,m^{(2-\lambda)/\mu}_{k}\big\}\,p(X)
\end{equation*}
for $\frac{L^{\lambda(k+2)+d/2+2}}{N^{(d-1)/2}}\leq\eta$ implies that
\begin{equation*}
\sup_{t\geq0}X(t)\leq \max\big\{C_{k}(g_0),X(0)\big\}\,.
\end{equation*}
This proves the first inequality of the statement and the propagation of $\|g\|_{L^{2}(\Omega_{L})}$.  Provided the latter, we use Lemma \ref{L2propagation} to conclude the second statement.
\end{proof}  
\section{Existence and regularity of the scheme}\label{sec:AccuracyConsistency}
\subsection{Existence} Now we are ready, thanks to the estimates of previous section, to produce a proof of existence and uniqueness of the numerical scheme.  We assume that $f_0\in L^{2}(\mathbb{R}^{d})$ is supported in $\Omega_{L}$, where the choice of the cut-off domain $\Omega_L$ was discussed in section~\ref{sec-cut-off}, and that $g_0=\Pi^{N}_{L}f_0$ satisfies
\begin{equation}\label{cofN}
\|g_0^{-}\|_{L^{2}(\Omega_L)}\approx 0\,.
\end{equation}
for $N \geq N_{0}(g_0)$ sufficiently large.  Observe also that defining the metric space $\mathcal{X}\subset\mathcal{C}(0,T;L^{2}(\Omega_{L}))$ as
\begin{align*}
\mathcal{X}:&=\big\{f\in\mathcal{C}(0,T;L^{2}(\Omega_{L})): (1)\,\sup_{t\in[0,T]}\|f(t)\|_{L^{2}(\Omega_{L})}\leq 2\mathfrak{c}^{0}_{2}(g_0)\,,\;(2)\,\sup_{t\in[0,T]}m_{k}(f)\leq 2\mathfrak{c}^{k}_{1}(g_0)\big\}\,,\\
&\text{and the operator}\;\;\mathcal{T}:\mathcal{X}\rightarrow\mathcal{C}(0,T;L^{2}(\Omega_{L}))\;\;\text{as}\;\; \mathcal{T}(f)(t)=g_0+\int^{t}_{0}Q_{c}(f)(s)\text{d}s\,,
\end{align*}
where $k\geq k_{*}\geq2$ and $\mathfrak{c}^{k}_{1}, \mathfrak{c}^{0}_{2}$ are those from Proposition \ref{propagation-moments-L2norms}, one has the estimates for some $a,b_{k}>0$
\begin{align*}
&\sup_{t\in[0,T]}\|\mathcal{T}(f)-\mathcal{T}(\tilde {f}\, )\|_{L^{2}(\Omega_{L})}\leq C(c^{k}_{1},c^{0}_{2})\,L^{a}\,T\,\sup_{t\in[0,T]}\|f - \tilde{f}\,\|_{L^{2}(\Omega_{L})}\,,\\
&\sup_{t\in[0,T]}m_{k}\big({\mathcal{T}(f)}\big) \leq m_{k}(g_0) + C(c^{k}_{1},c^{0}_{2})\,L^{b_{k}}\,T\,,\qquad f\,,\tilde{f}\in\mathcal{X}\,.
\end{align*}
As a consequence, choosing $T_{L}:=1/L^{a+b_{k}}$ for $L\geq L_0(g_{0})$ sufficiently large, it follow that $\mathcal{T}$ is a contraction with $\mathcal{T}(\mathcal{X})\subset\mathcal{X}$.  Using Banach fix point theorem, the scheme has a unique solution in $[0,T_{L}]$.  
\begin{thm}\label{extension}
Set $g_{0}=\Pi^{N}f_0 \in  L^{1}_{k}\cap L^{2}(\Omega_L)$, with $k\geq k_{*}\geq2$.  For any time $T>0$ and domain cut-off $L\geq L_0(T,g_0)>0$ there exists a number of modes $N_0(T,L, g_0)>0$ such that the Problem \eqref{fourierBTE} has a unique solution $g\in\mathcal{C}(0,T;L^{2}(\Omega_L))$ for any $N\geq N_0$ satisfying the estimates
\begin{equation*}
\sup_{ t\in[0,T] }\|g\|_{ L^{2}(\Omega_L) } \leq c^{0}_{k}(g_0),\qquad \sup_{ t\in[0,T] }m_{k}(g) \leq \mathfrak{c}^{k}_{1}(g_0) \,,
\end{equation*}
and negative mass estimate
\begin{align*}
&\|g^{-}(t)\|_{L^{2}(\Omega_L)}\ \leq C(\mathfrak{c}^{k}_{1},\mathfrak{c}^{0}_{2})\,e^{C\|g_0\|_{L^{1}_{2/\lambda}(\Omega_{L})}t}\\
&\hspace{1cm} \Big( \|g^{-}_{0}\|_{L^{2}(\Omega_{L})} + O\big(L^{2(1+\lambda)}/N^{(d-1)/2}\big) + \|g_0\|_{L^{1}_{2}(\Omega_{L})}\,O\big(1/L^{d/2+\lambda(k-1)}\big) \Big)\,.
\end{align*}
Furthermore, the sequence $\{g=g_{N}\}$ formed with initial condition $g_{0}$ converges strongly in $\mathcal{C}(0,T;L^{2}(\Omega_L))$ to $\bar{g}$, the solution of problem
\begin{equation}\label{limiteq}
\frac{\partial\bar{g}}{\partial t} = Q^{+}(\chi \bar{g},\chi \bar{g})  - Q^{-}(\bar{g},\chi\bar{g}) - \tfrac{1}{2}\Big(\bar{\gamma}_{1} + \sum^{d}_{j=1}\bar{\gamma}_{j+1}v_{j} + \bar{\gamma}_{d+2}|v|^{2}\Big)\,,\ \ (t,v)\in[0,T]\times\Omega_{L}\,,
\end{equation}
with initial condition $g_{0}=f_{0}$.  Above, the coefficients $\bar{\gamma}$ are given in Lemma \ref{l1} with parameters \eqref{Lconstrains}-\eqref{constraintEqns} evaluated at $Q^{+}(\chi \bar{g},\chi \bar{g})  - Q^{-}(\bar{g},\chi\bar{g})$.
\end{thm}
\begin{proof}
We start with $T>0$ given, $L>2\max\{1,(2\,c^1_{k}(g_{0}))^{1/\lambda k -2}\}$, and $N>0$ such that $L^{\lambda(k+2)+d/2+2}{N^{(-d+1)/2}}\leq\eta(g_{0})$.  We discussed that there exists a unique solution $g\in\mathcal{X}$ in the interval $I_{1}:=[0,1/L^{a+b_{k}}]$.  Note that the negative mass of such solution increases continuously in time.  Indeed, multiplying the scheme \eqref{basic-equation} by $g^{-}$, it readily follows that
\begin{align}\label{g^-est}
\begin{split}
\frac{\text{d}}{\text{d}t}\|g^{-}\|_{L^{2}(\Omega_{L})}&\leq C(c^{k}_{1},c^{0}_{2})\,L^{a}\,. \ \ \text{And, as a consequence}, \\
\|g^{-}(t_{1})\|_{L^{2}(\Omega_{L})} &\leq \|g^{-}(t_{0})| _{L^{2}(\Omega_{L})} + C(c^{k}_{1},c^{0}_{2})\,L^{a}(t_{1} - t_0)\,.
\end{split}
\end{align}
Since $g^{-}(0)\approx 0$, it means that the requirement on the negative mass of Proposition \ref{propagation-moments-L2norms} is satisfied in some interval $[0,t_{*}]\subset I_{1}$
\begin{equation}\label{nmass-ok}
0<\epsilon(t)\leq \epsilon(g_{0}) \,,\quad t\in [0,t_{*}]\,.
\end{equation}
Moreover, $L>0$ and $N>0$ were chosen to satisfy the requirements as well, therefore, estimate \eqref{nmest} holds in $[0,t_{*}]$.  Recalling the notation $O_{r}$, as  defined in  \eqref{O_r} and  integrating  estimate \eqref{g^-est}, it follows that
\begin{align*}
\|g^{-}(t)\|_{L^{2}(\Omega_{L})}& \leq e^{C\|g_0\|_{L^{1}_{2/\lambda}(\Omega_{L})}t}\Big(\|g^{-}_{0}\|_{L^{2}(\Omega_{L})} + \frac{4}{N^{(d-1)/2}}(\mathfrak{c}^{0}_{2}(g_{0}))^{2} + O_{d/2+\lambda(k-1)}\,c^{1}_{k}(g_{0})\,m_{0}(g_0)\Big)\\
&=:\varepsilon(t,L,N)\leq\varepsilon(T,L,N)\,.
\end{align*}
Now, note that
\begin{equation*}
\int_{\{g<0\}} |g(t,v)|\langle v \rangle^{2}\text{d}v \leq L^{d/2+2}\|g^{-}(t)\|_{ L^{2}(\Omega_{L} )}\leq L^{d/2+2}\varepsilon(t,L,N)\,,
\end{equation*} 
as a consequence, we can increase $L$ and $N$, if necessary, so that
\begin{align*}
\epsilon(t) :&= \frac{\int_{\{g<0\}} |g(t,v)|\langle v \rangle^{2}\text{d}v}{\int_{\{g\geq0\}} g(t,v)\langle v \rangle^{2}\text{d}v}=\frac{\int_{\{g<0\}} |g(t,v)|\langle v \rangle^{2}\text{d}v}{\int_{\Omega_{L}} g(t,v)\langle v \rangle^{2}\text{d}v - \int_{\{g<0\}} |g(t,v)|\langle v \rangle^{2}\text{d}v}\\
&\hspace{1cm}\leq \frac{L^{d/2+2}\varepsilon(T,L,N)}{\int_{\Omega_{L}} g_0(t,v)\langle v \rangle^{2}\text{d}v - L^{d/2+2}\varepsilon(T,L,N)}<\epsilon(g_{0})\,.
\end{align*}
Observe that we used the fact that the scheme conserves mass and energy and assumed that $\ k>1+2/\lambda$, so that  $L^{d/2+2}\varepsilon(T,L,N)$ vanishes as, both,  $L$ and then $N$ are chosen sufficiently large.  Therefore, for this choice of parameters $L\geq L_{0}(T,g_{0})$ and $N\geq N_{0}(T,L,g_{0})$, a continuation argument shows that the negative mass condition \eqref{nmass-ok} holds in fact in the whole interval $I_{1}$.  Thus, the \textit{a priori} estimates of Proposition \ref{propagation-moments-L2norms} hold in $I_{1}$ so that,
\begin{equation}\label{uniform-estimate-I1}
\|g(t)\|_{ L^{2}(\Omega_L) } \leq c^{0}_{k}(g_0),\qquad m_{k}(g(t)) \leq \mathfrak{c}^{k}_{1}(g_0) \,,\qquad \forall\;t\in I_{1}\,.
\end{equation}
Estimate \eqref{uniform-estimate-I1} shown that the set $\mathcal{X}/2$ is a stable set for the dynamics, thus, it allows us to uniquely extend the solution, by repeating the argument made for $I_{1}$, to the intervals $I_{i}:=[(i-1)/L^{a+b_{k}},i/L^{a+b_{k}}]$, with $i=1,2,\cdots$, until $[0,T]\subset \cup I_{i}$.  This proves global existence and uniqueness of the scheme.

Now, in the limit $N\rightarrow\infty$ one has that the sequence $\{g:=g^{N}\}\subset\mathcal{X}$.  Since,
\begin{equation*}
\big\|Q_{c}(f,f)(t)-Q_{c}(\tilde {f},\tilde{f})(t)\big\|_{L^{2}(\Omega_{L})}\leq C(c^{k}_{1},c^{0}_{2})\,L^{a}\|f(t) - \tilde{f}(t)\,\|_{L^{2}(\Omega_{L})}\,,\quad \forall\;f\,,\tilde{f}\in\mathcal{X}\,,
\end{equation*}
it follows from
\begin{equation*}
g(t)=g_{0} + \int^{t}_{0}Q_{c}(g,g)(s)\text{d}s\,,
\end{equation*}
that for any $N,M\geq N_0$ and $t\in[0,T]$
\begin{equation*}
\|g^{N}(t) - g^{M}(t)\|_{L^{2}(\Omega_{L})}\leq \|g^{N}_{0} - g^{M}_{0}\|_{L^{2}(\Omega_{L})} + C(c^{k}_{1},c^{0}_{2})\,L^{a}\int^{t}_{0}\|g^{N}(s) - g^{M}(s)\,\|_{L^{2}(\Omega_{L})}\text{d}s\,.
\end{equation*}
Thus, using Gronwall's lemma
\begin{equation*}
\|g^{N}(t) - g^{M}(t)\|_{L^{2}(\Omega_{L})}\leq \|g^{N}_{0} - g^{M}_{0}\|_{L^{2}(\Omega_{L})}e^{C(c^{k}_{1},c^{0}_{2})\,L^{a} T}\rightarrow 0\,,\quad\text{as}\quad N,M\rightarrow \infty\,.
\end{equation*}
Thus, $\{g^{N}\}$ is Cauchy and converges strongly to $\bar{g}$, the solution of the problem \eqref{limiteq} with initial condition $f_0=\lim_{N\rightarrow\infty}\Pi^{N}_{L}f_0$.
\end{proof}
\subsection{Uniform $H_{k}$ Sobolev regularity propagation}
In this section we work with functions in $H^{\alpha_0}(\Omega_L)$ and take multi-index $\alpha$ with $|\alpha|\leq\alpha_0$.  Recall that derivatives commute with the projection operator $\Pi^{N}_{2L}$, see \eqref{derivative_projection}, for functions in $H^{\alpha_0}_{0}(\Omega_{2L})$.  Therefore, distributing the derivatives in the arguments of the operator and using the estimates (\ref{extension_estimate}, \ref{extension-norm}) and \eqref{IPE}, yields
\begin{equation}\label{nafter}\begin{split}
\big\|\partial^{\alpha}&(\textbf{1}-\Pi^{N}_{2L})Q^{+}(\chi g,\chi g)\big\|_ {L^{2}(2\Omega_L)} \\
&= \big\|(\textbf{1}-\Pi^{N}_{2L})\partial^{\alpha}Q^{+}(\chi g,\chi g)\big\|_ {L^{2}(2\Omega_L)} \leq \frac{C\,L^{2(1+\lambda)}}{N^{(d-1)/2}}\|g\|^{2}_{H^{\alpha}(\Omega_{L})}\,,
\end{split}\end{equation}
where, we recall, that the constant $C:=C_{\chi}$ can be taken independent of $L\geq1.$

Next, in order to prove propagation of regularity let us fix $k\geq k_{*}\geq 2$ and $0\leq k'\leq k-1-\alpha_0(1+\lambda)$, and use an induction argument on the derivative order $|\alpha|$.  The initial step of the induction follows thanks to the propagation of $L^{2}_{k'}$-norms in Proposition \ref{propagation-moments-L2norms}.  For the case $|\alpha|\geq1$, assume the propagation of the $H^{|\alpha|-1}_{k'+(1+\lambda)}$-norms and differentiate equation \eqref{basic-equation} in velocity.  We arrive to
\begin{equation*}\begin{split}
\frac{\partial(\partial^{\alpha}g)}{\partial t}&=\partial^{\alpha}Q^{+}(\chi g,\chi g)- \partial^{\alpha}Q^{-}(g,\chi g)\\
&\qquad+\partial^{\alpha}\big(Q_c(g,g)-Q_{u}(g,g)\big) - \partial^{\alpha}\big(\textbf{1}-\Pi^{N}_{2L}\big)Q^{+}(\chi g,\chi g)\,.
\end{split}\end{equation*}
Multiply by $\partial^{\alpha}g \langle v \rangle^{2\lambda k'}$ and integrate in the velocity domain $\Omega_L$ to obtain
\begin{align}\label{controls}
\begin{split}
\tfrac{1}{2}\frac{\text{d}}{\text{d} t}\|\partial^{\alpha}g &\|^{2}_{L^{2}_{k'}(\Omega_L)} \leq \int_{\Omega_L}\big(\partial^{\alpha}Q^{+}(\chi g,\chi g)
- \partial^{\alpha}Q^{-}(g,\chi g)\big)\,\partial^{\alpha}g \langle v \rangle^{2\lambda k'} \\ &\ \ \ + \|\partial^{\alpha}g\|_{L^{2}_{k'}(\Omega_L)}\big\|\partial^{\alpha}\big(Q_c(g,g)-Q_{u}(g,g)\big)\big\|_{L^{2}_{k'}(\Omega_L)} \\
&\ \ \  + \|\partial^{\alpha}g\|_{L^{2}_{k'}(\Omega_L)}\big\|\partial^{\alpha}(\textbf{1}-\Pi^{N}_{2L})Q^{+}(\chi g,\chi g)\big\|_{L^{2}_{k'}(\Omega_L)}=:I_1+I_2+I_3\,.
\end{split}
\end{align}
Recall from Lemma \ref{l1} that the term $Q_c(g,g)-Q_{u}(g,g)$ is a second order polynomial, therefore its derivatives are at most a second order polynomial, thus Theorem \ref{t1} implies
\begin{equation}\label{ei2}\begin{split}
I_2 &\leq \|\partial^{\alpha}g\|_{L^{2}_{k'}(\Omega_L)}\Big(L^{\lambda k'}\big\|(\textbf{1}-\Pi^{N}_{2L})Q^{+}(\chi g,\chi g)\big\|_{L^{2}(\Omega_L)}\\
&\qquad \qquad + O_{d/2}\big(m_{k' +1}(g)m_{0}(g)+Z_{k' }\big)\Big).
\end{split}\end{equation}
Additionally, the term $I_3$ is controlled using \eqref{nafter}
\begin{equation}\label{ei3}
I_3 \leq \frac{L^{\lambda k' + 2(1+\lambda)}}{N^{(d-1)/2}} \|\partial^{\alpha}g\|_{L^{2}_{k'}(\Omega_L)}\|g\|^{2}_{H^{\alpha}(\Omega_{L})}\,.
\end{equation}
The term $I_1$ defined in \eqref{controls} can be controlled implementing the estimate introduced in \cite{BD} and used for the control of $H_{k'}$-norms in \cite[Theorem 3.5]{MV04}
\begin{align}\label{ei1}
\begin{split}
&I_1\leq C_{1}\left\|\partial^{\alpha} g\right\| _{L^{2}_{k'}(\Omega_L)}\|g\|^{2}_{H^{|\alpha|-1}_{k'+(1+1/\lambda)}(\Omega_{L})} - C(g_0)\left\|\partial^{\alpha} g\right\|^{2}_{L^{2}_{k'+1/2}(\Omega_L)}\\
&\qquad\leq C_{2}\left\|\partial^{\alpha} g\right\| _{L^{2}_{k'}(\Omega_L)} - C(g_0)\left\|\partial^{\alpha} g\right\|^{2}_{L^{2}_{k'+1/2}(\Omega_L)}\,,\\
&\qquad\qquad\qquad\text{where}\; C_{1}\|g\|^{2}_{H^{|\alpha|-1}_{k'+(1+1/\lambda)}(\Omega_{L})}\leq C_2\; \text{by induction.}
\end{split}
\end{align}
We obtain from inequalities \eqref{controls}, \eqref{ei2}, \eqref{ei3} \eqref{ei1} and \eqref{nafter}
\begin{equation*}
\frac{\text{d}}{\text{d} t}\|\partial^{\alpha}g\|_{L^{2}_{k'}(\Omega_L)}\leq C_2-\frac{C(g_{0})}{2}\|\partial^{\alpha} g\|_{L^{2}_{k'+1/2}(\Omega_L)} + \frac{L^{\lambda k' + 2(1+\lambda)}}{N^{(d-1)/2}} \|g\|^{2}_{H^{\alpha}(\Omega_{L})}.
\end{equation*}
Same inequality is valid for $\alpha=0$, therefore, it is concluded that
\begin{equation*}
\frac{\text{d}X}{\text{d} t}\leq C_2 - \frac{C(g_{0})}{2}X + \frac{L^{\lambda k' + 2(1+\lambda)}}{N^{(d-1)/2}} X^{2}.
\end{equation*}
where $X(t):=\|g\|_{H^{\alpha}_{k'}(\Omega_{L})}$.  From here, after taking $N\geq N_{0}(L,g_0)$ sufficiently large, it follows that
\begin{equation*}
X(t)\leq \max\big\{X(0),4C_{1}/C_{2}\big\}\,,\qquad t\in[0,T]\,.
\end{equation*}
Note that in each step of the induction one needs to add $(1+1/\lambda)$ moments, so that $\|g\|_{H^{|\alpha|-1}_{k'+(1+1/\lambda)}}$ is finite.  Having this in mind, let us state the result we just proved.  
\begin{prop}\label{H2propagation}
Fix $T>0$, $\alpha\geq0$, $k\geq k_{*}\geq2$ and $0\leq k'\leq k-1-\alpha(1+1/\lambda)$ and assume $g_0\in H^{\alpha}_{k'+\alpha(1+1/\lambda)}(\Omega_L)$.  Then, for any lateral size $L\geq L_{0}(T,g_0)$ there exists $N_0(T,L,g_0)>0$ such that
\begin{equation*}
\sup_{t\in[0,T]}\|g\|_{H^{\alpha}_{k'}(\Omega_L)}\leq \max\big\{\|g_0\|_{H^{\alpha}_{k'+\alpha|(1+1/\lambda)}(\Omega_L)},C_{k'}(g_0)\big\},\quad N\geq N_0\,.
\end{equation*}
\end{prop}
\section{$L^{2}_{k}$ and $H^{\alpha}_{k}$ error estimates}

We are now in position to write the error estimates for the spectral conservation scheme.  We start with errors in the $L^{2}_{k'}$-norm and, then, extend it to Sobolev norms.  Again, we start fixing $T>0$, the cut-off domain $L\geq L_{0}(T,g_0)$ and $N\geq N_{0}(T,L,g_0)$ sufficiently large so that $g$ exists in the interval $[0,T]$.  Here $k\geq k_{*}\geq2$ and $0\leq k'\leq k-1$ in order to meet the assumptions of Proposition \ref{propagation-moments-L2norms}.  From the identity
\begin{align}\label{E_0}
\begin{split}
Q(g,g) &= Q(\chi g,\chi g) \\
&\ \ \ + \big(Q((1-\chi)g,g) + Q(g,(1-\chi)g) + Q((1-\chi)g,(1-\chi)g)\big)\\
& =: Q(\chi g, \chi g) +E_{0}(g,g)\,,
\end{split}
\end{align}
and the definition of $Q_{u}$, one finds that
\begin{align}\label{E}
\begin{split}
Q_{u}(g,g) &= Q(g,g) - \big( E_{0}(g,g) + Q^{-}((1-\chi)g,\chi g) \big)\\
&= : Q(g,g) - E(g,g)\,.
\end{split}
\end{align}
Now, observe that subtracting the Boltzmann equation \eqref{singleEq} and its conserved projection approximation \eqref{fourierBTE} in $\Omega_L$ one obtains
\begin{align}\label{difference}
\begin{split}
\partial_{t}(f-g) &= Q(f,f)-Q_{c}(g,g) = \big(Q(f,f) - Q_{u}(g,g)\big)+\big(Q_{u}(g,g) - Q_{c}(g,g)\big)\\
& = \big(Q(f,f) - Q(g,g)\big)  + \big(Q_{u}(g,g) - Q_{c}(g,g)\big) + E(g,g)\,.
\end{split}
\end{align} 
Multiplying this equation by $(f-g) \langle v \rangle^{2\lambda k'}$ and integrating in $\Omega_L$
\begin{align*}
\tfrac{1}{2}\frac{\text{d}}{\text{d} t}&\|f-g\|^{2}_{L^{2}_{k'}(\Omega_L)}
=\int_{\Omega_L}\langle v \rangle^{2\lambda k'}(f-g)\big(Q(f,f) - Q(g,g)\big)\big)\,\text{d}v\\
&\hspace{-0.8cm}+\int_{\Omega_L} \langle v \rangle^{2\lambda k'}(f-g)\big(Q_{u}(g,g) - Q_{c}(g,g)\big)\,\text{d}v + \int_{\Omega_L} \langle v \rangle^{2\lambda k'}(f-g)E(g,g)\,\text{d}v=:I_1+I_2 + I_{3}\,.
\end{align*} 
The error term $I_{3}$, from the error term $E(g,g)$ in \eqref{E}, is simply controlled as
\begin{align}\label{I_{3}}
\begin{split}
I_{3}&\leq \|g\|_{L^{1}_{k'+1}(\mathbb{R}^{d})}\|(1-\chi)g\|_{L^{2}_{k'+1}(\mathbb{R}^{d})}\|f-g\|_{L^{2}_{k'}(\Omega_{L})} \leq O_{d/2+\lambda k''}\|f-g\|_{L^{2}_{k'}(\Omega_{L})}\\
\end{split}
\end{align}
where the last inequality holds provided the $L^{2}_{k'+1+d/2\lambda+k''}$ uniformly propagate.  Moreover, using Theorem \ref{t1} it follows that
\begin{multline*}
\big\|Q_{u}(g,g)-Q_c(g,g)\big\|_{L^{2}_{k'}(\Omega_L)}\\
\leq C_5\,L^{\lambda k'}\big\|(\text{1}-\Pi^{N}_{2L})Q^{+}(\chi g,\chi g)\big\|_{L^{2}(\Omega_L)} + O_{d/2+\lambda k''}m_{k'+1+k''}(g)m_{0}(g_0)\,.
\end{multline*}
Therefore, using Cauchy-Schwarz inequality and \eqref{IPE} inequality one controls the term $I_{2}$ as 
\begin{align}\label{I_{2}}
\begin{split}
I_2\leq \|f - & g\|_{L^{2}_{k'}{(\Omega_L)}}\bigg(\frac{C_{5}\,L^{\lambda k'}}{N^{(d-1)/2}}\|g\|^{2}_{L^{2}_{1+1/\lambda}(\Omega_{L})} + O_{d/2+\lambda k''}m_{k'+1+k''}(g)m_{0}(g_0)\bigg)\\
& = \|f-g\|_{L^{2}_{k'}{(\Omega_L)}}\bigg(O\big(L^{\lambda k'}/N^{(d-1)/2}\big) + O_{d/2+\lambda k''}\bigg)\,.
\end{split}
\end{align}
The term $I_{1}$ is more involved.  However, it is classical from the Boltzmann theory that the Dirichlet form of the linearized collision operator with polynomial weights is essentially nonpositive in the sense that 
\begin{align}\label{I_{1}}
\begin{split}
I_1=\tfrac{1}{2}&\int_{\mathbb{R}^{d}} - \int_{\mathbb{R}^{d}\setminus\Omega_{L}}\langle v \rangle^{2\lambda k'}(f-g)\big(Q(f+g,f - g) + Q(f - g,f + g)\big)\,\text{d}v \\
&\leq C_{k'}\|f-g\|^{2}_{L^{2}_{k'}(\mathbb{R}^{d})} + \big(\frac{c_{1}}{k'} + 2\|g_0\|_{L^{1}_{2\lambda^{-1}}}\,\epsilon - c_{2}  \big) \|f-g\|^{2}_{L^{2}_{k'+1/2}(\mathbb{R}^{d})}\\
&\hspace{2cm} + O_{d/2+\lambda k''}\|f+g\|_{L^{1}_{k'}(\mathbb{R}^{d})}\|f-g\|_{L^{2}_{k'}(\mathbb{R}^{d})}\|f-g\|_{L^{2}_{k'+1+d/2+k''}(\mathbb{R}^{d})}\\
&\leq C_{k'}\|f-g\|^{2}_{L^{2}_{k'}(\Omega_{L})} + O_{d/2+\lambda k''}\big(\|f-g\|_{L^{2}_{k'}(\Omega_{L})} + O_{d/2+\lambda k''}\big)\,.
\end{split}
\end{align}
The $\epsilon$-term, with $0<\epsilon<\epsilon(g_0)$, is added in the absorption (second) term to account for the fact that may be a set where $f+g$ is negative.  This is not a problem since this set is small $\{f+g<0\}\subset\{g<0\}$.  Here, $C_{k'}$ is a constant that depends on the moments $L^{1}_{k'+2}$ and $c_{i}:=c_{i}(f_{0},g_{0})$ depends only on the initial mass and some moment $2^{+}/\lambda$, see for instance \cite[Proposition 2.1]{DM04}.  In the last inequality, we are taking $k'$, and $L$, sufficiently large so that $c_{1}/k' + C_{k}/L^{\lambda} + 2\|g_0\|_{L^{1}_{2/\lambda}}\,\epsilon - c_{2} \leq  0$, which is achieved for any $\epsilon$ in the aforementioned range.  This estimate holds, of course, provided the $L^{2}_{k'+1+d/2\lambda + k''}$-norms propagate uniformly on $[0,T]$.  Defining $X(t):=\|f(t)-g(t)\|^{2}_{L^{2}_{k'}(\Omega_L)}$ and combining the estimates \eqref{I_{1}}, \eqref{I_{2}} and \eqref{I_{3}}
\begin{equation*}
\tfrac{1}{2}\frac{\text{d}X}{\text{d} t}(t)\leq C_{k'}X(t)+\Big(O\big(L^{\lambda k'}/N^{(d-1)/2}\big) + O_{d/2+\lambda k''}\Big)\sqrt{X} + O_{d+2\lambda k''}\,.
\end{equation*}
Thus, Gronwall's lemma implies
\begin{equation}\label{finalestimate}
\sup_{t\in[0,T]}\|f-g\|^{2}_{L^{2}_{k'}(\Omega_L)}\leq e^{C_{k'}T}\big(\|f_{0}-g_{0}\|^{2}_{L^{2}_{k'}(\Omega_L)} + O\big(L^{2 \lambda k'}/N^{(d-1)}\big) + O_{d+2\lambda k''}\big)\,.
\end{equation}
This proves the following theorem.
\begin{thm}[$L^{2}_{k}$-error estimate]\label{l2convergence}
Fix $k\geq k_{*}\geq2 $, $k''\geq0$, and $k(f_{0}) < k'\leq k-1-\frac{d^{+}}{2\lambda}-k''$ with $0\leq f_{0}\in L^{1}_{2}\cap L^{2}_{k}(\mathbb{R}^{d})$ be an initial datum and $f$ be the solution of the Boltzmann equation \eqref{singleEq}.  For any $T>0$ and cut-off domain $L(T,f_0)\geq L_0$ there exists $N_0(T,L,f_0)$ such that
\begin{equation*}
\sup_{t\in[0,T]}\|f - g \|_{L^{2}_{k'}(\Omega_L)}\leq e^{C_{k'}T}\big(\|f_{0}-g_{0}\|_{L^{2}_{k'}(\Omega_L)} + O\big(L^{\lambda k'}/N^{(d-1)/2}\big) + O_{d/2+\lambda k''}\big),\ \ N\geq N_0\,.
\end{equation*}
The factor $O_{d/2+\lambda k''}$ is defined by  \eqref{O_r}. The constants depend as $C_{k'}:=C_{k'}\big(\|f_0\|_{L^{2}_{k'}}\big)$.  In particular, the strong limit $\bar{g}$ of the sequence $\{g_{N}\}$ in $\mathcal{C}(0,T;L^{2}_{k}(\Omega_L))$ satisfies the same estimate. 
\end{thm}
We study next the improvement in the rate of convergence with respect to the number of modes $N$ of the approximating solutions towards the Boltzmann solution provided that the initial configuration is smooth and has at least initial mass and energy bounded.
\begin{thm}[$H^{\alpha}$-error estimates]\label{l2convergenceregularity}
Fix $k\geq k_{*}\geq 2$, $k''\geq0$, $\alpha_{0}\geq\alpha\geq0$, and $k(f_0)\leq k' \leq k-1- \alpha/2 - \frac{d^{+}}{2\lambda} - k''$ and let $0\leq f_{0} \in L^{1}_{2}\cap H^{\alpha}_{k}(\mathbb{R}^{d})$ be an initial datum and $f$ be the solution of the Boltzmann equation \eqref{singleEq}.  Fix $T>0$ and cut-off domain $L\geq L_0(T,f_0)$.  Then, there exists $N_{0}(T,L,f_0)$ such that
\begin{align}\label{T5.3-estimate}
\begin{split}
\sup_{t\in[0,T]}&\|f-g\|_{H^{\alpha}_{k'}(\Omega_L)} \leq e^{ \alpha C_{k'}T}\big(\|f_{0}-g_{0}\|_{H^{\alpha}_{k'+\alpha/2}(\Omega_L)} \\
&\hspace{2.5cm}+ O\big(L^{\lambda(k'  +\alpha/2) + \alpha_0}/N^{(d-1)/2 + \alpha_0}\big) + O_{d/2+\lambda k''}\big)\,,\quad N\geq N_0\,.
\end{split}
\end{align}
with the factor $O_{d/2+\lambda k''}$  defined  as in   \eqref{O_r}.
\end{thm}
\begin{proof}
Fix $\alpha_{0}\geq0$, $k\geq k_{*}\geq 2$, $k''\geq0$ and $k(f_0)\leq k' \leq k-1- \alpha_0/2 - \frac{d+1}{2\lambda} - k''$.  Now, we perform similar computations to those of the error estimates for $L^{2}_{k'}$, though, avoiding to resource to the values of $g$ near $\partial\Omega_{L}$.  Thus, we write, for $Q_{u}$ and $Q_{c}$ defined in \eqref{fourierBTE-0} and  \eqref{fourierBTE}, respectively, 
\begin{align}\label{difference-sob}
\begin{split}
Q(f,f) - Q_{c}(g,g) &= 
Q(\chi f, \chi f) - Q(\chi g, \chi g) - Q^{-}((1-\chi)(f-g),\chi g) \\
&\ \ \ + \big(Q_{u}(g,g) - Q_{c}(g,g)\big) + \tilde{E}(f,f)\,.
\end{split}
\end{align}
Here $\tilde{E}(f,f):=E_{0}(f,f) + Q^{-}((1-\chi)f,\chi g)$.  Thus, fixing any multi-index $\alpha$ with $|\alpha|\leq \alpha_{0}$, we apply the operator $\partial^{\alpha}$ to equation \eqref{difference-sob}, multiply it by $\partial^{\alpha}(f-g) \langle v \rangle^{2\lambda k'}$, and integrate in $\Omega_L$ to obtain
\begin{align*}
\tfrac{1}{2}&\frac{\text{d}}{\text{d} t}\|\partial^{\alpha}(f-g)\|^{2}_{L^{2}_{k'}(\Omega_L)}
=I^{\alpha}_1+I^{\alpha}_2 + I^{\alpha}_{3}\,,
\end{align*}
where,
\begin{align*}
I^{\alpha}_1&:= \int_{\Omega_L}\langle v \rangle^{2\lambda k'}\partial^{\alpha}(f-g)\,\big(\partial^{\alpha}Q(\chi f,\chi f) - \partial^{\alpha} Q(\chi g,\chi g) - \partial^{\alpha}Q^{-}((1-\chi)(f-g),\chi g)\big)\text{d}v\,,\\
I^{\alpha}_2&:= \int_{\Omega_L} \langle v \rangle^{2\lambda k'}\partial^{\alpha}(f-g)\,\partial^{\alpha}\big(Q_{u}(g, g)-Q_{c}(g,g)\big)\text{d}v\,,\\
I^{\alpha}_{3}&:= \int_{\Omega_L} \langle v \rangle^{2\lambda k'}\partial^{\alpha}(f-g)\,\partial^{\alpha}\tilde{E}(f,f)(v)\text{d}v\,.
\end{align*} 
Regarding the term $I^{\alpha}_{2}$, we directly use Theorem \ref{t1} to have
\begin{align*}
\big\|\partial^{\alpha}\big(Q_{u}(g,& g) - Q_c(g,g)\big)\big\|_{L^{2}_{k'}(\Omega_L)}\lesssim \big\| Q_{u}(g,g) - Q_c(g,g) \big\|_{L^{2}_{k'}(\Omega_L)}\\
&\leq C_5\,L^{\lambda k'}\big\|(\text{1}-\Pi^{N}_{2L})Q^{+}(\chi g,\chi g)\big\|_{L^{2}(\Omega_L)} + O_{d/2+\lambda k''}m_{k'+1+k''}(g)m_{0}(g_0)\,.
\end{align*}
Therefore, using Cauchy-Schwarz inequality, inequality \eqref{IPE}, and Lemma \ref{lemmaProjectionEstimate} implies 
\begin{align*}
I^{\alpha}_2& \leq \|\partial^{\alpha}(f-g)\|_{L^{2}_{k'}{(\Omega_L)}}\bigg(\frac{C_{5}\,L^{\lambda k' + \alpha_0 }}{N^{(d-1)/2+\alpha_{0} }}\|g\|^{2}_{H^{\alpha_{0}}_{1+1/\lambda}(\Omega_{L})} + O_{d/2+\lambda k''}m_{k'+1+k''}(g)m_{0}(g_0)\bigg)\\
&\hspace{1cm} = \|\partial^{\alpha}(f-g)\|_{L^{2}_{k'}{(\Omega_L)}}\bigg(O\big(L^{\lambda k' + \alpha_0}/N^{(d-1)/2 + \alpha_0}\big) + O_{d/2+\lambda k''}\bigg)\,.
\end{align*}
The term $I^{\alpha}_{3}$, containing the error term $\tilde{E}(f,f)$, is simply controlled as
\begin{align*}
I^{\alpha}_{3}& \leq C_{\alpha}\sum_{\alpha'+\beta'=\alpha}\Big(\|\partial^{\alpha'}(\chi g)\|_{L^{1}_{k'+1}(\mathbb{R}^{d})}+\|\partial^{\alpha'}f\|_{L^{1}_{k'+1}(\mathbb{R}^{d})}\Big)\|\partial^{\beta'}((1-\chi)f)\|_{L^{2}_{k'+1}(\mathbb{R}^{d})}\|\partial^{\alpha}(f-g)\|_{L^{2}_{k'}(\Omega_{L})}\\
&\leq C_{\alpha}\,\|(1-\chi)f\|_{H^{\alpha}_{k'+1}(\mathbb{R}^{d})}\|\partial^{\alpha}(f-g)\|_{L^{2}_{k'}(\Omega_{L})} \leq O_{d/2+\lambda k''}\|\partial^{\alpha}(f-g)\|_{L^{2}_{k'}(\Omega_{L})}\,,
\end{align*}
where the last inequality holds provided the $H^{\alpha}_{k'+1+d/2\lambda+k''}$-norm of $f$ uniformly propagates.  Finally, for the term $I^{\alpha}_{1}$ one checks that
\begin{align*}
\partial^{\alpha}&\big(Q(\chi f,\chi f) - Q(\chi g,\chi g) \big) = \tfrac{1}{2}\partial^\alpha\big(Q(\chi (f-g),\chi (f+g)) + Q(\chi(f+g),\chi (f-g))\big) \\
&=\tfrac{1}{2}\big(Q\big(\partial^{\alpha}(\chi (f-g)),\chi (f+g)\big) + Q\big(\chi(f+g), \partial^{\alpha}(\chi (f-g))\big)\big) + \sum_{\alpha'+\beta'<\alpha}\Gamma^{\alpha}_{\alpha',\beta'}\,,
\end{align*}
where
\begin{equation*}
\Gamma^{\alpha}_{\alpha',\beta'} := \tfrac{C^{\alpha}_{\alpha',\beta'}}{2}\big(Q\big(\partial^{\alpha'}(\chi (f-g)),\partial^{\beta'}\chi (f+g)\big) + Q\big(\partial^{\beta'}\chi(f+g), \partial^{\alpha'}(\chi (f-g))\big)\big)\,.
\end{equation*}
Observe that
\begin{align*}
\big\|\Gamma^{\alpha}_{\alpha',\beta'}&\big\|_{L^{2}_{k'-1/2}(\Omega_{L})}\leq C\,\|\partial^{\alpha'}\chi(f-g)\|_{L^{2}_{k'+1/2}(\mathbb{R}^{d})}\|\partial^{\beta'}\chi(f+g)\|_{L^{1}_{k'+1/2}(\mathbb{R}^{d})}\\
&\leq C\,\|\partial^{\alpha'}\chi(f-g)\|_{L^{2}_{k'+1/2}(\mathbb{R}^{d})} = C\,\|\partial^{\alpha'}(f-g)\|_{L^{2}_{k'+1/2}(\Omega_{L})}+O_{d/2+\lambda k''}\,,
\end{align*}
provided the $H^{\alpha}_{k'+1+d/2\lambda+k''}$-norms are propagated.  Therefore,
\begin{multline*}
\int_{\Omega_L} \langle v \rangle^{2\lambda k'}\partial^{\alpha}(f-g)\, \sum_{\alpha'+\beta'<\alpha}\Gamma^{\alpha}_{\alpha',\beta'}\text{d}v\\
\leq \|\partial^{\alpha}(f-g)\|_{L^{2}_{k'+1/2}(\Omega_{L})}\big(C_{\alpha}\sum_{|\alpha'|<|\alpha|}\|\partial^{\alpha'}\chi(f-g)\|_{L^{2}_{k'+1/2}(\Omega_{L})}+O_{d/2+\lambda k''}\big)\,.
\end{multline*}
Now, the leading order term in $I^{\alpha}_{1}$ is the Dirichlet form of the linearized Boltzmann operator with $\partial^{\alpha}\chi(f-g)$.  Thus, similar to what was done in the $L^{2}_{k'}$ error estimate, it follows that
\begin{align*}
\begin{split}
I^{\alpha}_1& \leq C_{k'}\|\partial^{\alpha}(f-g)\|^{2}_{L^{2}_{k'}(\mathbb{R}^{d})} + \big(\frac{c_{1}}{k'} + \epsilon - c_{2}\big) \|\partial^{\alpha}(f-g)\|^{2}_{L^{2}_{k'+1/2}(\mathbb{R}^{d})} + O_{d/2+k''}\|\partial^{\alpha}(f-g)\|_{L^{2}_{k'}(\mathbb{R}^{d})}\\
&\hspace{2cm} + C_{\alpha}\sum_{|\alpha'|<|\alpha|}\|\partial^{\alpha'}\chi(f-g)\|^{2}_{L^{2}_{k'+1/2}(\Omega_{L})}+O_{d+2k''} \leq C_{k}\|f-g\|^{2}_{L^{2}_{k'}(\Omega_{L})}\\
& + O_{d/2+\lambda k''}\big(\|\partial^{\alpha}(f-g)\|_{L^{2}_{k'}(\Omega_{L})} + O_{d/2+\lambda k''}\big) + C_{\alpha}\sum_{|\alpha'|<|\alpha|}\|\partial^{\alpha'}(f-g)\|^{2}_{L^{2}_{k'+1/2}(\Omega_{L})} + O_{d+2\lambda k''}\,.
\end{split}
\end{align*}
Accordingly, this holds provided the $H^{\alpha}_{k'+1/2+d/2\lambda+k''}$-norms propagate uniformly on $[0,T]$.  Additionally, in obtaining this estimate we have used the term with $\partial^{\alpha}Q^{-}((1-\chi)(f-g),\chi g)$ to complete the $L^{2}_{k'+1/2}$- absorbing norm in the whole $\mathbb{R}^{d}$.  As a consequence, defining $X^{\alpha}(t):=\|\partial^{\alpha}(f(t)-g(t))\|^{2}_{L^{2}_{k'}(\Omega_L)}$ and combining the estimates for $I^{\alpha}_1$, $I^{\alpha}_2$, and $I^{\alpha}_{3}$
\begin{align*}
\tfrac{1}{2}\frac{\text{d}X^{\alpha}}{\text{d} t}(t)\leq C_{k}X^{\alpha}(t) + & O\big(L^{\lambda k' + \alpha_0}/N^{(d-1)/2 + \alpha_0}\big)\sqrt{X^{\alpha}} \\
&+ O_{d/2+\lambda k''} + C_{\alpha}\sum_{|\alpha'|<|\alpha|}\|\partial^{\alpha'}(f-g)\|^{2}_{L^{2}_{k'+1/2}(\Omega_{L})}\,.
\end{align*}
Thus, Gronwall's lemma implies
\begin{align*}
X^{\alpha}(t) \leq e^{2C_{k}T}&\Big(X^{\alpha}(0) + O\big(L^{2\lambda k' + 2\alpha_0}/N^{d-1 + 2\alpha_0}\big)\\
& + O_{d+2\lambda k''} + C_{\alpha}\sum_{|\alpha'|<|\alpha|}\sup_{t\in[0,T]}\|\partial^{\alpha'}(f-g)(t)\|^{2}_{L^{2}_{k'+1/2}(\Omega_{L})}\Big)\,.
\end{align*}
Estimate \eqref{T5.3-estimate} follows by iteration of this formula on the multi-index order $|\alpha|=1,2,\cdots \alpha_0$, using Theorem \ref{l2convergence} as starting point.
\end{proof}

{\section{Long time behavior}
In this final section we address  the long time behavior for the semi-discrete problem given by the conservative spectral  scheme approximating the space homogeneous elastic Boltzmann equation for hard potentials with integrable angular cross section.  \\
Thus, we start by setting $g=\mathcal{M}_0 + h$, where $h:=g-\mathcal{M}_0$ is the perturbation from the global Maxwellian equilibrium defined in \eqref{eqMax}.
 Note that under this linearization
\begin{equation*}
Q_{c}(g,g) = Q_{c}(\mathcal{M}_0,\mathcal{M}_0) + Q_{c}(\mathcal{M}_0,h) + Q_{c}(h,\mathcal{M}_0) + Q_{c}(h,h)\,.
\end{equation*}
Introduce then the linear operators
\begin{align*}
\mathcal{L}_{c}(h) :&= Q_{c}(\mathcal{M}_0,h) + Q_{c}(h,\mathcal{M}_0)\,,\\
\mathcal{L}(\chi h) :&= Q(\mathcal{M}_0,\chi h) + Q(\chi h,\mathcal{M}_0)\,.
\end{align*}
The reader recognizes the latter $\mathcal{L}$ as the linearized Boltzmann operator.  With the estimations we have performed in previous section, it is clear that
\begin{align*}
\|\chi Q_{c}(\mathcal{M}_0,\mathcal{M}_0) \|_{H^{\alpha}_{k}(\mathbb{R}^{d})}\leq \text{O}\big( L^{\lambda k}/N^{\frac{d-1}{2}}\big) + \text{O}\big( 1/L^{\lambda k}\big)\,,\\
\|\chi \mathcal{L}_{c}(h) - \mathcal{L}(\chi h) \|_{H^{\alpha}_{k}(\mathbb{R}^{d})}\leq \text{O}\big( L^{\lambda k}/N^{\frac{d-1}{2}}\big) + \text{O}\big( 1/L^{\lambda k}\big)\,,\\
\|\chi Q_{c}(h,h) - Q(\chi h, \chi h) \|_{H^{\alpha}_{k}(\mathbb{R}^{d})}\leq \text{O}\big( L^{\lambda k}/N^{\frac{d-1}{2}}\big) + \text{O}\big( 1/L^{\lambda k}\big)\,.
\end{align*}
For the last two estimates we need $h$, thus $g$, having $\alpha$ derivatives and $2k$-moments in $\Omega_{L}$.  This, of course, is guaranteed by the results of Section 5 as long as the negative mass in $g$ is small, $\epsilon\leq\epsilon(g_0)$.  As a consequence,
\begin{equation}\label{sg}
\frac{\text{d}}{\text{d}t}\chi h = \mathcal{L}(\chi h) + Q(\chi h,\chi h) + \mathcal{R}(h)\,,
\end{equation}
where the remainder is of size $\|\mathcal{R}(h)\|_{H^{\alpha}_{k}(\mathbb{R}^{d})}\leq \text{O}\big( L^{\lambda k}/N^{\frac{d-1}{2}}\big) + \text{O}\big( 1/L^{\lambda k}\big)$.  Now, classical estimates on the Boltzmann operator and interpolation shows that
\begin{equation*}
\|Q(\chi h,\chi h)\|_{H^{\alpha}_{k}(\mathbb{R}^{d})}\leq C_{k}\|\chi h\|^{3/2}_{H^{\alpha}_{k}(\mathbb{R}^{d})}\,,
\end{equation*}
where the constant $C_{k}$ depends on $k'$-moments and the $H^{\alpha}_{k'}$-norm of $h$,  for some $k' \geq k+2\lambda + d$.  Furthermore, the linearized Boltzmann operation has spectral gap, say $\nu>0$, in $H^{\alpha}_{k}$.  See for example reference \cite{Des-Vil-2005, Mcon}.  Thus, we can integrate \eqref{sg} to obtain that
\begin{equation*}
\chi h(t) = \chi h_0 + \int^{t}_{0}e^{\mathcal{L}(t-s)}Q(\chi h,\chi h)(s)\text{d}s + \int^{t}_{0}e^{\mathcal{L}(t-s)}\mathcal{R}(h)(s)\text{d}s\,.
\end{equation*}  
Since, the remainder $\mathcal{R}(h)$ may not have zero mass, momentum, and energy, we apply the operator $1-\pi$, where $\pi$ is the standard projection on the Boltzmann null space in $H^{\alpha}_{k}(\mathbb{R}^{d})\subset L^{1}_{2/\lambda}(\mathbb{R}^{d})$, which is given by
\begin{equation*}
\pi h = \sum_{\phi\in\{1,v_{1},\cdots,v_{d},|v|^{2}\}}\int_{\mathbb{R}^{d}}h(v) \phi(v) \text{d}v\,\phi(v) \mathcal{M}(v)\,,\ \ \mathcal{M}\;\text{ is the normalized Maxwellian}.
\end{equation*} 
Using the fact that the semigroup and $\pi$ commutes, one has
\begin{equation*}
(1-\pi)\chi h(t) = (1-\pi)\chi h_0 + \int^{t}_{0}e^{\mathcal{L}(t-s)}Q(\chi h,\chi h)(s)\text{d}s + \int^{t}_{0}e^{\mathcal{L}(t-s)}(1-\pi)\mathcal{R}(h)(s)\text{d}s\,,
\end{equation*}  
where we used that $(1-\pi)Q(\cdot,\cdot) = Q(\cdot,\cdot)$.  Thus, applying the $H^{\alpha}_{k}$-norm we conclude that
\begin{align}\label{sg0.0}
\begin{split}
\|(1-\pi)\chi h(t)\|_{H^{\alpha}_{k}(\mathbb{R}^{d})} \leq \|(1-\pi)\chi h_{0}\|_{H^{\alpha}_{k}(\mathbb{R}^{d})} + \frac{1}{\nu}&\Big(\text{O}\big( L^{\lambda k}/N^{\frac{d-1}{2}}\big) + \text{O}\big( 1/L^{\lambda k}\big) \Big) \\
&+ C_{k}\int^{t}_{0}e^{-\nu(t-s)}\|\chi h(s)\|^{3/2}_{H^{\alpha}_{k}(\mathbb{R}^{d})}\text{d}s\,. 
\end{split}
\end{align}
Now, the conservation routine grants that $\pi h(t) = 0$ for any $t\geq0$.  Then,
\begin{equation*}
\|\pi \chi h\|_{H^{\alpha}_{k}(\mathbb{R}^{d})} = \|\pi (1- \chi) h\|_{H^{\alpha}_{k}(\mathbb{R}^{d})} \leq C_{k,\alpha} \| (1-\chi)h\|_{L^{1}_{k}(\mathbb{R}^{d})} = O\big(1/L^{\lambda k}\big)\,.
\end{equation*}
As a consequence,
\begin{equation}\label{sg0.00}
\|(1-\pi)\chi h(t)\|_{H^{\alpha}_{k}(\mathbb{R}^{d})} = \|\chi h(t)\|_{H^{\alpha}_{k}(\mathbb{R}^{d})} + O\big(1/L^{\lambda k}\big) = \| h(t) \|_{H^{\alpha}_{k}(\Omega_{L})} + O\big(1/L^{\lambda k}\big)\,.
\end{equation}
Thus, estimates \eqref{sg0.0}, \eqref{sg0.00}, and \eqref{extension_estimate} leads to the control
\begin{align}\label{sg0}
\begin{split}
\|h(t)\|_{H^{\alpha}_{k}(\Omega_{L})} \leq \|h_{0}\|_{H^{\alpha}_{k}(\Omega_{L})} + \frac{1}{\nu}&\Big(\text{O}\big( L^{\lambda k}/N^{\frac{d-1}{2}}\big) + \text{O}\big( 1/L^{\lambda k}\big) \Big) \\
&+ C_{k}\int^{t}_{0}e^{-\nu(t-s)}\|h(s)\|^{3/2}_{H^{\alpha}_{k}(\Omega_{L})}\text{d}s=:Y(t)\,. 
\end{split}
\end{align}
Observing that
\begin{align*}
Y'(t) = C_{k}\,\|h(t)\|^{3/2}_{H^{\alpha}_{k}(\Omega_{L})} - \nu\,C_{k}\int^{t}_{0}e^{-\nu(t-s)}\|h(s)\|^{3/2}_{H^{\alpha}_{k}(\Omega_{L})}\text{d}s\,,
\end{align*}
one concludes, using \eqref{sg0}, that
\begin{equation}\label{stability-ode} 
Y'(t) + \nu Y(t) \leq C_{k}Y^{3/2}(t) + \nu\|h_{0}\|_{H^{\alpha}_{k}(\Omega_{L})} + \text{O}\big( L^{\lambda k}/N^{\frac{d-1}{2}}\big) + \text{O}\big( 1/L^{\lambda k}\big)\,.
\end{equation}
This estimate tell us that if
\begin{equation}\label{sg1}
C_{k}\sqrt{Y_0} = C_{k}\sqrt{\|h_{0}\|_{H^{\alpha}_{k}(\Omega_{L})} + \frac{1}{\nu}\Big(\text{O}\big( L^{\lambda k}/N^{\frac{d-1}{2}}\big)  + \text{O}\big( 1/L^{\lambda k}\big) \Big) }\ll \nu\,,
\end{equation}
then
\begin{equation}\label{sg2}
\|h(t)\|_{H^{\alpha}_{k}(\Omega_{L})} \leq Y(t)\lesssim \|h_{0}\|_{H^{\alpha}_{k}(\Omega_{L})} + \frac{1}{\nu}\Big(\text{O}\big( L^{\lambda k}/N^{\frac{d-1}{2}}\big) + \text{O}\big( 1/L^{\lambda k}\big) \Big)\,,\quad t>0\,.
\end{equation}
This proves the following local stability estimate for the conservative semi-discrete solution.
\begin{prop}[Local stability for the semi-discrete scheme]\label{lstab}
Fix $\alpha_0\geq0$ and let $g_0\in H^{\alpha_0}_{2k}(\Omega_{L})$, with $k\geq k_{*} > 1 + \frac{d}{2\lambda}$, be a initial datum for the semi-discrete problem.  Assume that $\|g_0 - \mathcal{M}_0\|_{H^{\alpha}_{k}(\Omega_{L})}\leq \delta/2$, for $0<\delta\ll \min\{\nu, \epsilon(g_0)\}$. Then, there exist a lateral size $L_{0}(g_0,\nu)>0$, and a number of modes $N_0(g_0, L_{0},\nu)$ such that for any $\alpha\leq\alpha_0$
\begin{equation*}
\sup_{ t \geq0 } \|g - \mathcal{M}_0 \|_{H^{\alpha}_{k}(\Omega_L)} \leq \delta\,, \quad L\geq L_0\,,\; N\geq N_0\,,
\end{equation*}
where $\mathcal{M}_0$ is the Maxwellian having the same mass, momentum and energy of the initial configuration $g_0$.
\end{prop}
\begin{proof}
The result follows from the aforementioned discussion noticing that \eqref{sg2} is valid provided $L$ is taken first sufficiently large and then $N:=N(L)$, in a way that $\eqref{sg1}$ is satisfied.  Since the constant $C_{k}$ depends on propagations of moments and the norm $H^{\alpha_{0}}_{k}$, the validity of \eqref{sg2} holds provided the negative mass of $g$ is small.  However, this is clear since
\begin{equation*}
\|g^{-}\|_{L^{2}_{k}(\Omega_{L})}\leq \|(g-\mathcal{M}_{0})\text{1}_{\{g\leq0\}}\|_{L^{2}_{k}(\Omega_{L})}\leq \delta \ll \epsilon(g_0)\,.
\end{equation*}
\end{proof}
As a corollary of the error estimates and the local stability of the scheme, exponential relaxation to the Maxwellian equilibrium follows in Lebesgue and Sobolev norms.  Indeed, using the classical asymptotic Boltzmann theory \cite{Des-Vil-2005, Mcon}  for variable hard potentials
\begin{equation*}
\| f - \mathcal{M}_0 \|_{H^{\alpha}_{k}(\mathbb{R}^{d})} \leq C_{k}\,\|f_{0}\|_{H^{\alpha}_{k}(\mathbb{R}^{d})}\,e^{- \nu t}\,,
\end{equation*}
where  $\nu>0$ is the spectral gap of the linearized Boltzmann equation.  Thus, for any $\delta>0$ we can choose
\begin{equation}\label{me0}
T(\delta):=\ln\bigg( \frac{ 4\,C_{k}\,\|f_{0}\|_{H^{\alpha}_{k}(\mathbb{R}^{d})}}{\delta}\bigg)^{1/\nu}\,,\quad\text{so that}\quad
\sup_{t\geq T(\delta)/2}\| f - \mathcal{M}_0  \|_{H^{\alpha}_{k}(\mathbb{R}^{d})} \leq \delta/4\,.
\end{equation}
\begin{thm}[Convergence to the Maxwellian equilibrium]\label{cf} 
Fix $\alpha_0\geq0$ and let $f_0\in H^{\alpha_0}_{2k}(\mathbb{R}^{d})$, with $k\geq k_{*} > 1 + \frac{d}{2\lambda}$,  be a initial datum.  Then, for every $0<\delta\ll \min\{\nu,\epsilon(g_0)\}$ there exist a lateral size $L_{0}(f_0)>0$, and a number of modes $N_0(L,f_0)$ such that for any $\alpha\leq\alpha_0$
\begin{equation*}
\sup_{ t \geq T(\delta)/2 } \|g -\mathcal{M}_0  \|_{H^{\alpha}_{k}(\Omega_L)} \leq \delta\,, \quad L\geq L_0\,,\; N\geq N_0\,,
\end{equation*}
where $\mathcal{M}_0$ is the Maxwellian having the same mass, momentum and energy of the initial configuration $f_0$.
\end{thm}
\begin{proof}
Letting $T=T(\delta)$ in Theorem \ref{l2convergence} for the case $\alpha_0=0$ or Theorem \ref{l2convergenceregularity} for the case $\alpha_0>0$, one concludes that there exist a lateral size $L_0(T(\delta),f_0)$ and number of modes $N_0(T(\delta),L,f_0)$ such that
\begin{equation}\label{me}
\sup_{t\in [0,T]}\|f - g\|_{H^{\alpha}_{k}(\Omega_L)} \leq \delta/4\,,\quad L\geq L_{0}\,,\;N\geq N_0\,.
\end{equation}
Using triangle inequality with \eqref{me0} and \eqref{me} one has that
\begin{equation*}
\sup_{ t \in [T(\delta)/2,T(\delta)] } \|g -\mathcal{M}_0 \|_{H^{\alpha}_{k}(\Omega_L)} \leq \delta/2\,, \quad L\geq L_0\,,\; N\geq N_0\,.
\end{equation*}
The result follows after invoking the local stability result of Proposition \ref{lstab}.
\end{proof}
\noindent
\begin{rem}
Since the relaxation of the Boltzmann solution $f(t,v)$ is exponentially fast for variable hard potentials, the simulation times are relatively short as noticed in previous proof.  This makes conservative schemes very stable even when using relatively small working domains and number of modes. 
\end{rem}
\noindent\textbf{Completion of  proof of Theorem \ref{CT}.} We just need to discuss the time uniform nature of the constants appearing in the error estimates.  We first observe that  the conservative spectral scheme follows the nonlinear dynamics of the Boltzmann equation in the time range $[0, T(\delta)/2]$.  Next, for $t\geq T(\delta)/2$, the dynamics is relaxed around the thermal equilibrium, so that, it is controlled by the linear evolution.  Hence,
\begin{equation*}
\|f-g\|_{H^{\alpha}_{k}(\Omega_L)} = \|f- \mathcal{M}_0\|_{H^{\alpha}_{k}(\Omega_L)} + \| g -\mathcal{M}_{0} \|_{H^{\alpha}_{k}(\Omega_L)} \leq 2\delta\,,\quad \text{for} \quad t\geq T(\delta)/2\,.
\end{equation*}  
As a consequence, in the long run, $f-g$ is estimated by the minimum between estimate \eqref{T5.3-estimate} evaluated at $T(\delta)/2$ and $2\delta$.  As a consequence, we conclude 
\begin{align*}
\sup_{t\geq0}&\|f-g\|_{H^{\alpha}_{k}(\Omega_L)}\\
&\leq e^{ \alpha C_{k}T(\delta)}\Big(\|f_{0}-g_{0}\|_{H^{\alpha}_{k+\alpha/2}(\Omega_L)} + O\big(L^{\lambda(k  +\alpha/2) + \alpha_0}/N^{(d-1)/2 + \alpha_0}\big) + O_{d/2+\lambda k}\Big)\,,
\end{align*}
for $L\geq L_0(T(\delta),f_{0})$, $N\geq N_{0}(T(\delta),L_{0},f_{0})$   and  the term $O_{d/2+\lambda k}\big)$ as defined in \eqref{O_r}.  

Recalling \eqref{me0}, note that 
\begin{equation*}
e^{ \alpha C_{k}T(\delta)} \sim \bigg(\frac{4C_{k}\|f_0\|_{H^{\alpha}_{k}(\Omega_L)} }{\delta}\bigg)^{\alpha C_{k}/\nu}\,.
\end{equation*}  
The proof of Theorem \ref{CT} is concluded after minimizing in $\delta>0$, which gives $\theta=\alpha C_{k}/\nu$ in items 2. and 3. 
\qed}
%
%
\section{Appendix}
\subsection{Shannon Sampling Theorem}
The following result is an extension of the standard approximation estimate for regular functions by Fourier series expansions, \textit{Shannon Sampling Theorem}, to $H^{\alpha}(\Omega_L)$ space. We include here the result for completeness of the reading.
\begin{lem}[Fourier Approximation Estimate]\label{lemmaProjectionEstimate}
Let $g \in H^{\alpha}(\Omega_L)$, then
\begin{equation}\label{projectionEstimate}
\| (\textbf{1} - \Pi^N_{L}) g \|_{L^2(\Omega_L)} \leq \frac{1}{(\sqrt{2\pi})^{d}}\left(\frac{L}{2\pi N}\right)^{\alpha} \| g \|_{H^{\alpha}(\Omega_L)} \,.
\end{equation}
\end{lem}
\begin{proof}
Parseval's relation gives
\begin{equation*}
\| (\textbf{1} - \Pi^N_{L}) g \|_{L^2(\Omega_L)} = \sqrt{ \sum_{k> N} | \widehat{g}(\zeta_k) |^2 }\,.
\end{equation*}
Furthermore, properties of the Fourier transform implies
\begin{equation*}
| \widehat{g}(\zeta_k) | = \frac{1}{(\sqrt{2\pi})^{d}} \frac{\big|\widehat{D^{\alpha}g}(\zeta_k)\big|}{\prod_{j=1}^{d}|(\zeta^j_k)^{\alpha_j}|}\,.  
\end{equation*}
Therefore,
\begin{align*}
\sum_{k > N} | \widehat{g}_N(\zeta_k) |^2 &= \frac{1}{(2\pi)^d} \sum_{k> N} \frac{\big|\widehat{D^{\alpha}g}(\zeta_k) \big|^2}{\prod_{j=1}^{d}|(\zeta^j_k)^{\alpha_j}|^2}  \leq  \frac{1}{(2\pi)^d} \frac{\sum_{k> N} \big| \widehat{D^{\alpha}g}(\zeta_k) \big|^2}{\prod_{j=1}^{d}|(\zeta^j_{N})^{\alpha_j}|^2} \, .
\end{align*}
Observe that the sum in last inequality equals the $L^2$-norm square of $D^{\alpha}g-\Pi^{N}D^{\alpha}g$, therefore, 
\begin{align*}
\sum_{k> N} | \widehat{g}_N(\zeta_k) |^2 & \leq  \frac{1}{(2\pi)^d} \frac{\big\|D^{\alpha}g-\Pi^{N}D^{\alpha}g\big\|^2_{L^2(\Omega_L)}}{\prod_{j=1}^{d}|(\zeta^j_{N})^{\alpha_j}|^2}  \leq \frac{1}{(2\pi)^d} \frac{\| D^{\alpha}g \|^2_{L^2(\Omega_L)}}{\prod_{j=1}^{d}|(\zeta^j_{N})^{\alpha_j}|^2} \,.
\end{align*}
Conclude recalling the definition of $\zeta_{N} = \frac{2\pi N}{L}$.
\end{proof}
\subsection{Estimate on the decay of the collision operator}
\begin{thm}\label{moments}
The following estimate holds for any $k\geq0$ and $\lambda\in[0,2]$,
\begin{equation*}
\bigg|\int_{\mathbb{R}^{d}\setminus\Omega_L}Q(f,f)(v) \, dv\,\bigg |\leq O_{k}\big(m_{k+1}(f)m_{0}(f)+Z_{k}(f)\big).
\end{equation*}
The term $Z_{k}(f)$ is defined below in \eqref{Zk} and only depends on moments up to order $k$.  In particular one has
\begin{equation}\label{Zmk}
Z_{k}(f)\leq 2^{k}\, m_{1}(f)\,m_{k}(f).
\end{equation}
\end{thm}
\begin{proof}
For the negative part,
\begin{align*}
\bigg|\int_{\mathbb{R}^{d}\setminus\Omega_L}Q^{-}(f,f)(v)\text{d}v\,\bigg|&\leq L^{-\lambda k}\int_{\{|v|\geq L\}}Q^{-}(|f|,|f|)(v)|v|^{\lambda k}\text{d}v \leq L^{-\lambda k}\big(m_{k+1}m_0+m_{k}m_0\big)\,.
\end{align*}
For the positive part,
\begin{align*}
\bigg|\int_{\mathbb{R}^{d}\setminus\Omega_L}Q^{+}(f,f)(v)\text{d}v\,\bigg|&\leq L^{-\lambda k}\int_{\{|v|\geq L\}}Q^{+}(|f|,|f|)(v)|v|^{\lambda k}\text{d}v\\
&=L^{-\lambda k}\int_{\mathbb{R}^{2d}}|f(v)||f(v_{*})||u|^{\lambda}\int_{\mathbb{S}^{d-1}}|v'|^{\lambda k} b(\hat{u}\cdot\sigma)\text{d}\sigma \text{d}v_{*}\text{d}v
\end{align*}
Note,
\begin{align*}
\int_{\mathbb{S}^{d-1}}|v'|^{\lambda k} b(\hat{u}\cdot\sigma)\text{d}\sigma&\leq \|b\|_{L^{1}(S^{d-1})}\left(|v|^{2}+|v_{*}|^{2}\right)^{\lambda k/2} \leq \|b\|_{L^{1}(S^{d-1})}\sum^{ k}_{j=0}
\left(\begin{array}{c}
 k\\j
\end{array}\right)
|v|^{\lambda j}|v_{*}|^{\lambda(k-j)}\,.
\end{align*}
Use the inequality $|u|^{\lambda}\leq|v|^{\lambda}+|v_{*}|^{\lambda}$ with the previous expressions to obtain,
\begin{equation*}
\bigg|\int_{\mathbb{R}^{d}\setminus\Omega_L}Q^{+}(f,f)(v)\text{d}v\,\bigg|\leq 2\|b\|_{L^{1}(S^{d-1})}L^{-\lambda k}\big(m_{k+1}(f)m_0(f)+Z_{k}(f)\big)\,,
\end{equation*}
where
\begin{equation}\label{Zk}
Z_{k}(f):=\sum^{k-1}_{j=0}
\left(\begin{array}{c}
 k\\j
\end{array}\right)
m_{j+1}(f)m_{k-j}(f).
\end{equation}
Furthermore, note that interpolation implies for $0 \leq j \leq k-1$
\begin{align*}
m_{j+1}(f) &\leq m_{1}(f)^{\frac{k-1-j}{k-1}}\,m_{k}(f)^{\frac{j}{k-1}}\,,\quad m_{k-j}(f) \leq m_{1}(f)^{\frac{j}{k-1}}\,m_{k}(f)^{\frac{k-1-j}{k-1}}\,.
\end{align*}
Therefore,
\begin{equation*}
m_{j+1}(f)\,m_{k-j}(f) \leq m_{1}(f)\,m_{k}(f)\,, \quad 0 \leq j \leq k-1\,.
\end{equation*}
This implies that
\begin{equation*}
Z_{k}(f) \leq m_{1}(f)\,m_{k}(f) \sum^{k-1}_{j=0}
\left(\begin{array}{c}
 k\\j
\end{array}\right)
\leq 2^{k}\,m_{1}(f)\,m_{k}(f)\,.
\end{equation*}
\end{proof}

\subsection{$L^2$-theory of the collision operator}
\n The following theorems follow from the arguments in \cite{GPV04, AC, ACG}
\begin{thm}[Collision Integral Estimate for Elastic/ Inelastic Collisions]
For $f, g \in L^{1}_{k + 1}(\mathbb{R}^{d})\cap L^{2}_{k + 1}(\mathbb{R}^{d})$ one has the estimate
\begin{equation}
\| Q(f, g) \|_{L^2_k(\mathbb{R}^{d})}  \leq  C\;\big(\| f \|_{L^2_{k + 1}(\mathbb{R}^{d})} \| g \|_{L^1_{k + 1}(\mathbb{R}^{d})} + \| f \|_{L^1_{k + 1}(\mathbb{R}^{d})} \| g \|_{L^2_{k + 1}(\mathbb{R}^{d})}\big)
\label{collisionEstimate}
\end{equation}
where the dependence of the constant is $C:=C(d, \|b\|_{1})$.
\label{theoremCollisionEstimate}
\end{thm}
\noindent Theorem \ref{theoremCollisionEstimate} and Leibniz formula
\begin{equation}
\partial^\alpha Q(f, g) = \sum_{|j| \leq|\alpha|} \binom{\alpha}{j} Q(\partial^{\alpha-j}f, \partial^j g) \, ,\quad \text{for multi-indexes}\;j\,,\,\alpha\,,
\label{QLeibniz}
\end{equation}
proves the following theorem, see \cite[Section 4]{GPV04} for additional discussion.
\begin{thm}[Sobolev Bound Estimate]
Let $\mu > 1+\frac{d}{2\lambda}$.  For $f, g \in H^{\alpha}_{k+\mu}(\mathbb{R}^{d})$, the collision operator satisfies
\begin{equation}
\| Q(f, g) \|^2_{H^{\alpha}_{k}(\mathbb{R}^{d})}  \leq C\sum_{j \leq \alpha} \binom{\alpha}{j} \Big( \| f \|^2_{H^{\alpha - j}_{k + 1}(\mathbb{R}^{d})} \| g \|^2_{H^j_{k + \mu}(\mathbb{R}^{d})} +  \| f \|^2_{H^{\alpha-j}_{k + \mu}(\mathbb{R}^{d})} \| g \|^2_{H^{j}_{k + 1}(\mathbb{R}^{d})} \Big) \, ,
\label{regularityEstimate}
\end{equation}
where the dependence of the constant is $C:=C(d,\alpha,\|b\|_{1})$.
\label{theoremRegularityEstimate}
\end{thm}
\begin{cor}
Let $\mu>\frac{d}{2}+\lambda$.  For $f \in H^{\alpha}_{k+\mu}(\mathbb{R}^{d})$ the collision operator satisfies the estimate
\begin{equation}
\| Q(f, f) \|_{H^{\alpha}_{k}(\mathbb{R}^{d})} \leq C \| f \|^2_{H^{\alpha}_{k + \mu}(\mathbb{R}^{d})} \, .
\label{corSchwartzEstimate}
\end{equation}
The dependence of the constant is given by $C:=C(d,\mu,\|b\|_{1})$.
\label{corollarySchwartz}
\end{cor}
\n In this last section of the appendix we discuss briefly the gain of integrability in the gain collision operator, see \cite{AG} for a more detailed discussion.
\begin{thm}\label{Tlp}
The collision operator satisfies the estimate for any $\epsilon>0$ and $k\geq0$
\begin{equation*}
\|Q^{+}_{\lambda}(g,f)\|_{L^{2}_{k}(\mathbb{R}^{d})}\leq C\|b\|_{\infty}\|g\|_{L^{1}_{k}(\Omega_{L})} \big(\tfrac{\epsilon^{r'}}{r'}\,\|f\|_{L^{2}_{k}(\mathbb{R}^{d})}+\tfrac{1}{r\epsilon^r}\|f\|^{1-\theta}_{L^{1}_{k}(\Omega_{L})}\|f\|^{\theta}_{L^{2}_{k}(\mathbb{R}^{d})}\big)\,,
\end{equation*}
where $\theta=\frac{1}{d}$, $r= \frac{d-2}{\lambda}$ and $C_n$ constant depending only on the dimension.
\end{thm}

\section{Conclusion}
We have studied the global existence and error estimates for the homogeneous Boltzmann spectral method imposing conservation of mass, momentum and energy by Lagrange constrained optimization.  The methods and estimates presented in the document show that imposing conservation of these quantities stabilizes the long time behavior of the discrete problem because enforces the collisional invariants.  In some sense, this in turn enforces the numerical approximation of the linearized collisional operator to have the same null space as the true linearized collision operator, which is the one in charge of the long time dynamics.  In particular, the work domain and the number of modes can be chosen such that the discrete solution approximates with any desired accuracy the stationary state of the original Boltzmann problem in the long run.  Although, spurious tail behavior is experienced when the optimization is imposed due to the addition of a quadratic polynomial corrector, the natural property of creation of moments remains in the semi-discrete problem.  This allows to minimize such spurious behavior by appropriate choice of simulation parameters.  We point out here that other correctors, such as Gaussian, might be more suitable in this respect.  Furthermore, conservation of mass and energy limits the negative mass produced by the numerical scheme which is essential for long time accurate simulations. 

\section{Acknowledgements}
R. Alonso acknowledges the support from ONR grants N000140910290 and NSF-RNMS 1107465.  I. M. Gamba acknowledges support from NSF grant DMS-1413064.  H. S. Tharkabhushanam acknowledges support from NSF grant DMS-0807712.  The authors thank and gratefully acknowledged the hospitality and support from the Institute of Computational Engineering and Sciences and the University of Texas Austin. R. J. Alonso is grateful to the Bolsa de Produtividade em Pesquisa CNPq and the Department of Mathematics at PUC-Rio for their support.

}
\bigskip
\bibliographystyle{amsplain}
\bibliography{bibliography1}

\providecommand{\bysame}{\leavevmode\hbox to3em{\hrulefill}\thinspace}
\providecommand{\MR}{\relax\ifhmode\unskip\space\fi MR }
\providecommand{\MRhref}[2]{%
  \href{http://www.ams.org/mathscinet-getitem?mr=#1}{#2}
}
\providecommand{\href}[2]{#2}
\begin{thebibliography}{10}

\bibitem{ACGM}
R.~Alonso, J.~Canizo, I.~Gamba, and C.~Mouhot, \emph{A new approach to the
  creation and propagation of exponential moments in the boltzmann equation},
  Comm. Part. Diff. Equat. \textbf{38} (2013), no.~1, 155--169.

\bibitem{AC}
R.~Alonso and E.~Carneiro, \emph{Estimates for the {B}oltzmann collision
  operator via radial symmetry and fourier transform.}, Adv. Math. \textbf{223}
  (2010), no.~2, 511--528.

\bibitem{ACG}
R.~Alonso, E.~Carneiro, and I.~M. Gamba, \emph{Convolution inequalities for the
  {B}oltzmann collision operator}, Comm. Math. Physics. \textbf{298} (2010),
  no.~2, 293--322.

\bibitem{AG08}
R.~Alonso and I.~M. Gamba, \emph{${L}^1-{L}^{\infty}$ maxwellian bounds for the
  derivatives of the solution of the homogeneous boltzmann equation}, Journal
  de Math\'ematiques Pures et Appliqu\'ees \textbf{89} (2008), no.~6, 575--595.

\bibitem{AG}
\bysame, \emph{Gain of integrability for the {B}oltzmann collisional operator},
  Kinetic and Related Models \textbf{4} (2011), no.~4, 41--51.

\bibitem{ALod}
R.~Alonso and B.~Lods, \emph{Free cooling and high-energy tails of granular
  gases with variable restitution coefficient}, Commun. Math. Sci. \textbf{11}
  (2013), no.~3, 807--862.

\bibitem{AlonsoLods2014}
R.~Alonso and B.~Lods, \emph{Boltzmann model for viscoelastic particles:
  Asymptotic behavior, pointwise lower bounds and regularity.}, Commun. Math.
  Phys. \textbf{331} (2014), no.~2, 545--591.

\bibitem{bird}
G.~A. Bird, \emph{Molecular gas dynamics}, Clarendon Press, Oxford, 1994.

\bibitem{bobylevFT}
A.~V. Bobylev, \emph{Exact solutions of the nonlinear {B}oltzmann equation and
  the theory of relaxation of a {M}axwellian gas}, Translated from
  Teoreticheskaya i Mathematicheskaya Fizika \textbf{60} (1984), 280 -- 310.

\bibitem{Bobylev97}
A.~V. Bobylev, \emph{Moment inequalities for the boltzmann equation and
  applications to spatially homogeneous problems.}, J. Statist. Phys.
  \textbf{88} (1997), no.~5-6, 1183--1214.

\bibitem{BCG00}
A.~V. Bobylev, J.~A. Carrillo, and I.~M. Gamba, \emph{On some properties of
  kinetic and hydrodynamic equations for inelastic interactions}, J. Statist.
  Phys. \textbf{98} (2000), no.~3-4, 743--773. \MR{MR1749231 (2001c:82063)}

\bibitem{bobyCerci99}
A.~V. Bobylev and C.~Cercignani, \emph{Discrete velocity models without
  nonphysical invariants}, Journal of Statistical Physics \textbf{97} (1999),
  677--686.

\bibitem{BGP04}
A.~V. Bobylev, I.~M. Gamba, and V.~A. Panferov, \emph{Moment inequalities and
  high-energy tails for {B}oltzmann equations with inelastic interactions}, J.
  Statist. Phys. \textbf{116} (2004), no.~5-6, 1651--1682. \MR{MR2096050
  (2005g:82111)}

\bibitem{bobylevRjasanow0}
A.~V. Bobylev and S.~Rjasanow, \emph{Difference scheme for the {B}oltzmann
  equation based on the {F}ast {F}ourier {T}ransform}, European journal of
  mechanics. B, Fluids \textbf{16:22} (1997), 293--306.

\bibitem{BR99}
\bysame, \emph{Fast deterministic method of solving the {B}oltzmann equation
  for hard spheres}, Eur. J. Mech. B Fluids \textbf{18} (1999), no.~5,
  869--887.

\bibitem{BR00}
\bysame, \emph{Numerical solution of the {B}oltzmann equation using fully
  conservative difference scheme based on the {F}ast {F}ourier {T}ransform},
  Transport Theory Statist. Phys. \textbf{29} (2000), 289--310.

\bibitem{BD}
F.~Bouchut and L.~Desvillettes, \emph{A proof of the smoothing properties of
  the positive part of boltzmann's kernel}, Rev. Mat. Iberoamericana
  \textbf{14} (1998), no.~1, 47--61.

\bibitem{broadwell}
J.~E. Broadwell, \emph{Study of rarefied shear flow by the discrete velocity
  method}, J. Fluid Mech. \textbf{19} (1964), 401--414.

\bibitem{CIP}
C.~Cercignani, I.~Reinhard, and M.~Pulvirenti, \emph{The mathematical theory of
  dilute gases}, Applied Mathematical Sciences, vol. 106, Springer-Verlag, New
  York, 1994.

\bibitem{DM04}
L.~Desvillettes and C.~Mouhot, \emph{About ${L}^p$ estimates for the spatially
  homogeneous {B}oltzmann equation}, Ann. I. H. Poincar\'{e} \textbf{22}
  (2005), 127--142.

\bibitem{Des-Vil-2005}
L.~Desvillettes and C.~Villani, \emph{On the trend to global equilibrium for
  spatially inhomogeneous kinetic systems: the {B}oltzmann equation}, Invent.
  Math. \textbf{159} (2005), no.~2, 245--316. \MR{2116276}

\bibitem{FM}
F.~Filbet and C.~Mouhot, \emph{Analysis of spectral methods for the homogeneous
  {B}oltzmann equation}, Transactions of the american mathematical society
  \textbf{363} (2010), 1947 -- 1980.

\bibitem{filbetRusso1}
F.~Filbet, C.~Mouhot, and L.~Pareschi, \emph{Solving the {B}oltzmann equation
  in nlogn}, SIAM J. Sci. Comput. \textbf{28} (2006), 1029--1053.

\bibitem{filbetRusso}
F.~Filbet and G.~Russo, \emph{High order numerical methods for the space non
  homogeneous {B}oltzmann equation}, Journal of Computational Physics
  \textbf{186} (2003), 457--480.

\bibitem{Gab-Par-Tos}
E.~Gabetta, L.~Pareschi, and G.~Toscani, \emph{Relaxation schemes for nonlinear
  kinetic equations}, SIAM J. Numer. Anal. \textbf{34} (1997), 2168--2194.

\bibitem{GaHa}
I.~M. Gamba and J.~R. Haack, \emph{A conservative spectral method for the
  {B}oltzmann equation with anisotropic scattering and the grazing collisions
  limit}, Jour. Comp. Phys. \textbf{270} (2014), 40--57.

\bibitem{GPV04}
I.~M. Gamba, V.~Panferov, and C.~Villani, \emph{On the {B}oltzmann equation for
  diffusively excited granular media}, Comm. Math. Phys. \textbf{246} (2004),
  no.~3, 503--541.

\bibitem{GPV08}
\bysame, \emph{Upper {M}axwellian bounds for the spatially homogeneous
  {B}oltzmann equation}, Arch. Rat. Mech. Anal \textbf{194} (2009), 253--282.

\bibitem{GRW04}
I.~M. Gamba, S.~Rjasanow, and W.~Wagner, \emph{Direct simulation of the
  uniformly heated granular {B}oltzmann equation}, Mathematical and Computer
  Modelling \textbf{42} (2005), 683--700.

\bibitem{GT09}
I.~M. Gamba and S.~H. Tharkabhushanam, \emph{Spectral - lagrangian methods for
  collisional models of non - equilibrium statistical states}, Journal of
  Computational Physics \textbf{228} (2009), no.~6, 2012--2036.

\bibitem{GT08}
\bysame, \emph{Shock and boundary structure formation by spectral-lagrangian
  methods for the inhomogeneous boltzmann transport equation}, Jour. Comp. Math
  \textbf{28} (2010), 430--460.

\bibitem{grad1969}
H.~Grad, \emph{Singular and nonuniform limits of solutions of the {B}oltzmann
  equation.}, Transport Theory (Proc. Sympos. Appl. Math., New York, 1967),
  SIAM-AMS Proc., Vol. I, Amer. Math. Soc., Providence, R.I., 1969,
  pp.~269--308. \MR{MR0255205 (40 \#8410)}

\bibitem{HerParSea}
M.~Herty, L.~Pareschi, and M.~Seaid, \emph{Discrete-velocity models and
  relaxation schemes for traffic flows}, SIAM J. Sci. Comput. \textbf{28}
  (2006), 1582--1596.

\bibitem{ibragRjasanow}
I.~Ibragimov and S.~Rjasanow, \emph{Numerical solution of the {B}oltzmann
  equation on the uniform grid}, Computing \textbf{69} (2002), 163--186.

\bibitem{kawa81}
S.~Kawashima, \emph{Global solution of the initial value problem for a discrete
  velocity model of the {B}oltzmann equation}, Proc. Japan Acad. Ser. A Math.
  Sci. \textbf{57} (1981), 19--24.

\bibitem{desvillettes93}
Desvillettes. L., \emph{Some applications of the method of moments for the
  homogeneous boltzmann and kac equations}, Arch. Rational Mech. Anal.
  \textbf{123} (1993), no.~4, 387--404.

\bibitem{Landau37}
L.~D. Landau, \emph{Kinetic equation for the case of {C}oulomb interaction},
  Zh. Eks. Teor. Phys. \textbf{7} (1937), 203.

\bibitem{LanLifStat}
L.~D. Landau and E.~M. Lifschitz, \emph{Statistical physics},
  Butterworth-Heinemann (1980), 3rd ed.

\bibitem{mieuss00}
L.~Mieussens, \emph{Discrete-velocity models and numerical schemes for the
  {B}oltzmann-bgk equation in plane and axisymmetric geometries}, Journal of
  Computational Physics \textbf{162} (2000), 429--466.

\bibitem{MMR06}
S.~Mischler, C.~Mouhot, and M.~Rodriguez~Ricard, \emph{Cooling process for
  inelastic {B}oltzmann equations for hard spheres. {I}. {T}he {C}auchy
  problem}, J. Stat. Phys. \textbf{124} (2006), no.~2-4, 655--702.
  \MR{MR2264622 (2007g:82053)}

\bibitem{Mcon}
C.~Mouhot, \emph{Rate of convergence to equilibrium for the spatially
  homogeneous boltzmann equation with hard potentials}, Comm. Math. Phys.
  \textbf{261} (2006), 629 -- 672.

\bibitem{mouhotPareschi}
C.~Mouhot and L.~Pareschi, \emph{Fast algorithms for computing the {B}oltzmann
  collision operator}, Math. Comp. \textbf{75} (2006), 1833--1852.

\bibitem{MV04}
C.~Mouhot and C.~Villani, \emph{Regularity theory for the spatially homogeneous
  {B}oltzmann equation with cut-off}, Arch. Ration. Mech. Anal. \textbf{173}
  (2004), no.~2, 169--212. \MR{MR2081030 (2006f:82073)}

\bibitem{MHGM14}
A.~Munafo, J.~R. Haack, I.~M. Gamba, and T.~E. Magin, \emph{A
  spectral-lagrangian {B}oltzmann solver for a multi-energy level gas}, Jour.
  Comp. Phys. \textbf{264} (2014), 152--176.

\bibitem{pareschiPerthame}
L.~Pareschi and B.~Perthame, \emph{A {F}ourier spectral method for homogenous
  {B}oltzmann equations}, Transport Theory Statist. Phys. \textbf{25} (2002),
  369--382.

\bibitem{pareschiRusso}
L.~Pareschi and G.~Russo, \emph{Numerical solution of the {B}oltzmann equation.
  i. spectrally accurate approximation of the collision operator}, SIAM J.
  Numerical Anal. (Online) \textbf{37} (2000), 1217--1245.

\bibitem{RjaWa05}
S.~Rjasanow and W.~Wagner, \emph{Stochastic numerics for the {B}oltzmann
  equation}, Springer, Berlin, 2005.

\bibitem{SE}
E.~Stein, \emph{Singular integrals and differentiability properties of
  functions}, Princeton University Press, Princeton, N.J., 1970.

\bibitem{Wagner92}
W.~Wagner, \emph{A convergence proof for {B}ird's direct simulation monte carlo
  method for the {B}oltzmann equation}, Journal of Statistical Physics (1992),
  1011--1044.

\bibitem{wennberg97}
B.~Wennberg, \emph{Entropy dissipation and moment production for the
  {B}oltzmann equation}, Jour. Statist. Phys. \textbf{86} (1997), no.~5-6,
  1053--1066.

\bibitem{DGBTE-Zhang-Gamba14}
C.~Zhang and I.~M. Gamba, \emph{A conservative discontinuous galerkin solver
  for space homogeneous boltzmann equation}, Submitted for publication (2016).

\bibitem{zhang-gamba-Landau16}
\bysame, \emph{A conservative scheme for vlasov poisson landau modeling
  collisional plasmas}, Submitted for publication (2016).

\end{thebibliography}
\end{document}